\documentclass[letterpaper,11pt]{article}
\usepackage{graphicx}
\usepackage{amsmath}
\usepackage{amsfonts}
\usepackage{hyperref}
\usepackage{amsthm}
 \usepackage{amscd}
\usepackage{mathrsfs}
\usepackage{caption}
\usepackage{subcaption}
\usepackage{float}
\usepackage{amssymb} 

\usepackage[top=1in,left=1in,right=1in,bottom=1in]{geometry}

\newtheorem{theorem}{Theorem}[section]
\newtheorem{proposition}[theorem]{Proposition}
\newtheorem{assumption}[theorem]{Assumption}

\newtheorem{definition}[theorem]{Definition}

\newtheorem{remark}[theorem]{Remark}

\hypersetup{
	colorlinks = true,
	linkcolor  = red,
	citecolor  = green,
	filecolor  = cyan,
	urlcolor   = magenta,}

\newcommand{\ds}{\displaystyle}
\newcommand\beq{\begin{equation}}
\newcommand\ene{\end{equation}}

\numberwithin{equation}{section}

\title{Asymptotic  stability of  Bernstein-Greene-Kruskal  (BGK) waves, Landau Damping,  and scattering theory\thanks{2020 MSC, 35P25; 35Q83; 35Q85; 35Q86; 47A40; 85A05. }}
\author{
Ricardo Weder\thanks{weder@unam.mx.}\\
Departamento de F\'\i sica Matem\'atica\\
Instituto de Investigaciones en
Matem\'aticas Aplicadas y en Sistemas\\
Universidad Nacional Aut\'onoma de M\'exico\\
Apartado Postal 20-126, IIMAS-UNAM\\
 Ciudad de M\'exico, CP 01000, M\'exico}

\date{}

\begin{document}
\baselineskip=14pt

\maketitle
.
\begin{abstract}
We consider the two species (electrons and positive charged ions) Vlasov-Poisson system linearized around  BGK waves. We formulate the problem as an equivalent Vlasov-Amp\`ere system, that we write as a system of Schr\"odinger type  in an appropriate Hilbert space, and with a selfadjoint Vlasov-Amp\`ere operator as a {\it Hamiltonian}.  We develop a complete stationary scattering theory. We  identify the absolutely continuous spectrum  and the singular spectrum  of the Vlasov-Amp\`ere operator, we construct the generalized Fourier maps, we prove that the wave operators exist, are complete, satisfy Birman's invariance principle, and that the stationary formulae hold. Using these results we prove that the BGK waves are asymptotically stable. We obtain
 a precise description of the large time behaviour of the solutions to the Vlasov-Amp\`ere system.  Namely,  for large times the phase-space densities of electrons and ions  are asymptotic to the phase-space densities of solutions of the unperturbed  Vlasov-Amp\`ere system. This implies that they follow the trajectories  of  solutions to Newton's equations for electrons and ions with the potential of the  BGK wave, in the sense that they are transported along these trajectories. Furthermore, we prove that Landau damping holds, that is to say, the electric field tends to zero in pointwise sense for large times.

\end{abstract}

\section{Introduction}\label{intro}
The Vlasov-Poisson system is the standard kinetic model to describe the dynamics of  collisionless plasmas.
The stability of its steady states, i.e. time-independent solutions,   is a main topic in plasma physics. In this paper we study the  asymptotic stability  of the  Vlasov-Poisson system linearized around Bernstein-Greene-Kruskal waves (BGK waves), that are exact, space-inhomogeneous, space periodic steady states. Since their discovey by Bernstein, Greene, and Kruskal in 1957 \cite{bkg} they have been intensively studied. They play an important role in understanding the long-time dynamics of the Vlasov-Poisson system. Numerical evidence shows that they can act as attractors for the solutions near an homogeneous equilibrium. See, for example,  \cite{sch}, and \cite{vcvm}. 
We consider a two-species, electrons and positive charged ions, plasma in one time and one space dimension.  The unknowns of  this Vlasov-Poisson system  are the phase-space densities  of ions and of electrons and the induced  electric  field. The phase-space densities $f_\pm(t, x,v),$  respectively, of ions,  and electrons, with  $ t \in \mathbb R, x,v \in \mathbb R,$   give, respectively,  the density of  ions and of electrons at the position $x$ in space with velocity $v,$ at the time $t.$ The induced electric field  $E(t,x)$  is the electric field induced by the ions and the electrons.  The phase-space densities and the induced electric field satisfy the Vlasov-Poisson system

\begin{eqnarray}\label{0.1}
\partial_t f_\pm(t,x,v)+v \partial_x f_\pm(t,x,v)\pm \frac{e_\pm}{m_\pm}E(t,x)\partial_v f_\pm(t,x,v)=0,\\
\partial_x E(t,x)=   \rho(t,x),\label{0.2}\\
\label{0.3} \rho(t,x):= \int_{\mathbb R}  e_+ f_+(t,x,v)\, dv- e_-\int_{\mathbb R} f_-(t,x,v)\, dv+ \rho_{\rm e}(x),
\end{eqnarray} 
where $e_\pm$ are, respectively, the charge of the ions and of the electrons, $m_\pm$ are,  respectively, the mass of the ions and of the electrons, $E(t,x)$ is the induced electric field,  $\rho(t,x)$ is the charge density that is the sum of the charge density 
$$
 \int_{\mathbb R}  e_+ f_+(t,x,v)\, dv- e_-\int_{\mathbb R} f_-(t,x,v)\, dv
$$
induced by the ions and the electrons, and of $\rho_{\rm e}(x)$ that is the external electric charge that is constant in time and can be positive, negative, or zero.  We consider the general case where the plasma is nonneutral, i.e. we do not suppose that the total charge is zero. For a discussion of the physics of nonneutral plasmas see \cite{davidson}. For the derivation of the Vlasov-Poisson system \eqref{0.1}-\eqref{0.3} see chapter seven of  \cite{krall}.

To simplify the notation we take all the constants equal to one and consider the system,
\begin{eqnarray}\label{0.4}
\partial_t f_\pm(t,x,v)+v \partial_x f_\pm(t,x,v)\pm E(t,x) \partial_v f_\pm(t,x,v)=0,\\
\partial_x E(t,x)= \rho(t,x),\label{0.5}\\
\label{0.6} \rho(t,x):= \int_{\mathbb R}   f_+(t,x,v)\, dv-\int_{\mathbb R} f_-(t,x,v)\, dv+ \rho_e(x).
\end{eqnarray} 

 Equation \eqref{0.5} is the Gauss law.  We introduce the electric potential $\varphi(t,x)$ such that
\beq\label{0.7}
 E(t,x)=- \partial_x \varphi(t,x),
 \ene
  and from \eqref{0.5} we obtain the Poisson equation,
\beq\label{0.8}
- \partial^2_x \varphi(t,x)=  \rho(t,x).
\ene
 The  steady states, $(f_{0,\pm}(x,v), E_0(x),\varphi_0(x))$  are time-independent solutions to the  system \eqref{0.4}-\eqref{0.8}. Namely,
\begin{eqnarray}\label{0.9}
v \partial_x f_{0,\pm}(x,v)\pm E_0(x)\partial_v f_{0,\pm}(x,v)=0,\\
\partial_x E(x)=   \rho_0(x),\label{0.10}\\
\label{0.11} \rho_0(x):= \int_{\mathbb R}   f_{0,+}(x,v)\, dv-\int_{\mathbb R} f_{0,-}(x,v)\, dv+ \rho_{\rm e}(x),
\end{eqnarray} 
\beq\label{0.12}
 E_0(x)=- \partial_x \varphi_0(x),
 \ene
\beq\label{0.13}
- \partial^2_x \varphi_0(x)=  \rho_0(x).
\ene
We consider periodic steady states, the BGK waves,  of the form
   \beq\label{0.14}
f_{0,\pm}(x,v)= h_\pm\left(\mathcal E_\pm(x,v)\right),
\ene
with
\beq\label{0.15}
\mathcal E_\pm(x,v):= \frac{1}{2} v^2 \pm \varphi_0(x).
\ene
The functions $h_\pm$ are the energy distribution, respectively,  of ions and electrons. Further,   $\mathcal E_\pm$ is the energy that is constant in time for the  solutions to the characteristic   equations of the steady state for the ions, respectively, the electrons,
$$
\frac{d}{dt} x(t)= v(t), \qquad \dot{v}(t)= \mp \left(\frac{d}{d x}\varphi_0\right)(x(t)),
$$
or, equivalently, the solutions to Newton's equations
\beq\label{0.15b}
\ddot x(t)=  \mp \left(\frac{d}{d x}\varphi_0\right)(x(t)).
\ene
Remark that if $h_\pm$ are continuously differentiable $f_{0,\pm}(x,v)$ satisfy \eqref{0.9}.
Further, $\varphi_0(x)$ is the self-consistent potential that fulfills the Poisson equation,
\beq\label{0.16}
- \frac{d^2}{dx^2} \varphi_0(x)= q(\varphi_0(x))+ \rho_{\rm e}(x),
\ene
where,
\beq\label{0.17}
q(\lambda):=\int_{\mathbb R}   h_+\left( \frac{1}{2} v^2 + \lambda\right)\, dv-\int_{\mathbb R} h_-\left(\frac{1}{2} v^2 - \lambda\right)\, dv, \qquad \lambda\in \mathbb R.
\ene
We are interested in potentials $\varphi_0(x)$  that gives raise to regions of trapping of electrons and of ions. This trapping is  often present in BGK  waves.  It produces electron and ions holes, and it has important consequences in the dynamics of plasmas. See for example \cite{hut} and \cite{sch}. For this purpose, we assume that the potential   $\varphi_0(x)$ is a given function that is    even,  three  times continuously differentiable,  and  periodic. We denote the period by $P.$ Moreover,  we suppose that $\varphi_0(-P/2)=\varphi_0(P/2)= \min_{x \in [-P/2, P/2]} \varphi_0(x),$ and $\varphi_0(0)= \max_{x \in [-P / 2,P/2]}\varphi_0(x).$ Further, $\varphi_0(x)$ is strictly increasing for $x \in [-P/2,0]$ and strictly decreasing  for $ x\in [0, P /2].$ Moreover, the external charge $ \rho_{\rm e}(x)$ is computed  from the potential $\varphi_0(x)$ using \eqref{0.16}, \eqref{0.17}. A similar procedure is used in \cite{bruno2} in the  single species case, electrons, and an ion background  that acts as an external charge.

We are interested in studying   the  system \eqref{0.4}-\eqref{0.8}linearized around the BGK  wave \linebreak $(f_{0,\pm}(x,v), E_0(x),\varphi_0(x))$ . For this purpose, we consider solutions to \eqref{0.4}-\eqref{0.8} that are periodic in $x$ with period $P.$  Hence, we study
\eqref{0.4}-\eqref{0.8} for $(x,v) \in \mathbb T_{P} \times \mathbb R$ where $\mathbb T_{P}:= \mathbb R / P \mathbb Z$  is the one-dimensional torus of period $P$. We parametrize  $\mathbb T_{P}$ with the coordinate $ x \in [-P/2, P/2],$ where we identify $x=-P/2$ with $x=P/2.$ Observe that for periodic solutions \eqref{0.7} implies,
\beq\label{0.18}
\int_{-P/2}^{P/2} E(t,x) dx=0, \qquad t \in \mathbb R.
\ene
In the following we always assume that \eqref{0.18} holds.

We find it convenient to reformulate the   system  \eqref{0.4}-\eqref{0.8} as an equivalent 
Vlasov-Amp\`ere system that we introduce in Section~\ref{vabg}.  This has the advantage that, after linearization, the Vlasov-Amp\`ere system can be written as a Schr\"odinger type system  with the    Vlasov-Amp\`ere operator as a  {\it Hamiltonian}. The Vlasov-Amp\`ere operator is self-adjoint in an appropriate Hilbert space that we introduce. 
 This formulation turns out to be quite convenient   to bring to the fore the powerful methods of spectral and scattering theory of selfadjoint operators in order to study our problems. For a similar approach in the one specie case see \cite{bruno2}, \cite{bruno1}, and for the one specie case with a constant magnetic field see \cite{cdrw}.   Let us define the following operator
\beq\label{0.20}
\mathcal I f(x):= f(x)- \frac{1}{P} \int_{-P/2}^{P/2} f(x)\, dx.
\ene
In Section~\ref{vabg} we prove that the Vlasov-Poisson system \eqref{0.4}-\eqref{0.8} for solutions that satisfy \eqref{0.18} is equivalent to the following Vlasov-Amp\`ere system,
 \begin{eqnarray}\label{0.21}
\partial_t f_\pm(t,x,v)+v \partial_x f_\pm(t,x,v)\pm E(t,x)\partial_v f_\pm(t,x,v)=0, x \in (-P/2, P/2), t,v  \in \mathbb R,\\\label{0.22}
 \partial_t E(t,x)= \mathcal I \int_{\mathbb R} (-f_+(t,x,v)+ f_-(t,x,v))\, v \, dv,  x \in (-P/2, P/2) , t \in \mathbb R,\\
 \label{0.23}\int_{-P/2}^{P/2} E(t,x) dx=0,\qquad  t \in \mathbb R,\\ \label{0.24}
\partial_x E(0,x)=  \int_{\mathbb R}  f_+(0,x,v)\, dv -\int_{\mathbb R} f_-(0,x,v)\, dv+\rho_{e}(x),  x \in (-P/2, P/2).
\end{eqnarray} 
  Remark that  in \eqref{0.24} we assume that the Gauss law is valid initially at time zero.
  
We linearize the  system    \eqref{0.21}-\eqref{0.24} around  the BGK wave   $(f_{0,\pm}(x,v)= h_\pm(\mathcal E_\pm(x,v)), E_0(x)= -\partial_x \varphi_0(x))$  as follows. 
We assume that $h_\pm$ are continuously differentiable and that $h'_\pm <0.$
We take,
   \begin{eqnarray}\label{0.25}
  f_\pm(t,x,v)= f_{0,\pm}(x,v)+ \varepsilon  u_\pm(t,x,v)  \sqrt{|h_\pm'( \mathcal E_\pm(x,v)  )|}    ,\\ E(t,x)=E_0(x)+ \varepsilon F(t,x,v),\label{0.26}\\
  \int_{-P/2}^{P/2} F(t,x)\, dx=0.\label{0.27}
  \end{eqnarray}
We introduce \eqref{0.25}-\eqref{0.27} into  \eqref{0.21}-\eqref{0.24}, we keep the terms linear in $\varepsilon,$
and we obtain the linearized Vlasov-Amp\`ere system, 
   \begin{eqnarray}\label{0.28}
\partial_t u_\pm(t,x,v)+v \partial_x u_\pm(t,x,v)\mp \partial_x \varphi_0(x)\partial_ v u_\pm(t,x,v) \mp F(t,x) v \sqrt{|h'_\pm\left(\mathcal E_\pm(x,v)\right)|}
=0, \\\label{0.29}
 \partial_t F(t,x)= \mathcal I \int_{\mathbb R} \left(-u_+(t,x,v) \sqrt{|h'_+\left(\mathcal E_+(x,v)\right)|}      
  + u_-(t,x,v)   \sqrt{|h'_-\left(\mathcal E_-(x,v)\right)|}\right) \, v \,dv,\\ 
 \label{1.62a}\int_{-P/2}^{P/2} F(t,x) dx=0,\\\label{0.30}
\partial_x F(0,x)=  \int_{\mathbb R} ( u_+(0,x,v)\,  \sqrt{|h'_+\left(\mathcal E_+(x,v)\right)|}  -
 u_-(0,x,v)\, \sqrt{|h'_-\left(\mathcal E_-(x,v)\right)|})dv
\end{eqnarray} 
  where $x \in (-P/2, P/2),$ and  $t, v \in \mathbb R.$ To write the linearized Vlasov-Amp\`ere system as a Schr\"odinger type system we introduce some notation. We define
  \beq\label{0.31}
  U(t,x,v):= \begin{pmatrix} u_-(t,x,v)\\ u_+(t,x,v)\\ F(t,x)
  \end{pmatrix}.
  \ene
 Then, \eqref{0.28}, \eqref{0.29} are equivalent to the following  Vlasov-Amp\`ere system,
 \beq\label{0.32}
 i \partial_t  U= \mathcal A  U,
 \ene
 where the Vlasov-Amp\`ere operator $\mathcal A$ is given by
 \beq\label{0.33}
 \mathcal A= \mathcal A_0 + \mathcal V,
 \ene
 where,
 \beq\label{0.34}
 \mathcal A_0:= -i\begin{bmatrix} v \partial_x+\partial_x \varphi_0(x) \partial_v&0&0\\[.3cm]
 0&   v \partial_x-\partial_x \varphi_0(x) \partial_v&0\\[.3cm] 0&0&0\end{bmatrix},
 \ene
  and
    \beq\label{0.35}\mathcal V:= -i \begin{bmatrix} 0&0&  v \sqrt{|h'_-\left(\mathcal E_-(x,v)\right)|}\\[.3cm]
  0&0& - v \sqrt{|h'_+\left(\mathcal E_+(x,v)\right)|}\\[.3cm]
  -\mathcal I \int_{\mathbb R} \bullet \sqrt{|h'_-\left(\mathcal E_-(x,v)\right)|} \,v dv &      
  \mathcal I \int_{\mathbb R} \bullet   \sqrt{|h'_+\left(\mathcal E_+(x,v)\right)|} \, v \,dv&0
  \end{bmatrix}.
  \ene

 The operator $\mathcal A_0$ is the unperturbed Vlasov-Amp\`ere operator that describes the transport along the solutions to Newton's equations \eqref{0.15b} with the potential $\varphi_0(x)$ of the BGK wave and $\mathcal V$ is the perturbation that takes into account the interaction with the induced electric field $F(t,x).$ We consider the Vlasov-Amp\`ere system \eqref{0.32} in the Hilbert space of states,
  \beq\label{0.36} 
 \mathcal H:= L^2((-P/2,P/2) \times \mathbb R)\oplus  L^2((-P/2,P/2) \times \mathbb R)\oplus L^2_0((- P/2,P/2)),
 \ene
 where
 \beq\label{0.37}
 L^2_0((-P/2,P/2)):=\{ F \in L^2((-P/2,P/2)):  \int_{-P/2}^{P/2} F(x)\, dx=0 \}.
 \ene
The unperturbed Vlasov-Amp\`ere operator is  realized as a selfadjoint operator in $\mathcal H,$ that we also denote by $\mathcal A_0.$ Further,  the Vlasov-Amp\`ere operator $\mathcal A$ is selfadjoint with domain equal to the domain of $\mathcal A_0.$  Moreover, if  the Gauss law \eqref{0.30} holds at time zero, it holds for all times. 
  
  \begin{remark}\label{reno}{\rm We made the assumption  that $ h'_\pm <0$ to be able to write the linearized Vlasov-Amp\`ere system  as the Vlasov-Amp\`ere system \eqref{0.31}-\eqref{0.35} in the Hilbert space $\mathcal H$ \eqref{0.36}-\eqref{0.37}   with the selfadjoint Vlasov-Amp\`ere operator $\mathcal A$ as the {\it Hamiltonian}. However, for all our results   in this paper on the    Vlasov-Amp\`ere system \eqref{0.31}-\eqref{0.35} itself  this assumption is not necessary.}
  \end{remark}

 Note that, since $\mathcal A$ is selfadjont  the evolution given by \eqref{0.32} is unitary and hence, the norm in $\mathcal H$ of the solutions to the Vlasov-Amp\`ere system \eqref{0.32} are constant in time. So, we know from the scratch that the   Vlasov-Amp\`ere system is spectrally stable. Hence, we concentrate  in the asymptotic stability and we study the large time  behaviour,  in  particular  Landau-damping. To  solve these problems we use stationary scattering theory. Actually, stationary scattering theory is a powerful method that has been used in many problems in mathematical physics. See, for example,  \cite{aw}, \cite{kur},  \cite{rs3}, \cite{rw91}, \cite{ya}, \cite{ya1}, and the references quoted there.  In fact,  our approach  in stationary scattering theory  is general and  it can be used for other problems. Actually, we already used it  in \cite{wegal} in the gravitational case to study the asymptotic stability and Landau damping.
 
   We first construct a spectral representation of the unperturbed Vlasov-Amp\`ere operator $\mathcal A_0.$ For this purpose, we introduce appropriate energy-angle variables associated with the solutions to  Newton's equations with the potential $\varphi_0(x)$ of the  BGK wave.  In the region with trapped orbits for the electrons and the ions   the bottom of the  energy is an elliptic equilibrium point, and in the region with free orbits for electrons and ions the bottom of the energy is a hyperbolic equilibrium point. We have to take both  them into account, and this requires a careful definition of the angle variables in the various regions.  This allows us to obtain a spectral representation in terms of {\it trace maps} that we use in our perturbation analysis. Then, we study the  continuous and the singular spectra of the Vlasov-Amp \`ere operator, we construct the generalized Fourier maps and we prove that they are surjective, for this purpose our spectral representation in terms of {\it trace maps} plays a crucial role. Furthermore, we prove that the wave operators exist, are complete, are given by the stationary formulae, and that Birman's invariance principle holds.   We set up our perturbation analysis in spaces of H\" older continuous functions.
   The reason for this  is the following.  The perturbation term in the Vlasov-Amp\`ere  operator   includes an integral operator (see \eqref{0.33}, \eqref{0.35}, \eqref{6.1}, and \eqref{6.2}). The issue is that  the kernel of this integral operator is  singular at the elliptic and the hyperbolic points of the Hamiltonians for the motion of the electrons and the ions in the potential of the BGK wave. It is possible to  carefully analyze these singularities in spaces of   H\"older continuous functions, that  are less demanding concerning regularity than other spaces, like  Sobolev spaces.   
   
Using these results we obtain a precise description of the asymptotic stability of the linearized Vlasov-Amp\`ere system, and equivalently of the linearized Vlasov-Poisson system.  In Theorem~\ref{landam} we prove, for a large class of energy distributions $h_\pm$ and potentials $\varphi_0(x),$ 
that  the  phase-space densities $ f_\mp(t,x,v)$,   of the solutions to the Vlasov-Amp\`ere system  \eqref{0.32}-\eqref{0.35} with initial values in the subspace of absolute continuity of the Vlasov-Amp\`ere operator    are asymptotic for large times to the phase-space densities of solutions of the unperturbed Vlasov-Amp\`ere system,
\beq\label{0.38}
 i \partial_t  U= \mathcal A_0  U.
 \ene
 Namely, we prove in Theorem~\ref{landam},
 \beq\label{0.38xx}\begin{array}{l}
\lim_{t \to \infty}  \| f_-(t,\cdot,\cdot)-f^{(+)}_-(t,\cdot,\cdot) \|_{ L^2((-P/2,P/2) \times \mathbb R)}=0,
\\[.3cm]
\lim_{t \to \infty}  \| f_+(t,\cdot,\cdot)-   f^{(+)}_+(t,\cdot,\cdot) \|_{ L^2((-P/2,P/2) \times \mathbb R)}=0,
\\[.3cm]
\lim_{t \to -\infty}  \| f_-(t,\cdot,\cdot)-   f^{(-)}_-(t,\cdot,\cdot) \|_{ L^2((-P/2,P/2) \times \mathbb R)}=0,
\\[.3cm]
\lim_{t \to -\infty}  \| f_+(t,\cdot,\cdot)-   f^{(-)}_+(t,\cdot,\cdot) \|_{ L^2((-P/2,P/2) \times \mathbb R)}=0,
\end{array}
\ene 
 where  $f^{(+)}_\mp(t,x,,v)$ and  $f^{(-)}_\mp(t,x,v)$ are solutions to the  unperturbed Vlasov-Amp\`ere system \eqref{0.38}  that are defined in Theorem~\ref{landam} in terms of the initial values of the solution to the Vlasov-Amp\`ere system \eqref{0.32}-\eqref{0.35} and of the wave operators. We remark that the solutions to the Vlasov-Amp\`ere system  \eqref{0.32}-\eqref{0.35} with initial values in the subspace of absolute continuity of the Vlasov-Amp\`ere operator satisfy the Gauss law \eqref{0.30} for all $ t \in \mathbb R.$ This gives a precise description of the large time behaviour of the phase-space densities. Namely, they follow the trajectories  of the solutions to Newton's equations with the potential of the  BGK wave, in the sense that they are transported along these trajectories.          
 Moreover, we prove in Theorem~\eqref{landam} that  Landau damping holds for the electric field in pointwise sense. Namely, that, for $ 0<\alpha < 1/2,$
 \beq\label{0.39} 
\lim_{t \to \pm \infty} \| F(t,\cdot)\|_{\hat{C}_{\alpha}((-P/2,P/2))}=0,
\ene
where $  \| \cdot\|_{\hat{C}_{\alpha}((-P/2,P/2))} $ is the norm in the Banach space  $\hat{C}_{\alpha}((-P/2,P/2))$  of  H\" older continuous functions of order $\alpha$ in $(-P/2,P/2).$    To our knowledge, Theorem~\ref{landam} is the first time
where  Landau damping is proved to hold for a two species Vlasov-Poisson system in a nonhomogeneous background. 

The stability of BGK waves is currently a topic of active  research in the physics literature. See \cite{hut2024} and \cite{sch2012} for reviews.
Previous results in the mathematics literature   are the following. In \cite{gl} the linear spectral stability of nonhomogeneous BGK waves is proved in the cases of one and of two  species. In the pioneering work \cite{bruno2} B. Despr\'es proved for the first time Landau damping, i.e. the decay of the electric field,   in pointwise sense,  for  the one-dimensional linearized Vlasov-Poisson system with one species (electrons) in a nonhomogeneous background, with one region of trapped particles. He considers a Boltzmannian  energy distribution and solutions with small period. He postulates the properties of the potential of the steady state and  defines the fixed density of ions  by means of the Poisson equation. He  works with a  Vlasov-Amp\`ere system and a Lippmann-Schwinger variational equation. He proves that the Lippmann-Schwinger equation is well posed if the period is small enough. He uses a representation in terms of Hermite functions that is valid for Boltzmannian distributions.
See also \cite{bruno1} where
 the linearized one-dimensional Vlasov-Poisson system with one species (electrons),   Boltzmannian energy distribution, and a general electric potential  for the steady state is considered. A trace class criterion  is obtained for the perturbed and the unperturbed Vlasov Amp\`ere systems,  that implies  the existence and the completeness of the wave operators. The paper \cite{hm} considers a set up that is similar to the one of \cite{bruno2}.  However, \cite{hm} considers a general class of energy distributions. Further, in \cite{hm} solutions with small period are studied, The paper \cite{hm} proves that the linearized operator around the BGK wave has no embedded eigenvalues inside the essential spectrum. Further, in \cite{hm} it is proved  that for a subclass  of BGK waves, with monotone  energy distribution, for the solutions with  initial states in the intersection of the domain of the transport operator with the closure of its range  the time average of the $L^2$ norm of the electric field tends to zero for large times, provided that the parameter that controls the size of the period and of the electron hole  is small enough. The paper \cite{bghp} proves that the linear stability/instability of small BGK waves in the one species case (electrons) is determined by the sign of the derivative of the  energy distribution at zero energy. Furthermore, there are results on the instability of BGK waves against perturbations with a period that is a multiple, larger than one, of the period   of the BGK wave. See for example, \cite{guostra}.

  Landau damping is a fundamental phenomenon in plasma physics that was discovered in the seminal  work of L. D Landau \cite{landau}. It has been extensively studied. A major breakthrough was the work of Mouhot and Villani  \cite{villani}, who proved   Landau damping in the nonlinear case. See also \cite{bmm} and \cite{grn}.  When the Vlasov-Poisson system is considered with an external constant magnetic field  Landau damping disappears and the electric field is oscillatory in time, as was shown by Bernstein \cite{bernstein}. For the study of this problem in the mathematics literature see \cite{bedro} and 
 \cite{cdrw}. In \cite{cdrw} it was also proved that there are time independent solutions. For recent reviews of these results, that  are for perturbations of homogeneous steady states, see  \cite{bedro1} and \cite{nyu}.

 The paper is organized as follows. In Section~\ref{notdef} we   introduce  notations and  definitions that we use. In Section~\ref{twospec} we  introduce  the BGK waves that we consider. In Section~\ref{vabg} we introduce the Vlasov-Amp\`ere system and we prove that it is equivalent to the Vlasov-Poisson system. Furthermore, we linearize the Vlasov-Amp\`ere system around the BGK wave. Moreover, we formulate the linearized Vlasov-Amp\`ere system as an equation  of Schr\"odinger type in the Hilbert space $\mathcal H$ that we introduce. In Section~\ref{enan} we introduce the appropriate energy-angle variables for the regions with trapped and with free orbits for the electrons and for the ions.  In Section~\ref{unp} we construct a selfadjoint  realization of the formal unperturbed Vlasov-Amp\`ere operator and we construct  a spectral representation in terms of {\it trace maps}.
In Section~\ref{vlasamp} we obtain our results in the absolutely continuous and the singular spectra of the Vlasov-Amp\`ere operator.    In Section~\ref{fouma}  we construct the generalized Fourier maps and we prove that they are surjective. In Section~\ref{wave} we prove that the wave operators exist, are complete, are given by the stationary formulae, and that Birman's invariance principle holds. In Section~\ref{ladam} we obtain our results   on the large time asymptotic of the distribution functions and on the Landau damping of the electric field.
Finally, in Appendix~\ref{apex} we give auxiliary results on the energy-angle variables that we need.

\section{Notations and definitions}\label{notdef}

We denote by  $\mathbb S_1$  the unit circle.   We parametrize  $S_1$ with the coordinate $ \mu \in [0, 1],$ where we identify $\mu=0$ with $\mu==1.$
 Further,   for  $ P >0$ we denote by $\mathbb T_{P}:= \mathbb R  / P\mathbb Z$   the one-dimensional torus of period $P$. We parametrize  $\mathbb T_{P}$ with the coordinate $ x \in [-P/2, P/2],$ where we identify $x=-P/2$ with $x=P/2.$  
 For a set $I \subset \mathbb R$   we designate by  $\chi_I(x)$ the characteristic function of $I.$  
 For any interval $[a,b]$ we denote,
 $
 C^1_{\rm p}([a,b]):=\{ f \in C^1([a,b]): f(a)=f(b)\}.
 $
For a separable Hilbert space $ \mathcal G,$  for an open  interval $I \subset \mathbb R,$ and for $0< \alpha \leq 1,$  we denote by  $C_\alpha(I, \mathcal G)$ the Banach space of all bounded  functions defined on $I,$ with values in $\mathcal G,$ that are H\"older continuous with exponent $\alpha.$  
Further, we designate by $\hat{C}_\alpha(I, \mathcal G)$ the completion of 
$C^\infty(I, \mathcal G)$ in the norm of $C_\alpha(I, \mathcal G).$ Clearly $ \hat{C}_\alpha(I,\mathcal G) \subset C_\alpha(I, \mathcal G).$ If $I$ is bounded, by Proposition ~A.1 of \cite{wegal}, for $ \alpha_1 > \alpha_2 $ we have, $C_{\alpha_1}(I,\mathcal G) \subset  
\hat{C}_{\alpha_2}(I,\mathcal G).$ 
Moreover, if $I$ is bounded we denote,
$
\hat{C}_{\alpha,0}(I, \mathcal G):=\{ f \in \hat{C}_\alpha(I, \mathcal G): \int_I f=0  \}.
$
In the case $\mathcal G=\mathbb C$ we use the notation $C_\alpha(I), \hat{C}_\alpha(I),$ and $\hat{C}_{\alpha,0}(I).$
Let    $I$ be a set of real numbers, $\mu$ a sigma-finite complete measure on $I,$ and $\mathcal G$ a Hilbert space. By $L^2(I, \mathcal G; \mu)$ we denote the Hilbert space of all  functions from $I$ into $\mathcal G$ that are measurable and square integrable (Bochner)  with respect to $\mu.$ 
If $\mu$ is the Lebesgue measure we use the notation  $L^2(I, \mathcal G).$  For the definition of these spaces see \cite{hp}.  
For an interval $(a,b)$ we denote by $H^1((a,b))$  the Sobolev space of all  complex-valued, square-integrable functions defined on $(a,b)$ whose  derivate   in distribution sense  is given by a square-integrable function \cite{adams}. We denote by $H^{1,{\rm p}}((a,b))$ the closed subspace of $H^1((a,b))$  of all  functions $f$ in $H^1((a,b))$ that satisfy $f(a)=f(b).$  Moreover we designate, $ H^{1,0}((a, b)):=\{ f \in H^{1}(a,b): \int_{(a, b)}  f(x) dx =0 \}.$
For an interval $I$ and a separable Hilbert space $\mathcal G,$ we denote by $H_1(I, \mathcal G)$ the 
closure of $C^\infty_0(\overline{I}, G)$ in the norm
$
\begin{array}{l}
\| f\|_{H_{1}(I, G)}:= \left[  \| f\|^2_{L^2(I, G)} +|\| f'  \|^2_{L^2(I, \mathcal G)}\right]^{1/2}.
\end{array}
$
$H_1(I, \mathcal G)$ is a Hilbert space.
 We denote by $C$ a generic constant that does not have  to take the same value when it appears in different places.

\section{The BGK waves} \label{twospec}
We consider periodic   steady states, the BGK waves, of the form \eqref{0.14}, \eqref{0.15}.
In the following assumption we state conditions that we assume on the energy distributions  $h_\pm$ and on the potential of the BKG waves. As mentioned in the introduction, we assume that the potential satisfies assumptions that correspond to the existence of  electron and ion holes, and that the external charge $\rho_{\rm e}(x)$ is computed used \eqref{0.16} and \eqref{0.17}. 
 \begin{assumption} \label{assum0} {\rm
The following holds.
\begin{enumerate}
\item[\rm (a)] The energy distributions  $h_\pm(\lambda)$ are nonnegative   continuously differentiable functions defined for $\lambda \in \mathbb R$ such  that  $h_\pm(v^2/2)$  are integrable functions of $v\in \mathbb R.$

\item[(b)]
The potential $\varphi_0(x)$ is real valued,  even,  three  times continuously differentiable and periodic with period $P.$ 
 Moreover, $\varphi_0(-P/2)=\varphi_0(P/2)= \min_{x \in [-P/2, P/2]} \varphi_0(x),$ and
$\varphi_0(0)= \max_{x \in [-P / 2,P/2]}\varphi_0(x).$ Further, $\varphi_0(x)$ is strictly increasing for
$x \in [-P/2,0]$ and strictly decreasing  for $ x\in [0, P /2].$ Moreover, $\varphi''(-P/2)= \varphi''(P/2) >0,$ and
 $\varphi''(0) <0.$ The charge density $\rho_{\rm e}(x)$ is computed from the potential   $\varphi_0(x)$ using  \eqref{0.16} and \eqref{0.17}.

\item[\rm (c)]
The following inequalities hold,

\beq\label{1.22bb}
(\varphi'_0(x))^2+2 (\varphi_0(-P/2)-\varphi_0(x)) \varphi_0''(x) > 0, \qquad  \,\text{\rm for a.e.}\,    x \in [-P, 0],
\ene
and
\beq\label{1.22b}
(\varphi'_0(x))^2+2 (\varphi_0(0)-\varphi_0(x)) \varphi_0''(x) > 0, \qquad  \,\text{\rm for a.e.}\,    x \in [-P/2, P/2].
\ene

\end{enumerate}
}
\end{assumption}
In Item (c) in Assumption~\ref{assum0}  equation  \eqref{1.22bb}, respectively equation  \eqref{1.22b},   is a sufficient condition  for the strict monotonicity of the period of the periodic orbits of Newton's equation  \eqref{0.15b} with the potential of the BGK wave, for  the ions, respectively the electrons.   See \cite{cw} and the proof of items (d) and  (i) of Proposition~\ref{proper}. If in particular cases it is possible to prove directly the monotonicity of these periods, item (c) in Assumption~\ref{assum0} can be dropped.  An example of a potential that fulfills Asumption~\ref{assum0} is the potential of the HMF  (Hamiltonian Mean Field) model   $\varphi_0(x) = \beta \cos{\frac{2\pi x}{P} }$ with a positive constant $\beta.$ See  \cite{bsdn}- \cite{hm3}, \cite{hm6}-\cite{hm12}, \cite{hm13},  \cite{hm1}, and \cite{hm2}.  
\section{The Vlasov-Amp\`ere system}\label{vabg}
We are interested in studying   the Vlasov-Poisson system \eqref{0.4}-\eqref{0.8} linearized around  BGK waves that fulfill Assumption~\ref{assum0}. For this purpose, we consider solutions to \eqref{0.4}-\eqref{0.8} that are periodic in $x$ with period $P.$ We recall  that for periodic solutions \eqref{0.7} implies  \eqref{0.18}. 
In the following we always assume that \eqref{0.18} holds, and we study
\eqref{0.4}-\eqref{0.8} for $(x,v) \in \mathbb T_{P} \times \mathbb R.$ 
As mentioned in the introduction, we find it convenient to reformulate the  Vlasov-Poisson system  \eqref{0.4}-\eqref{0.8} as an equivalent Vlasov-Amp\`ere system, as was done in 
 \cite{bruno2}, \cite{bruno1}, and  \cite{cdrw} in the one species case.
    
Suppose that $(f_\pm, E)$ is a solution to  \eqref{0.4}-\eqref{0.6} for $ x \in (-P/2, P/2), t, v \in \mathbb R,$ that fulfills \eqref{0.18} and with $f_\pm$  that tend to zero fast enough as  $v\to \pm \infty.$ Then, 
\beq\label{1.39}
\partial_x\left( \partial_t E(t,x)+ \int_{\mathbb R}\,  (f_+(t,x,v)-f_-(t,x,v)) \, v \,dv \right)= 0.
\ene
By \eqref{1.39},
\beq\label{1.40}
 \partial_t E(t,x)+ \int_{\mathbb R} (f_+(t,x,v)-f_-(t,x,v)) \, v \,dv =f(t),
 \ene
 for a function $f$ that only depends on $t.$ Further, integrating both sides of \eqref{1.40}  from $x= -P/2$
  to $ x= P/2$ and using \eqref{0.18} we obtain,
 \beq\label{1.41}
 P f(t)= \int_{-P/2}^{P/2} \, dx \left( \int_{\mathbb R}\,  (f_+(t,x,v)-f_-(t,x,v)) \, v\, dv\right).
 \ene
 By \eqref{1.40} and \eqref{1.41} we get,
 \beq\label{1.42}
 \partial_t E(t,x)= \mathcal I \int_{\mathbb R} (-f_+(t,x)+ f_-(t,x))\, v dv, \qquad x \in (-P/2, P/2), t \in \mathbb R,
 \ene
where $\mathcal I$  is the operator defined in \eqref{0.20}.
Hence, we have proved that the solution $(f_\pm, E)$ satisfies the Vlasov-Amp\`ere system \eqref{0.21}, \eqref{0.22}. On the other hand, assume  that $(f_\pm, E)$ is a solution to the Vlasov-Amp\`ere system \eqref{0.21}, \eqref{0.22} for  $ x \in (-P/2, P/2), t, v \in \mathbb R,$ that fulfills \eqref{0.18}  and that tends to zero as $v\to \pm \infty.$ Then,
\beq\label{1.44}
\begin{array}{l}
\partial_t(  \partial_xE(t,x)-  \int_{\mathbb R} (f_+(t,x)- f_-(t,x))\,  dv -\rho_e(x))= \\\partial_x  \int_{\mathbb R} (-f_+(t,x)+ f_-(t,x))\, v dv -\partial_t   \int_{\mathbb R} (f_+(t,x)- f_-(t,x)\,  dv =0.
\end{array}
\ene
Equation \eqref{1.44} shows that  Gauss law \eqref{0.5} is satisfied for all times provided that it holds initially at time $t=0.$
Summing up we have proved that the  system   \eqref{0.4}-\eqref{0.6},  for $ x \in (-P/2, P/2), t, v \in \mathbb R,$ for solutions that fulfills \eqref{0.18}  is equivalent to the Vlasov-Amp\`ere system \eqref{0.21}-\eqref{0.24}.  Recall that  in \eqref{0.24} we assume that the Gauss law is valid initially at time zero. Note that when we formulate the Vlasov-Poisson system as the equivalent Vlasov-Amp\`ere system we do not introduce the electric potential \eqref{0.7}. However, we keep \eqref{0.18} that follows from \eqref{0.7} for  $\varphi(t,x)$  periodic in $x.$
  
   In what follows we study the Vlasov-Amp\`ere system \eqref{0.21}- \eqref{0.24}.
  \subsection{Linearization}\label{linear}
  Let $(f_{0,\pm}(x,v)= h_\pm(\mathcal E_\pm(x,v)), E_0(x)= -\partial_x \varphi_0(x))$ be a  BGK wave that satisfies Assumption~\ref{assum0}. Clearly,  $(f_{0,\pm}(x,v)=h_\pm(\mathcal E_\pm(x,v)), E_0= -\partial \varphi_0(x))$ is a solution to the Vlasov-Poisson system \eqref{0.4}-\eqref{0.8}, \eqref{0.18}  and if $h_\pm(\mathcal E_\pm(x,v)) v$ are integrable functions of $v$ for $ x \in [-P/2, P/2]$ it solves  the Vlasov-Amp\`ere system \eqref{0.21}-\eqref{0.24}. We linearize the Vlasov-Poisson system \eqref{0.4}-\eqref{0.8}, \eqref{0.18} around  $(f_{0,\pm}(x,v)= h_\pm(\mathcal E_\pm(x,v)), E_0(x)= -\partial_x \varphi_0(x))$ as follows. We take
  \begin{eqnarray}\label{1.49}
  f_\pm(t,x,v)= f_{0,\pm}(x,v)+ \varepsilon g_\pm(t,x,v),\\ E(t,x)=E_0(x)+ \varepsilon F(t,x,v),\label{1.50}\\
  \int_{-P/2}^{P/2} F(t,x)\, dx=0.\label{1.51}
  \end{eqnarray}
  We introduce \eqref{1.49}-\eqref{1.51} into \eqref{0.4}-\eqref{0.6}, \eqref{0.18},
  we keep only the terms linear in $\varepsilon$ and we obtain the linearized Vlasov-Poisson  system
  \begin{eqnarray}\label{1.52a}
  \partial_t g_\pm(t,x,v)+v \partial_x g_\pm(t,x,v)\mp \partial_x \varphi_0(x)\partial_ v g_\pm(t,x,v) \pm F(t,x) v h'_\pm\left(\mathcal E_\pm(x,v)\right)
=0, \\\label{1.53a}
\partial_x F(t,x)= \rho(t,x), \\
\label{1.54a} \rho(t,x):= \int_{\mathbb R}   g_+(t,x,v)\, dv-\int_{\mathbb R} g_-(t,x,v)\, dv,\\\label{1.54ab}\int_{-P/2}^{P/2} F(t,x) dx=0.
\end{eqnarray}
We linearize the  Vlasov-Amp\`ere system \eqref{0.21}-\eqref{0.24} in the same way and we obtain the linearized Vlasov-Amp\`ere system, 
   \begin{eqnarray}\label{1.52}
\partial_t g_\pm(t,x,v)+v \partial_x g_\pm(t,x,v)\mp \partial_x \varphi_0(x)\partial_ v g_\pm(t,x,v) \pm F(t,x) v h'_\pm\left(\mathcal E_\pm(x,v)\right)
=0, \\\label{1.53}
 \partial_t F(t,x)= \mathcal I \int_{\mathbb R} (-g_+(t,x,v)+ g_-(t,x,v))\, v dv,\\
 \label{1.54}\int_{-P/2}^{P/2} F(t,x) dx=0,\\\label{1.55}
\partial_x F(0,x)=  \int_{\mathbb R} ( g_+(0,x,v)\, dv -\int_{\mathbb R} g_-(0,x,v)\, dv.
\end{eqnarray} 

As in the proof of \eqref{1.44} we prove  that the Gauss law 
\beq\label{1.55b}
\partial_xF(t,x)=  \int_{\mathbb R} g_+(t,x,v) dv   - \int_{\mathbb R}g_-(t,x,v)\,  dv, \qquad x\in (-P/2, P/2), t \in \mathbb R 
 \ene
 is satisfied for all times since by \eqref{1.55} it holds for $t=0.$ 
  The systems \eqref{1.52a}-\eqref{1.54ab} and \eqref{1.52}-\eqref{1.55} are equivalent. The proof is the similar to the proof of the equivalence of the systems  \eqref{0.4}-\eqref{0.8}, \eqref{0.18}  and  \eqref{0.21}-\eqref{0.24} that we gave above.
  
  In what follows we study the linearized Vlasov-Amp\`ere system \eqref{1.52}-\eqref{1.55}. 
 To formulate   \eqref{1.52}-\eqref{1.55} as a system of Schr\"odinger type in an appropriate Hilbert space we assume that $h'_\pm(\lambda) <0,$ and we introduce the  new dependent variables $  u_\pm(t,x,v)$ such that,
  \beq\label{1.59}
  g_\pm(t,x,v)=  u_\pm(t,x,v) \, \sqrt{|h_\pm'( \mathcal E_\pm(x,v)  )|}.
  \ene
  In terms of $(u_\pm(t,x,v), F(t,x))$ the linearized Vlasov-Amp\`ere system  \eqref{1.52}- \eqref{1.55}  is written as in \eqref{0.28}-\eqref{0.30}.
As mentioned in the introduction, we write the linearized Vlasov-Amp\`ere system \eqref{0.28}-\eqref{0.30} as the Vlasov-Amp\`ere system of Schr\"odinger type \eqref{0.31}-\eqref{0.35} in the Hilbert space of states \eqref{0.36}, \eqref{0.37}. Further, the Gauss law  \eqref{0.30} holds at $t=0.$ Recall that, as mentioned in Remark~\ref{reno},  the assumption that  $h'_\pm(\lambda) <0$ is only used to write the linearized  Vlasov-Amp\`ere system   \eqref{1.52}-\eqref{1.55}  as in \eqref{0.32}-\eqref{0.35} and it is not necessary for our further results in this paper.
   
   Summing up, our problem is reduced to the study of the solutions to the Vlasov-Amp\`ere system \eqref{0.31}, \eqref{0.32},  with the Vlasov-Amp\`ere operator defined in \eqref{0.33}, \eqref{0.34}, and \eqref{0.35}, in the Hilbert space $\mathcal H$ defined in \eqref{0.36} and \eqref{0.37}. Further, we require that initially, at time zero, the solutions satisfy the Gauss law \eqref{0.30}.

     By a direct computation it can be verified that the  Vlasov-Amp\`ere operator $\mathcal A$  with domain $ C^1_{\rm p}([-P/2,P/2])\otimes C^1_0(\mathbb R)\oplus  C^1_{\rm p}([-P/2,P]/2)\otimes C^1_0(\mathbb R)\oplus L^2_0((-P/2,P/2))$ is a  symmetric operator in $\mathcal H.$ We will construct a selfadjoint extension of $\mathcal A,$ but before that we introduce  the energy-angle variables in the next section.

  \section{The energy- angle variables}\label{enan}
   Action-angle variables \cite{arnold} are extensively used in the study of the Vlasov-Poisson system. See, for example, \cite{bsdn}, \cite{gl},  \cite{mo1}, and \cite{mo2}. In this section  we use   closely related energy-angle variables that  play an important role in our study of the spectral and scattering theory of the Vlasov-Amp\`ere operator. 

  In this section we always assume that items (b)  and (c) of Assumption~\ref{assum0} hold. We begin with the 
  discussion of the appropriate angles.
  \subsection{Angles for the electrons} \label{anelec}
  We have to distinguish the region with trapped orbits  and the region with  free orbits  for the  solutions of  Newton's equation  for the electrons  with the potential of the  BGK wave
\beq\label{1.73}
\ddot x(t)=\varphi_0(x(t)).
\ene
Recall that the energy   that is constant in time  for the solutions to \eqref{1.73} is given by
  \beq\label{1.74}
  \mathcal E_-:= \frac{1}{2} v^2- \varphi_0(x).
  \ene
  \subsubsection{The region with trapped orbits for the electrons}
  Let us denote
  \beq\label{1.76}
  I_{-, 0}:= (-\varphi_0(0), -\varphi_0(P/2)).
  \ene
  For $ \mathcal E_- \in I_{-,0}$ the solutions to \eqref{1.73} are periodic orbits that oscillate between  $ x_-(\mathcal E_-)$ and  $ x_+(\mathcal E_-) $ with period,
  \beq\label{1.77}
  T_{-,0}(\mathcal E_-)= 2 \int_{x_-(\mathcal E_-)}^{x_+(\mathcal E_-)}
   \frac{1}{\sqrt{2(\mathcal E_-+\varphi_0(x))}}\, dx,\qquad \mathcal E_- \in I_{-,0}.
  \ene
  The quantities     $x_\pm(\mathcal E_-)$    are the turning points of the orbit. They satisfy      $ -P/2 < x_-(\mathcal E_-) < 0 < x_+(\mathcal E_-) < P/2$      and they are given by,
  \beq\label{1.78}
  -\varphi_0(x_\pm(\mathcal E_-))= \mathcal E_-, \qquad \mathcal E_- \in I_{-,0}.
  \ene
  Since $\varphi_0(x)$ is even, $x_-(\mathcal E_-)= -x_+(\mathcal E_-).$
  In this region we define the angle as follows,
  \beq\label{1.79}\begin{array}{c}
  \theta_{-,0}(x,\mathcal E_-):=  \ds \frac{1}{T_{-,0}(\mathcal E_-)}  \int_{x_-(\mathcal E_-)}^{x}
  \ds \frac{1}{\sqrt{2(\mathcal E_-+\varphi_0(y))}}\, dy,\\[.3cm] \mathcal E_-\in I_{-,0},\qquad
    x_-(\mathcal E_-)\leq x \leq x_+(\mathcal E_-).
    \end{array}
    \ene
    Note that $ 0 \leq \theta_{-,0}(x,\mathcal E_-) \leq 1/2.$
   \subsubsection{The region with free orbits for the electrons}
   We define,
   \beq\label{1.80}
   I_{-,1}:= ( -\varphi_0(P/2), \infty).
   \ene
   In this region the solutions to  \eqref{1.73} with energy $\mathcal E_-\in  I_{-,1}$ travel from
   $x=-P/2$ to $x= P/2$ with positive velocity or from    $x=P /2$ to $x= -P/2$ with negative velocity. The period,
   or the travel time, in this case is the time that they take traveling from  $x=-P/2$ to $x= P/2$  or from     $x=P/2 $ to $x= -P/2.$  It is given by
   \beq\label{1.81}
   T_{-,1}(\mathcal E_-)=  \int_{-P/2}^{P/2}
   \frac{1}{\sqrt{2(\mathcal E_-+\varphi_0(x))}}\, dx,\qquad \mathcal E_- \in I_{-,1}.
  \ene
    In this region we define the angle as follows,
  \beq\label{1.82}\begin{array}{c}
  \theta_{-,1}(x,\mathcal E_-):= \ds  \frac{1}{T_{-,1}(\mathcal E_-)}  \int_{-P/2}^{x}
  \ds \frac{1}{\sqrt{2(\mathcal E_-+\varphi_0(y))}}\, dy,\\ [.3cm]\mathcal E_-\in I_{-,1},\qquad
    -P/2\leq x \leq  P/2.
   \end{array}
    \ene
   Note that $ 0 \leq \theta_{-,1}(x,\mathcal E_-) \leq 1.$
  
  \subsection{Angles  for the ions} \label{anions}
As in the case of the electrons we have  to consider the region with trapped orbits and the region with free orbits 
  for the  solutions of  Newton's equation  for the ions  with the potential of the  BGK wave 
\beq\label{1.83}
\ddot x(t)=-\varphi_0(x(t)).
\ene
Recall that the energy   that is constant in time for the solutions to \eqref{1.83} is given by
  \beq\label{1.84}
  \mathcal E_+:= \frac{1}{2} v^2+ \varphi_0(x).
  \ene

  \subsubsection{The region with trapped orbits for the ions}
  Let us denote
  \beq\label{1.86}
  I_{+, 0}:= ( \varphi_0(P/2), \varphi_0(0)).
  \ene
In this region there are two types of solutions to \eqref{1.83} with energy  $\mathcal E_+$ in $ I_{+, 0}$. First, the trapped orbit  that travels  from  $ x=-P/2$ to $x_-(\mathcal E_{+})<0,$ and back to $x=-P/2.$ The quantity  $ x_-(\mathcal E_{+}) < 0$ is the smallest of the two solutions to
\beq\label{1.87}
    \varphi_0(x)=\mathcal E_+.
    \ene
   The  period, or the travel time is  $T_{+,0}(\mathcal E_+)/2$    where
\beq\label{1.88}
   T_{+,0}(\mathcal E_+)= 4 \int_{-P/2}^{x_-(\mathcal E_+)}
   \frac{1}{\sqrt{2(\mathcal E_+-\varphi_0(x))}}\, dx,\qquad \mathcal E_+ \in I_{+,0}.
  \ene
  In this region we define the period, or the travel time, as $T_{+,0}(\mathcal E_+)/2$ to simplify the notation later on. Note that the quantity $T_{+,0}(\mathcal E_+)$ is the period in phase space $(x,v)$ upon the identification of
  $-P/2$ with $P/2.$
  For these solutions we define the angle   as follows
    \beq\label{1.89}\begin{array}{c}
  \theta_{+,0,-}(x,\mathcal E_+):= \ds  \frac{1}{ T_{+,0}(\mathcal E_+)}  \int_{-P/2}^{x}
  \ds \frac{1}{\sqrt{2(\mathcal E_+-\varphi_0(y))}}\, dy,\\ [.3cm]\mathcal E_+\in I_{+,0},\qquad
    -P/2 \leq x \leq x_-(\mathcal E_+).
    \end{array}
    \ene
    Note that $ 0 \leq \theta_{+,0,-}(x,\mathcal E_+) \leq 1/4.$
The second type of solutions travels from $x=P/2$ to   $ x_+(\mathcal E_{+})$ 
and back to $x=P/2.$ The quantity  $ x_+(\mathcal E_{+})>0$ is the largest of the two solutions to  \eqref{1.87}.
Note that as $\varphi_0(x)$ is even,   $x_+(\mathcal E_{+})= -x_-(\mathcal E_{+}).$ Furthermore, the  period, or the travel time, of these solutions  is  $T_{+,0}(\mathcal E_+)/2.$  For these solutions we define the angle   as follows
    \beq\label{1.91}\begin{array}{c}
  \theta_{+,0,+}(x,\mathcal E_+):= \ds \frac{1}{T_{+,0}(\mathcal E_+)}  \int_{ P/2}^{x}
   \ds \frac{1}{-\sqrt{2(\mathcal E_+-\varphi_0(y))}}\, dy,\\ \mathcal E_+\in I_{+,0}, \qquad  x_+(\mathcal E_{+}) \leq x \leq P/2.
   \end{array}
    \ene
    Note that $ 0 \leq \theta_{+,0,+}(x,\mathcal E_+) \leq 1/4.$  Remark that as $\varphi_0(x)$ is even and $x_-(\mathcal E_-)=- x_+(\mathcal E_+),$ we have,  $ \theta_{+,0,-}(x,\mathcal E_+)=   \theta_{+,0,+}(-x,\mathcal E_+),$ for $  -P/2 \leq x \leq x_-(\mathcal E_+).$ 
  
     \subsubsection{The region with free orbits for the ions}
   We define,
   \beq\label{1.92}
   I_{+,1}:= ( \varphi_0(0), \infty).
   \ene
   In this region the solutions to  \eqref{1.83} with energy $\mathcal E_+\in  I_{+,1}$ travel from
   $x=-P/2$ to $x= P/2$ with positive velocity, or from    $x=P/2 $ to $x= -P/2$ with negative velocity. The period,
   or the travel time, in this case is the time that they take traveling from  $x=-P/2$ to $x= P/2$  or from     $x=P/2 $ to $x= -P/2.$  It is given by
   \beq\label{1.93}
   T_{+,1}(\mathcal E_+)=  \int_{-P/2}^{P/2}
   \frac{1}{\sqrt{2(\mathcal E_+-\varphi_0(x))}}\, dx,\qquad \mathcal E_+ \in I_{+,1}.
  \ene
   In this region we define the angle as follows,
  \beq\label{1.94}\begin{array}{c}
  \theta_{+,1}(x,\mathcal E_+):=  \ds \frac{1}{T_{+,1}(\mathcal E_+)}  \int_{-P/2}^{x}
  \ds \frac{1}{\sqrt{2(\mathcal E_+-\varphi_0(y))}}\, dy,\\[.4cm] \mathcal E_+\in I_{+,1},\qquad
    -P/2 \leq x \leq  P/2.
    \end{array}
    \ene
   Note that $ 0 \leq \theta_{+,1}(x,\mathcal E_+) \leq 1.$
\subsection{The transformation to energy-angle variables}
In this subsection we introduce the  transformation to energy-angle variables. We first introduce some notation.
Recall that we  parametrize  $\mathbb T_{P}$ with the coordinate $ x \in [-P/2, P/2],$ where we identify $x=-P/2$ with $x=P/2.$

We define,
\beq\label{1.9}
\Omega_{-,0}:=\{ (x,v)\in \mathbb T_{P} \times \mathbb R: \mathcal E_-(x,v) \in I_{-,0}\},
\ene
\beq\label{1.96}
\Omega_{-,1}^+:=\{ (x,v)\in \mathbb T_{P}  \times \mathbb R: \mathcal E_-(x,v) \in I_{-,1}, v > 0\},
\ene
\beq\label{1.97}
\Omega_{-,1}^-:=\{ (x,v)\in  \mathbb T_{P} \times \mathbb R: \mathcal E_-(x,v) \in I_{-,1}, v <0 \},
\ene
\beq\label{1.98}\begin{array}{c}
\Omega_{+,0}^{-}:=\{ (x,v)\in \mathbb T_{P}  \times \mathbb R: \mathcal E_+(x,v) \in I_{+,0}, -P/2 \leq x\leq    x_-(\mathcal E_+)\},
\end{array}
\ene
\beq\label{1.99}\begin{array}{c}
\Omega_{+,0}^{+}:=\{ (x,v)\in \mathbb T_{P}  \times \mathbb R: \mathcal E_+(x,v) \in I_{+,0}, x_+(\mathcal E_+) \leq x\leq P/2 \},
\end{array}
\ene
\beq\label{1.101}
\Omega_{+,1}^+:=\{ (x,v)\in \mathbb T_{P} \times \mathbb R: \mathcal E_+(x,v) \in I_{+,1}, v >0\},
\ene
\beq\label{1.100}
\Omega_{+,1}^-:=\{ (x,v)\in \mathbb T_{P} \times \mathbb R: \mathcal E_+(x,v) \in I_{+,1}, v <0\},
\ene
\beq\label{1.103}
\Omega_{-, \rm s}:= \{ (x,v) \in \mathbb T_{P} : \mathcal E_-(x,v)= -\varphi_0(-P/2)=- \varphi_0(P/2) \},
\ene
and
\beq
\label{1.102}
\Omega_{+, \rm s}:= \{ (x,v) \in \mathbb T_{P}\times \mathbb R: \mathcal E_+(x,v)= \varphi_0(0) \}.
\ene
The sets $\Omega_{\pm,\rm s}$ are the  separatrices  between trapped and free orbits. 
Remark that,
\beq\label{1.104}
\mathbb T_{P}\times \mathbb R= \Omega_{-,0}\cup \Omega_{-,1}^+\cup\Omega_{-,1}^- \cup \Omega_{-, \rm s}\cup (0, 0),
\ene
and
\beq\label{1.105}
\mathbb T_{P}\times \mathbb R=  \Omega_{+,0}^- \cup \Omega_{+,0}^+\cup \Omega_{+,1}^+\cup\Omega_{+,1}^-\cup \Omega_{+, \rm s}\cup
[(-P/2,0) \equiv (P/2,0)].
\ene
The point $(0,0)$  is the elliptic point for   the Hamiltonian  $\mathcal E_-(x,v)$ of Newton's equation \eqref{1.73}.  
Further, the points  $(\mp P/2,0)$ are hyperbolic points for  $\mathcal E_-(x,v).$ Similarly, $(\mp P/2,0)$ are elliptic points for the Hamiltonian    $\mathcal E_+(x,v)$ of Newton's equation \eqref{1.83} and $(0,0)$ is the hyperbolic point of $\mathcal E_+(x,v).$ At these points the transformation to energy-angle variables that we introduce below has singularities.

For the energy-angle variables we use the notation $\mu$ for the angle and $\mathcal E_\pm$ for the energy.
We do so for clarity because  the angle variable  $\mu(x,v)$ is defined  in different ways, in the trapped and free regions,  in terms of the angles defined in Subsections~\ref{anelec} and \ref{anions}. Recall that we denote by  $\mathbb S_1$  the unit circle, i.e.,  $[0,1]$ with $0$ and $1$ identified. 

We define the following transformations.
\begin{enumerate}
\item
Let $\mathcal M_{-,0}$ be the following bijection  from $\Omega_{-,0}$  onto  $I_{-,0}\times S_1$
\beq\label{1.106}\begin{array}{l}
\mathcal M_{-,0}(x, v)= (\mu, \mathcal E_- ),   \mu:=\theta_{-,0}(x,\mathcal E_-),       \qquad v \geq 0, \\[5pt]
 M_{-,0}(x, v)= (\mu, \mathcal E_-),\mu:=1- \theta_{-,0}(x,\mathcal E_-), \qquad v < 0,
 \end{array}
 \ene 
 where $\theta_{-,0}(x,\mathcal E_-)$ is defined in \eqref{1.79}. Note that
 \beq\label{1.107}\begin{array}{l}
1-  \theta_{-,0}(x,\mathcal E_-):= \frac{1}{2} +\ds\frac{1}{T_{-,0}(\mathcal E_-)}  \int_{x}^{x_+(\mathcal E_-)}
  \ds \frac{1}{\sqrt{2(\mathcal E_-+\varphi_0(y))}}\, dy, \\\mathcal E_-\in I_{-,0},\qquad
    x_-(\mathcal E_-)\leq x \leq x_+(\mathcal E_-).
    \end{array}
    \ene
    To construct the inverse of $\mathcal M_{-,0}$ let us  proceed as follows. For a fixed $ \mathcal E_-\in I_{-,0},$ let
\beq\label{1.108}
 t \in \mathbb R \to (X_{-,0}(t,\mathcal E_-) ,V_{-,0}(t,\mathcal E_-))
 \ene 
 be the unique solution to Newton's  equation \eqref{1.73} that satisfies the initial condition
\beq\label{1.109}
(X_{-,0}(0,\mathcal E_-),V_{-,0}(0,\mathcal E_-))= (x_-(\mathcal E_-), 0).
\ene
Recall that  $(X_{-,0}(t,\mathcal E_-) ,V_{-,0}(t,\mathcal E_-))$ is periodic  in $t$ with the period $T_{-,0}(\mathcal E_-)$ given in \eqref{1.77}.  For $(\mu, \mathcal E_-) \in  S_1\times I_{-,0}$
the quantity $(x,v)$ is given by,
\beq\label{1.110}\begin{array}{c}
(x(\mu,\mathcal E_-),v(\mu,\mathcal E_-))=\left(\mathcal M_{-,0}\right)^{-1}(\mu, \mathcal E_-):= \\ [.5cm](X_{-,0}(\mu T_{-,0}(\mathcal E_-),\mathcal E_-),V_{-,0}(\mu T_{-,0}(\mathcal E_-), \mathcal E_-)),
\end{array}
\ene
where for $\mu  \in [0,  1/2], v \geq 0,$  for  $\mu  \in [1/2, 1], v \leq 0,$ and $v=0,$ for $\mu= 1/2.$
Furthermore,
\beq\label{1.111}
dx\,  dv= T_{-,0}(\mathcal E_-)\, d\mu\, d\mathcal E_-.
\ene

\item
Let $\mathcal M_{-,1,+}$ be the following bijection  from $\Omega_{-,1}^+$  onto  $ I_{-,1}\times S_1$
\beq\label{1.112}\begin{array}{l}
\mathcal M_{-,1,+}(x, v)= (\mu, \mathcal E_- ), \mu:= \theta_{-,1}(x,\mathcal E_-),
 \end{array}
 \ene
 where $\theta_{-,1}(x,\mathcal E_-)$ is given in \eqref{1.82}.
To construct the inverse of $\mathcal M_{-,1,+}$ let us  proceed as follows. For a fixed $ \mathcal E_-\in I_{-,1},$ let
\beq\label{1.113}
 t \in \mathbb R \to (X_{-,1,+}(t,\mathcal E_-) ,V_{-,1,+}(t,\mathcal E_-))
 \ene 
 be the unique solution to Newton's equation \eqref{1.73} that satisfies the initial condition
\beq\label{1.114}
(X_{-,1,+}(0,\mathcal E_-),V_{-,1,+}(0,\mathcal E_-))= (-P/2, \sqrt{2(\mathcal E_-+\varphi_0(-P/2))}).
\ene
Recall that  $(X_{-,1,+}(t,\mathcal E_-) ,V_{-,1,+}(t,\mathcal E_-))$ travels from $x=-P/2$ to $x=P/2$  in the time  $T_{-,1}(\mathcal E_-)$ given in \eqref{1.81}.  For $(\mu, \mathcal E_-) \in  S_1\times I_{-,1}$
the quantity $(x,v)$ is given by
\beq\label{1.115}\begin{array}{c}
(x(\mu,\mathcal E_-),v(\mu,\mathcal E_-))= \left(\mathcal M_{-,1,+}\right)^{-1}(\mu, \mathcal E_-):=\\[.5cm](X_{-,1,+}(\mu T_{-,1}(\mathcal E_-),\mathcal E_-),V_{-,1,+}(\mu T_{-,1}(\mathcal E_-), \mathcal E_-)).
\end{array}
\ene
Furthermore,
\beq\label{1.116}
dx\,  dv= T_{-,1}(\mathcal E_-)\, d\mu\, d\mathcal E_-.
\ene

\item
Let $\mathcal M_{-,1,-}$ be the following bijection  from $\Omega_{-,1}^-$  onto  $I_{-,1}\times S_1$
\beq\label{1.117}\begin{array}{l}
\mathcal M_{-,1,-}(x, v)= (\mu, \mathcal E_- ), \mu:=1-\theta_{-,1}(x,\mathcal E_-),  
 \end{array}
 \ene
 where $\theta_{-,1}(x,\mathcal E_-)$ is given in \eqref{1.82}.
 Note that
 \beq\label{1.17b}
 1-\theta_{-,1}(x,\mathcal E_-)=\frac{1}{T_{-,1}(\mathcal E_-)} \int_{P/2}^x \frac{1}{- \sqrt{2(\mathcal E_-+\varphi_0(y))}}\, dy.
 \ene
 To construct the inverse of $\mathcal M_{-,1,-}$ let us  proceed as follows. For a fixed $ \mathcal E_-\in I_{-,1},$ let
\beq\label{1.119}
 t \in \mathbb R \to (X_{-,1,-}(t,\mathcal E_-) ,V_{-,1,-}(t,\mathcal E_-))
 \ene 
 be the unique solution to  Newton's equation \eqref{1.73} that satisfies the initial condition
\beq\label{1.120}
(X_{-,1,-}(0,\mathcal E_-),V_{-,1,-}(0,\mathcal E_-))= (P/2 , -\sqrt{2(\mathcal E_-+\varphi_0(P/2))}).
\ene
Recall that  $(X_{-,1,-}(t,\mathcal E_-) ,V_{-1,-}(t,\mathcal E_-))$ travels from $x=P/2$ to $x=-P/2$  in the time  $T_{-,1}(\mathcal E_-)$ given in \eqref{1.81}.  For $(\mu, \mathcal E_-) \in  S_1\times I_{-,1}$
the quantity $(x,v)$ is given by.
\beq\label{1.121}\begin{array}{c}
(x(\mu,\mathcal E_-),v(\mu,\mathcal E_-))= \left(\mathcal M_{-,1,-}\right)^{-1}(\mu, \mathcal E_-):=\\[.5cm](X_{-,1,-}(\mu T_{-,1}(\mathcal E_-),\mathcal E_-),V_{-,1,-}(\mu T_{-,1}(\mathcal E_-), \mathcal E_-)).
\end{array}
\ene
Furthermore,
\beq\label{1.122}
dx\,  dv= T_{-,1}(\mathcal E_-)\, d\mu\, d\mathcal E_-.
\ene
\item
Let $\mathcal M_{+,0}$ be the following bijection  from $\Omega_{+,0}^-\cup \Omega_{+,0}^+$  onto  $I_{+,0}\times S_1$
\beq\label{1.123}\begin{array}{l}
\mathcal M_{+,0}(x, v)= (\mu, \mathcal E_+ ), \mu:=\frac{1}{4}-\theta_{+,0,-}(x,\mathcal E_+),\qquad  (x,v) \in \Omega_{+,0}^-, v \leq 0, \\[5pt]
\mathcal  M_{+,0}(x, v)= (\mu, \mathcal E_+), \mu:=\frac{1}{4}+\theta_{+,0,+}(x,\mathcal E_+),(x,v) \in \Omega_{+,0}^+, \qquad v \leq  0,\\[5pt]
\mathcal  M_{+,0}(x, v)= (\mu, \mathcal E_+), \mu:=\frac{1}{2}+\frac{1}{4}-\theta_{+,0,+}(x,\mathcal E_+),(x,v) \in \Omega_{+,0}^+, \qquad v \geq 0, \\[5pt]
  \mathcal  M_{+,0}(x, v)= (\mu, \mathcal E_+), \mu:=\frac{3}{4}+\theta_{+,0,-}(x,\mathcal E_+),(x,v) \in \Omega_{+,0}^-, \qquad v \geq 0,
   \end{array}
   \ene
   where $\theta_{+,0,-}(x,\mathcal E_+)$ is defined in \eqref{1.89} and $\theta_{+,0,+}(x,\mathcal E_+)$
in \eqref{1.91}.  Note that
 \beq\label{1.124}
 \frac{1}{4}-\theta_{+,0,-}(x,\mathcal E_+)=\frac{1}{T_{+,0}(\mathcal E_+)} \int_{x_-(\mathcal E_+)}^x \frac{1}{-\sqrt{2(\mathcal E_+-\varphi_0(y))}}\, dy
 \ene
 and,
 \beq\label{1.125}
 \frac{1}{4}-\theta_{+,0,+}(x,\mathcal E_+)=\frac{1}{T_{+,0}(\mathcal E_+)} \int_{x_+(\mathcal E_+)}^x \frac{1}{\sqrt{2(\mathcal E_+-\varphi_0(y))}}\, dy.
 \ene
 To construct the inverse of $\mathcal M_{+,0}$ let us  proceed as follows. For a fixed $ \mathcal E_+\in I_{+,0},$ let
\beq\label{1.126}
 t \in \mathbb R \to (X_{+,0,-}(t,\mathcal E_+) ,V_{+,0,-}(t,\mathcal E_+))
 \ene
 be the unique solution to Newton's equation \eqref{1.83} that satisfies the initial condition
\beq\label{1.127}
(X_{+,0,-}(0,\mathcal E_+),V_{+,0,-}(0,\mathcal E_+))= (-P/2,  \sqrt{2(\mathcal E_+-\varphi_0(-P)/2)}).
\ene
Recall that  $(X_{+,0,-}(t,\mathcal E_+) ,V_{+,0,-}(t,\mathcal E_+))$ travels
from $ x=-P/2$  to 
$x_-(\mathcal E_+),$   
and back to $-P/2$, in the time  $T_{+,0}(\mathcal E_+)/2$ with $T_{+,0}(\mathcal E_+)$ given in \eqref{1.88}.  Further, for a fixed $ \mathcal E_+\in I_{+,0},$ let
\beq\label{1.127b}
 t \in \mathbb R \to (X_{+,0,+}(t,\mathcal E_+) ,V_{+,0,+}(t,\mathcal E_+))
 \ene
 be the unique solution to Newton's equation \eqref{1.83} that satisfies the initial condition
\beq\label{1.132}
(X_{+,0,+}(0,\mathcal E_+),V_{+,0,+}(0,\mathcal E_+))= (P/2, - \sqrt{2(\mathcal E_+-\varphi_0(P_{\delta/2}))}).
\ene
 Recall that  $(X_{+,0,+}(t,\mathcal E_+) ,V_{+,0,+}(t,\mathcal E_+))$ travels from $x=P/2$ to $x_+(\mathcal E_+),$   and back to $P/2$, in the time  $T_{+,0}(\mathcal E_+)/2$  with $T_{+,0}(\mathcal E_+)$  given in \eqref{1.88}.  
 For $(\mu, \mathcal E_+) \in  S_1\times I_{+,0}$
the quantity $(x,v)$ is given by,
\beq\label{1.132b}\begin{array}{c}
(x(\mu,\mathcal E_+),v(\mu,\mathcal E_+))= \left(\mathcal M_{+,0}\right)^{-1}(\mu, \mathcal E_+):=\\[2pt](X_{+,0,-}(\frac{1}{4}+\mu) T_{+,0}(\mathcal E_+),\mathcal E_+),V_{+,0,-}(\frac{1}{4}+\mu) T_{+,0}(\mathcal E_+), \mathcal E_+)), \qquad 0\leq \mu \leq \frac{1}{4},\\[2pt]
(x(\mu,\mathcal E_+),v(\mu,\mathcal E_+))= \left(\mathcal M_{+,0}\right)^{-1}(\mu, \mathcal E_+):=\\[2pt](X_{+,0,+}((\mu-\frac{1}{4}) T_{+,0}(\mathcal E_+),\mathcal E_+),V((\mu-\frac{1}{4}) T_{+,0}(\mathcal E_+), \mathcal E_+)), \qquad \frac{1}{4} \leq \mu \leq \frac{3}{4},\\[2pt]
(x(\mu,\mathcal E_+),v(\mu,\mathcal E_+))= \left(\mathcal M_{+,0}\right)^{-1}(\mu, \mathcal E_+):=\\[2pt](X_{+,0,-}((\mu- \frac{3}{4}) T_{+,0}(\mathcal E_+),\mathcal E_+),V_{+,0,-}((\mu-\frac{3}{4}) T_{+,0}(\mathcal E_+), \mathcal E_+)), \qquad \frac{3}{4} \leq \mu \leq 1,
\end{array}
\ene
where for $\mu  \in [0,  1/2], v \leq 0,$   and  for $ \mu \in [1/2, 1], v \geq 0.$  Recall that we have identified $- P/2$ with $ P/2.$ 
Furthermore,
\beq\label{1.128}
dx\,  dv= T_{+,0}(\mathcal E_+)\, d\mu\, d\mathcal E_+.
\ene

\item
Let $\mathcal M_{+,1,+}$ be the following bijection  from $\Omega_{+,1}^+$  onto  $I_{+,1}\times S_1$
\beq\label{1.135}\begin{array}{l}
\mathcal M_{+,1,+}(x, v)= (\mu, \mathcal E_+ ), \mu:= \theta_{+,1}(x,\mathcal E_+),
 \end{array}
 \ene
 where $\theta_{+,1}(x,\mathcal E_+)$ is given in \eqref{1.94}.
To construct the inverse of $\mathcal M_{+,1,+}$ let us  proceed as follows. For a fixed $ \mathcal E_+\in I_{+,1},$ let
\beq\label{1.136}
 t \in \mathbb R \to (X_{+,1,+}(t,\mathcal E_+) ,V_{+,1,+}(t,\mathcal E_+))
 \ene 
 be the unique solution to Newton's equation \eqref{1.83} that satisfies the initial condition
\beq\label{1.137}
(X_{+,1,+}(0,\mathcal E_+),V_{+,1,+}(0,\mathcal E_+))= (-P/2, \sqrt{2(\mathcal E_+-\varphi_0(-P/2))}).
\ene
Recall that  $(X_{+,1,+}(t,\mathcal E_+) ,V_{+,1,+}(t,\mathcal E_+))$ travels from $x=-P/2$ to $x=P/2$  in the time  $T_{+,1}(\mathcal E_+)$ given in \eqref{1.93}.  For $(\mu, \mathcal E_+) \in  S_1\times I_{+,1}$
the quantity $(x,v)$ is given by
\beq\label{1.138}\begin{array}{c}
(x(\mu,\mathcal E_+),v(\mu,\mathcal E_+))= \left(\mathcal M_{+,1,+}\right)^{-1}(\mu, \mathcal E_-):=\\[.5cm](X_{+,1,+}(\mu T_{+,1}(\mathcal E_+),\mathcal E_+),V_{+,1,+}(\mu T_{+,1}(\mathcal E_+), \mathcal E_+)).
\end{array}
\ene
Furthermore,
\beq\label{1.139}
dx\,  dv= T_{+,1}(\mathcal E_+)\, d\mu\, d\mathcal E_+.
\ene

\item
Let $\mathcal M_{+,1,-}$ be the following bijection  from $\Omega_{+,1}^-$  onto  $ I_{+,1}\times S_1$
\beq\label{1.140}\begin{array}{l}
\mathcal M_{+,1,-}(x, v)= (\mu, \mathcal E_+ ), \mu:=1-\theta_{+,1}(x,\mathcal E_+),  
 \end{array}
 \ene
 where $\theta_{+,1}(x,\mathcal E_+)$ is given in \eqref{1.94}.
 Note that
 \beq\label{1.141}
 1-\theta_{+,1}(x,\mathcal E_+)=\frac{1}{T_{+,1}(\mathcal E_+)} \int_{P/2}^x \frac{1}{- \sqrt{2(\mathcal E_+-\varphi_0(y))}}\, dy.
 \ene
 To construct the inverse of $\mathcal M_{+,1,-}$ let us  proceed as follows. For a fixed $ \mathcal E_+\in I_{+,1},$ let
\beq\label{1.142}
 t \in \mathbb R \to (X_{+,1,-}(t,\mathcal E_+) ,V_{+,1,-}(t,\mathcal E_+))
 \ene 
 be the unique solution to Newton's equation \eqref{1.83} that satisfies the initial condition
\beq\label{1.143}
(X_{+,1,-}(0,\mathcal E_+),V_{+,1,-}(0,\mathcal E_+))= (P/2, -\sqrt{2(\mathcal E_+-\varphi_0(P/2))}).
\ene
Recall that  $(X_{+,1,-}(t,\mathcal E_+) ,V_{+,1,-}(t,\mathcal E_+))$ travels from $x=P/2$ to $x=- P/2$  in the time  $T_{+,1}(\mathcal E_+)$ given in \eqref{1.93}.  For $(\mu, \mathcal E_+) \in  S_1\times I_{+,1}$
the quantity $(x,v)$ is given by
\beq\label{1.144}\begin{array}{c}
(x(\mu,\mathcal E_+),v(\mu,\mathcal E_+))= \left(\mathcal M_{+,1,-}\right)^{-1}(\mu, \mathcal E_-):=\\[.5cm](X_{+,1,-}(\mu T_{+,1}(\mathcal E_+),\mathcal E_+),V_{+,1,-}(\mu T_{+,1}(\mathcal E_+), \mathcal E_+)).
\end{array}
\ene
Furthermore,
\beq\label{1.145}
dx\,  dv= T_{+,1}(\mathcal E_+)\, d\mu\, d\mathcal E_+.
\ene

\end{enumerate}

\section{The unperturbed Vlasov-Amp\` ere operator}\label{unp}
 In this section we always assume that items (b) and (c) of Assumption\ref{assum0} hold.
We obtain a selfadjoint realization of the formal unperturbed Vlasov-Amp\`ere operator defined in \eqref{0.34}, and   we construct a spectral representation of the selfadjoint realization  in terms of {\it trace maps} that, as we mentioned in the introduction, plays a crucial role in the proof of the surjectivity of the generalized Fourier maps that we give onTheorem~\ref{theogf}.  For this purpose, first we introduce a unitary transformation from the Hilbert space $\mathcal H$ defined in \eqref{0.36}, \eqref{0.37} onto a Hilbert space $\hat{\mathcal H}$ that we define in terms of the energy-angle variables.
We define
\beq\label{5.1}
\hat{\mathcal H}:= \hat{\mathcal H}_-\oplus \hat{\mathcal H}_+\oplus L^2_0((-P/2,P/2)),
\ene
where,
\beq\label{5.2}\begin{array}{l}
\hat{\mathcal H}_\mp:=L^2(I_{\mp,0}\times (0,1), T_{\mp,0}(\mathcal E_\mp) \,d\mathcal E_\mp\, d\mu  )\oplus
L^2(I_{\mp,1}\times (0,1), T_{\mp,1}(\mathcal E_\mp)\, d\mathcal E_\mp\, d\mu  )\oplus\\ [0.2cm] L^2(I_{\mp,1}\times (0,1), T_{\mp,1}(\mathcal E_\mp)\, d\mathcal E_\mp \,d\mu  ).
\end{array}
\ene
Remark that,
\beq\label{5.3}\begin{array}{l}
\hat{\mathcal H}_\mp:=L^2(I_{\mp,0},L^2(S_1);T_{\mp,0}(\mathcal E_\mp) \,d\mathcal E_\mp )\oplus
L^2(I_{\mp,1}, L^2(S_1); T_{\mp,1}(\mathcal E_\mp)\, d\mathcal E_\mp  )\\ [0.2cm]\oplus L^2(I_{\mp,1}, L^2(S_1); T_{\mp,1}(\mathcal E_\mp)\, d\mathcal E_\mp ).
\end{array}
\ene
We denote the vectors in $\hat{\mathcal H}$ as follows,
\beq\label{5.4}
G=\begin{pmatrix} G_-(\mu, \mathcal E_-)\\G_+(\mu,\mathcal E_+)\\F(x)
\end{pmatrix},\ene
where
\beq\label{5.5}
G_\mp(\mu,\mathcal E_\mp)=\begin{pmatrix} g_{\mp,0}(\mu, \mathcal E_\mp)\\g_{\mp,1,+}(\mu, \mathcal E_\mp)\\g_{\mp,1,-}(\mu, \mathcal E_\mp)
\end{pmatrix}.
\ene
We define the  unitary operator $\mathbf V$  from $\mathcal H$ onto $\hat{\mathcal H},$
\beq\label{5.7}
G= \mathbf  V U:= \begin{pmatrix}  (\mathbf V u_-)(\mu,\mathcal E_-)\\  (\mathbf V u_+)(\mu,\mathcal E_+)\\F(x)
\end{pmatrix},\ene
where
\beq\label{5.8}
G_\mp(\mu,\mathcal E_\mp)=   (\mathbf V u_\mp)(\mu,\mathcal E_\mp):=   \begin{pmatrix} u_\mp\left(  (\mathcal M_{\mp,0})^{-1}(\mu, \mathcal E_\mp)\right)
\\ u_\mp\left(  (\mathcal M_{\mp,1,+})^{-1}(\mu, \mathcal E_\mp)\right)\\  u_\mp\left(  (\mathcal M_{\mp,1,-})^{-1}(\mu, \mathcal E_\mp)\right) \end{pmatrix}.
\ene  
We denote
\beq\label{5.9}
\Omega_-:=  \Omega_{-,0}\cup \Omega_{-,1}^+\cup \Omega_{-,1}^-,
\ene
and
\beq\label{5.9b}
\Omega_+:=  \Omega_{+,0}^-\cup  \Omega_{+,0}^+\cup \Omega_{+,1}^+\cup  \Omega_{+,1}^-.
\ene

It follows from a direct computation that for $U \in \mathcal H$ with $ u_\mp \in C^1_0(\Omega_\mp),$
\beq\label{5.10}
\mathbf V \mathcal A_0 U= {\mathcal A}_{0}^{(0)} \mathbf V U,
\ene
where
\beq\label{5.11}
{\mathcal A}_{0}^{(0)}:= \begin{bmatrix} {\mathcal A}_{0,-}^{(0)}&0&0\\
0& {\mathcal A}_{0,+}^{(0)}&0\\
0&0&0
\end{bmatrix},
\ene
with 
\beq\label{5.12}
{\mathcal A}_{0,\mp}^{(0)}:= \begin{bmatrix}\ds \frac{-i\partial_\mu}{T_{\mp,0}(\mathcal E_\mp)}&0&0\\
 0&\ds\frac{-i\partial_\mu}{T_{\mp,1}(\mathcal E_{\mp})} &0\\
 0&0&\ds\frac{-i\partial_\mu}{T_{\mp,1}(\mathcal E_\mp)}
 \end{bmatrix}.
 \ene
We  denote by  ${a}_0$ the   selfadjoint operator in $L^2(S_1)$ that  acts as
\beq\label{5.14}
a_0 f(\mu):= -i \frac{d}{d_\mu} f(\mu),
\ene
with the domain,
\beq\label{5.15}
D[a_0]:= H^{1,\rm p}((0,1)).
\ene
We recall that we parametrize  $S_1$ with the coordinate $ \mu \in [0, 1],$ where we identify $\mu=0$ with $\mu =1.$ The operator $a_0$ is a selfadjoint extension  of the  symmetric differential operator  $-i \partial_\mu$  with domain 
$\{ f\in C^1([0,1]): f(0)=f(1)=0\}.$  We  denote by $\hat{\mathcal A}_0$ the selfadjoint extension of the formal differential operator $\mathcal A^{(0)}_0$   that acts as
\beq\label{5.16}
\hat{\mathcal A}_0G:= \begin{bmatrix} \hat{\mathcal A}_{0,-}&0&0\\
0& \hat{\mathcal A}_{0,+}&0\\
0&0&0
\end{bmatrix}G
\ene
where
\beq\label{5.17}
\hat{\mathcal A}_{0,\mp}:= \ds\begin{bmatrix} \ds\frac{a_0}{T_{\mp,0}(\mathcal E_\mp)}&0&0\\
 0&\ds\frac{a_0}{T_{\mp,1}(\mathcal E_\mp)} &0\\
 0&0&\ds\frac{a_0}{T_{\mp ,1}(\mathcal E_\mp)}
 \end{bmatrix}.
 \ene
The domain of $\hat{\mathcal A}_0$ is given by, 
\beq\label{5.19}
D(\hat{A}_0):=\{  G\in \hat{\mathcal H}: G_-\in D[\hat{\mathcal A}_{0,-}],  G_+\in D[\hat{\mathcal A}_{0,+}], F\in L^2_0((0,1)) \},
\ene
where,
\beq\label{5.20}\begin{array}{l}
D[\hat{\mathcal A}_{0,\mp}]:=\left\{ \right. G_\mp \in \hat{\mathcal H}_\mp: g_{\mp,0}(\cdot, \mathcal E_\mp) \in H^{1, \rm p}((0,1)), \, {\rm  a.e.}\, \mathcal E_\mp\in I_{\mp,0};  g_{\mp,1,+}(\cdot, \mathcal E_\mp) \in  H^{1, \rm p}((0,1)),\, \\[0.1cm]{ \rm a.e.} \, \mathcal E_\mp\in I_{\mp,1};  g_{\mp,1,-}(\cdot, \mathcal E_\mp) \in  H^{1, \rm p}((0,1)),\, {\rm  a.e.}\, \mathcal E_\mp\in I_{\mp,1}, \, {\rm and} 
\\[0.1cm]
\int_{ I_{\mp,0}} \, \|g_{\mp,0}(\cdot, \mathcal E_\mp)\|_{ H^{1, \rm p}((0,1))}^2  {T^{-1}_{\mp,0}(\mathcal E_\mp)}\, d\mathcal E_\mp  < \infty,\\[.1cm]
\int_{ I_{\mp,1}} \,  \| g_{\mp,1,+}(\cdot, \mathcal E_\mp)\|_{ H^{1, \rm p}((0,1))}^2   {T^{-1}_{\mp ,1}(\mathcal E_\mp)}\, d\mathcal E_\mp  < \infty,\\[0.1cm]
\int_{ I_{\mp,1}} \, \|   g_{\mp,1,-}(\cdot, \mathcal E_\mp)\|_{ H^{1, \rm p}((0,1))}^2  {T^{-1}_{\mp,1}(\mathcal E_\mp)}\, d\mathcal E_\mp< \infty \left\} \right. .
\end{array}
\ene
We define a selfadjoint extension of the unperturbed  Vlasov-Amp\` ere operator \eqref{0.34}, that we also denote by $\mathcal A_0,$ by a unitary transformation of $\hat{\mathcal A}_0$,
\beq\label{5.22}
\mathcal A_0:= \mathbf V^{-1} \hat{\mathcal A}_0 \mathbf V.
\ene   
In order to characterize the domain of $\mathcal A_0$  we define the action of the operators
$$
\mathcal D_\pm:=-i( v \partial_x\pm \partial_x \varphi_0 \partial_v)
$$ 
in weak sense.
\begin{definition}  {\rm We say that  $f \in L^2( \mathbb T_{P}\times \mathbb R)$ belongs to $D[\mathcal D_\pm ],$  if there is a function $g_\pm   \in L^2(\mathbb T_{P}\times \mathbb R)$ such that,
\beq\label{5.23}\begin{array}{c}
\ds \left (f, \mathcal D_\pm u\right)_{\ds L^2(\mathbb T_{P} \times R)}=\ds \left(g_\pm,u\right)_{\ds L^2(\mathbb T_{P}\times \mathbb R)}, \qquad u \in  C^1_0(\Omega_\pm)\end{array}.
\ene
Further,
\beq\label{5.25}
\mathcal D_\pm f:=g_\pm.
\ene
}
\end{definition}
In the following theorem we characterize the domain of $\mathcal A_0,$  the kernel of
 $\hat{\mathcal A_0},$ and the kernel of $\mathcal A_0.$

\begin{theorem}\label{thoesefad}
Suppose that items (b) and (c) of Assumption~\ref{assum0} hold. Then.
\begin{enumerate}
\item[\rm (a)]The unperturbed Vlasov-Amp\`ere operator $\mathcal A_0$ defined in \eqref{5.22} is selfadjoint and its domain is given by,
\beq\label{5.26}
D[\mathcal A_0]=\{ U \in \mathcal H: U= (\mathbf V)^{-1} G, G \in D[\hat{\mathcal A}_0]\}.
\ene
\item[\rm (b)]
\beq\label{5.27}
\mathcal A_0 \,U=  \mathcal DU, \, \hbox{\rm where}\, \mathcal D:= \begin{bmatrix} \mathcal D_-&0&0\\
 0&  \mathcal D_+&0\\ 0&0&0\end{bmatrix}, \qquad U \in D[\mathcal A_0].
 \ene
 \item[\rm (c)]
 \beq\label{5.28}\hbox{\rm Ker}[\hat{\mathcal A}_0]= \{ G \in D[\hat{\mathcal A}_0]: G_\mp \, \hbox{\rm are independent of} \, \mu \}.
 \ene
 \item[\rm(d)]
 \beq\label{5.29}
 \hbox{\rm Ker}[{\mathcal A}_0]= \mathbf V^{-1} \hbox{\rm Ker}[\hat{\mathcal A}_0].
 \ene
 \end{enumerate}
\end{theorem}

\begin{proof}Item (a) follows from \eqref{5.22}. Let us prove (b). Suppose that $ U \in D[\mathcal A_0].$ Hence, by \eqref{5.26} $U= \mathbf V^{-1} G$ for some
$G\in D[\hat{\mathcal A}_0].$ Then, for all $ L=(l_-, l_+,0)$ with $l_\mp \in C^1_0(\Omega_\mp),$ by \eqref{5.10},
\beq\label{5.29b}
\left( U, \mathcal D L \right)_{\mathcal H}= \left( G, \hat{\mathcal A}_0 \mathbf V L \right)_{\hat{\mathcal H}}=
\left( \hat{\mathcal A}_0 G,  \mathbf V L\right)_{\hat{\mathcal H}}= \left(  \mathcal A_0 U, L\right)_{\mathcal H}.
\ene
From \eqref{5.29b} it follows that  $U\in D[\mathcal D]$ and $ \mathcal A_0 U= \mathcal D U.$ This proves (b). Item (c) follows from \eqref{5.14}, \eqref{5.16}, and \eqref{5.17}. Finally, (d) follows from (c) and  \eqref{5.22}.

\end{proof}

We denote by $\mathbf  F_1$ the Fourier series  for periodic functions defined on $\mathbb S_1,$
\beq\label{5.82}
\mathbf  F_1 f= \{  f_n \}_{n \in \mathbb Z},  
\ene
where
\beq\label{5.83}
f_n= (\mathbf  F_1f)_n := \frac{1}{\sqrt{2\pi}} \int_{[0,1]}\, e^{-i2\pi n \mu}\, f(\mu)\, d\mu.
\ene
As is well known, $\mathbf  F_1$ is unitary from $L^2(\mathbb S_1)$ onto $l^2,$ the standard Hilbert
 of all square summable complex valued sequences $\{f_n\}_{ n \in  \mathcal Z}$ . The inverse is given by,
  \beq\label{5.84}
 \mathbf  F_1^{-1} \{f_n\}= \frac{1}{\sqrt{2\pi}}\sum_{n \in \mathbb Z} e^{i2\pi n x} f_n, \qquad \{f_n\}\in l^2.
 \ene
 Let $ \mathcal H_{\mathbf  F}$ be the following Hilbert space,
 \beq\label{5.84b}\begin{array}{l}
 \mathcal H_{\mathbf  F}:= L^2(I_{-,0}, l^2; T_{-,0}(\mathcal E_-)  d\mathcal E_-)\oplus L^2(I_{-,1}, l^2; T_{-,1}(\mathcal E_-) d\mathcal E_-)\oplus \\L^2(I_{-,1}, l^2;T_{-,1}(\mathcal E_-) d\mathcal E_-)\oplus L^2(I_{+,0}, l^2; T_{+,0}(\mathcal E_+) d\mathcal E_+)\oplus L^2(I_{+,1}, l^2; T_{+,1}(\mathcal E_+) d\mathcal E_+)\oplus \\L^2(I_{+,1}, l^2;  T_{+,1}(\mathcal E_+)d\mathcal E_+)\oplus L^2_0(-P/2, P/2).
 \end{array}
\ene
We designate the functions  $ B \in \mathcal H_{\mathbf  F}$ as follows,
\beq\label{5.84c}
 B=\begin{pmatrix}B_-\\ B_+\\ F\end{pmatrix}
 \ene
 where,
 \beq\label{5.84d}
 B_\mp= \begin{pmatrix}\{ b_{\mp,0,n}\}\\\{b_{\mp,1,+,n}\}\\ \{b_{\mp,1,-,n}\}
 \end{pmatrix}.
 \ene
 We denote by $\mathbf F$ the following unitary operator from $\hat{\mathcal H}$ onto $\mathcal H_{\rm F},$
 \beq\label{5.84f}
 \mathbf F G=\begin{pmatrix} (\mathbf F G)_-\\(\mathbf F G)_+\\ F\end{pmatrix},
 \ene
 where
 \beq\label{5.84h}
 (\mathbf F G)_\mp:= \begin{pmatrix}\mathbf  F_1 g_{\mp,0}\\ \mathbf  F_1 g_{\mp,1, +}\\ \mathbf  F_1 g_{\mp,1,-}
 \end{pmatrix}.
 \ene
Under $\mathbf F$ the operator $ \hat{\mathcal A}_0$ transforms as follows,
\beq\label{5.84i}
\hat{\mathcal A}_{0, {\mathbf  F}}= \mathbf F \hat{\mathcal A}_{0} \mathbf F^{-1},
\ene
where,
\beq\label{5.84j}
\hat{\mathcal A}_{0,{\mathbf  F}}:= \begin{bmatrix} \hat{\mathcal A}_{0,{\mathbf  F},-}&0&0\\
0& \hat{\mathcal A}_{0,{\mathbf  F},+}&0\\
0&0&0
\end{bmatrix},
\ene
with
$$
\hat{\mathcal A}_{0,{\mathbf  F}, \mp}= \oplus_{n \in \mathbb Z} \hat{\mathcal A}_{0,{\mathbf  F}, \mp,n},
$$
where
\beq\label{5.84k}
\hat{\mathcal A}_{0,{\mathbf  F},\mp,n}:= \begin{bmatrix}\ds \frac{2 \pi n}{T_{\mp,0}(\mathcal E_\mp)}&0&0\\
 0&\ds\frac{2 \pi n}{T_{\mp,1}(\mathcal E_\mp)} &0\\
 0&0&\ds\frac{2\pi n}{T_{\mp,1}(\mathcal E_\mp)}
 \end{bmatrix}.
 \ene
In order to construct a spectral representation for $\mathcal A_0$  we introduce the following notation. We denote by $\mathbb Z_0$ the nonzero integers, by $\mathcal N$ the positive integers,  and we define,
    \beq\label{5.49}
  \beta_{\mp,0,n}(\mathcal E_{\mp}):= \frac{(2\pi n)}{ T_{\mp,0}(\mathcal E_\mp)},\qquad  \mathcal E_{\mp}\in I_{\mp,0}, n \in \mathbb Z_0,
\ene
  \beq\label{5.50}
 \beta_{-,0,n,\rm max}:= \frac{2\pi n}{T_{-,0}(-\varphi_0(0))}, \beta_{+,0,n,\rm max}:= \frac{2\pi n}{T_{+,0}(\varphi_0(P/2))},
 n \in \mathbb N.
 \ene
 We remark that by Proposition~\ref{proper}   (c) and (d) $T_{-,0}(-\varphi_0(0))= 2\pi /\sqrt{-\varphi_0''(0)}$ is the minimum of $T_{-,0}(\mathcal E_-)$ for $\mathcal E_- \in I_{-,0}.$ Further, by  Proposition~\ref{proper}   (h) and (i) $T_{+,0}(\varphi_0(P/2))= 2\pi /\sqrt{\varphi_0''(P/2)}$ is the minimum of $T_{+,0}(\mathcal E_+)$ for $\mathcal E_+ \in I_{+,0}.$
   Since by Proposition~\ref{proper} (d) and (i)  $T_{\mp,0}(\mathcal E_\mp)$ is strictly  increasing, the function $\beta_{\mp,0,n}(\mathcal E_\mp)$ is invertible. 
   We denote  the  inverse function by (see Proposition~\ref{proper} (a) and (f))
   \beq\label{5.51}\begin{array}{l}
   \mathcal E_{\mp,0,n}(\beta), n \in \mathbb Z_0, \, {\rm where} \,\beta \in (0, \beta_{\mp,0,n,\rm max})\, {\rm for} \, n>0, \\ {\rm and } \,\beta \in ( -\beta_{\mp,0,-n,\rm max},0)\,\, {\rm for}\, n <0.
  \end{array}
   \ene
    We have that,
  \beq\label{5.52}
  \mathcal E_{\mp,0,n}(\beta_{\mp,0,n}(\mathcal E_\mp))= \mathcal E_{\mp}, \qquad \mathcal E_\mp\in I_{\mp,0},n \in \mathbb Z_0. 
  \ene
    By the inverse function theorem,
   \beq\label{5.53}
 p_{\mp,0,n}(\beta):= \mathcal E_{\mp,0,n}'(\beta)= -\frac{T_{\mp,0}^2(\mathcal E_{\mp,0,n}(\beta))}{2 \pi n T'(\mathcal E_{\mp,0,n}(\beta))}.
  \ene
Moreover, we designate 
\beq\label{5.54}
  \beta_{\mp,1,n}(\mathcal E_\mp):= \frac{(2\pi n)}{ T_{\mp,1}(\mathcal E_\mp)},\qquad   \mathcal E_{\mp}\in I_{\mp,1}, n \in \mathbb Z_0.
\ene
By Proposition ~\ref{proper} (d) and (j), $ \beta_{\mp,1,n}(\mathcal E_\mp)$ is invertible.    We denote  the  inverse function by (see Proposition~\ref{proper} (b) and (g)),

 \beq\label{5.55}
   \mathcal E_{\mp,1,n}(\beta), n \in \mathbb Z_0, \, {\rm where} \, \beta \in (0, \infty)\, {\rm for} \, n>0, \,{\rm and}  \,\beta \in ( -\infty,0)\, {\rm for}\, n <0.
   \ene
 We have that,
 \beq\label{5.56}
  \mathcal E_{\mp,1,n}(\beta_{\mp,1,n}(\mathcal E_\mp))= \mathcal E_{\mp}, \qquad \mathcal E_\mp \in I_{\mp,1}, n \in \mathbb Z_0.
  \ene
  Moreover,
   \beq\label{5.57}
 p_{\mp,1,n}(\beta):= \mathcal E_{\mp,1,n}'(\beta)= -\frac{T_{\mp,1}^2(\mathcal E_{\mp,1,n}(\beta))}{2 \pi n T'(\mathcal E_{\mp,1,n}(\beta))}.
  \ene
We introduce Hilbert spaces that we use to construct  a spectral representation of $\mathcal A_0.$
 \beq\label{5.66}
 \mathcal Q_{\mp,0,0}= L^2(I_{\mp,0},T_{\mp,0}(\mathcal E_\mp)   d\mathcal E_\mp),
 \ene
 \beq\label{5.67}
 Q_{\mp,0,n}:= L^2( (0, \beta_{\mp,n, {\rm max}}), d\beta), \qquad n \in \mathbb N,
 \ene
  \beq\label{5.68}
 Q_{\mp,0,n}:= L^2( (-\beta_{\mp,-n, {\rm max}},0), d\beta) \qquad -n \in \mathbb N,
 \ene
  \beq\label{5.69}
 \mathcal Q_{\mp,1,0}= L^2( I_{\mp,1}, T_{\mp,1}(\mathcal E_\mp) d\mathcal E_\mp),
 \ene
 \beq\label{5.70}
Q_{\mp,1,n}:= L^2( (0,\infty), d\beta), \qquad n \in \mathbb N,
 \ene
  \beq\label{5.71}
 Q_{\mp,1,n}:= L^2( (-\infty,0), d\beta), \qquad   - n \in \mathbb N.
 \ene

 Further, we designate,
 \beq\label{5.78}\begin{array}{c}
\tilde{\mathcal H}:= \oplus_{n \in \mathbb Z}\left [ Q_{-,0,n}\oplus  Q_{-,1,n}\oplus Q_{-1,n} \right] \oplus_{n \in \mathbb Z}\\
 \left [  Q_{+,0,n}\oplus  Q_{+,1,n}\oplus Q_{+,1,n} \right]  \oplus L^2_0((-P/2,P/2)).\end{array}
\ene
 We denote the functions in $\tilde{\mathcal H}$ as follows,
\beq\label{5.79}
\begin{pmatrix} \{K_{-,n}\}\\\{ K_{+,n}\}\\ F\end{pmatrix}
\ene
where
\beq\label{5.80}
\{K_{\mp,n}\}= \left\{\begin{array}{c} k_{\mp,0,n}\\k_{\mp,1,+,n}\\k_{\mp,1,-,n}\end{array}
\right\}.
\ene
Let $\tilde{\mathbf V} $ be the following unitary operator from ${\mathcal H}_{\mathbf F}$ onto $\tilde{\mathcal H}$ given by,
\beq\label{5.81}
   \tilde{\mathbf V}B= \begin{pmatrix} (\tilde{\mathbf V}B)_-\\   (\tilde{\mathbf V}B)_+    \\F
     \end{pmatrix}
     \ene
     where,
    \beq\label{5.86}
     (\tilde{\mathbf V} B)_{\mp,0}:= \begin{pmatrix} b_{\mp,0,0}\\
    b_{\mp,1,+,0}\\ b_{\mp,1,-,0}
   \end{pmatrix},
   \ene
  
  \beq\label{5.87}
     (\tilde{\mathbf V} B)_{\mp,n}:= \begin{pmatrix}   \sqrt{ |p_{\mp,0,n}(\beta)| \, T_{\mp,0}(\mathcal E_{\mp,0,n}(\beta)) }   \,   b_{\mp,0,n}(\mathcal E_{\mp,0,n}(\beta))\\[.5cm]
       \sqrt{ |p_{\mp,1,n}(\beta)|  \,T_{\mp,1}(\mathcal E_{\mp,1,n}(\beta)) } \,    b_{\mp,1,+,n}(\mathcal E_{\mp,1,n}(\beta))\\[.5cm]  \sqrt{ |p_{\mp,1,n}(\beta)|  \,T_{\mp,1}(\mathcal E_{\mp,1,n}(\beta)) }   \,b_{\mp,1,-,n}(\mathcal E_{\mp,1,n}(\beta))
   \end{pmatrix}, \qquad  n \in \mathbb Z_0.
   \ene
  We define,
  \beq\label{5.89}
  \mathcal F:= \hat{\mathbf V} \mathbf F.
  \ene
  The operator $\mathcal F$ is unitary from $\hat{\mathcal H}$ onto $ \tilde{\mathcal H}.$ 
 Further, under ${\mathcal F}$ the operator $\hat{\mathcal A}_0$ transform as follows
   \beq\label{5.90}
   \tilde{\mathcal A}_0:=  {\mathcal F} \hat{\mathcal A}_0 {\mathcal F}^{-1}.
   \ene
   For  $K=(\{K_{-,n}\},  \{K_{+,n}\},  F)^T\in D[\tilde{\mathcal A}_0]$ and   $ \tilde{\mathcal A}_0 K=( \tilde{\mathcal A}_0 K)_{-,n},  (\tilde{\mathcal A}_0 K)_{+,n}, H)^T$  we have,

   \beq\label{5.91}\begin{array}{c}
   ( \tilde{\mathcal A}_0 K)_{\mp,0,0}=
   ( \tilde{\mathcal A}_0 K)_{\mp,1,+,0}=
   ( \tilde{\mathcal A}_0 K)_{\mp,1,-,0}=0,\\
    ( \tilde{\mathcal A}_0 K)_{\mp,0,n}= \beta\, k_{\mp,0,n},  \qquad n \in \mathcal Z_0,  \\
( \tilde{\mathcal A}_0 K)_{\mp,1,+,n}= \beta\, k_{\mp,1,+,n},  \qquad n \in \mathcal Z_0,\\
 (\tilde{\mathcal A}_0 K)_{\mp,1,-,n}= \beta\, k_{\mp,1,-,n},  \qquad n \in \mathcal Z_0,\\
 H\equiv 0.
\end{array}
\ene
For an operator $B$ in a Hilbert space we denote by $\sigma(B)$ its spectrum \cite{kato}. If $B$ is selfadjoint we 
designate  $\sigma_{\rm ac}(B)$ its absolutely continuous spectrum \cite{kato} and  by $\sigma_{\rm ess}({B})$ its essential spectrum, namely the complement in the spectrum of $B$ of the set of  the isolated eigenvalues of $B$ with finite multiplicity.

   In the following theorem we characterize the spectrum of $\mathbf A_0.$
   \begin{theorem}\label{sp0} Suppose that items (b) and (c) of Assumption~\ref{assum0} hold.
   Then, the  point spectrum of $\mathcal A_0,$ $\hat{\mathcal A}_0,$ and $\tilde{\mathcal A}_0,$ consists of the infinite dimensional eigenvalue zero and they have no singular continuous spectrum. Further,
   $\sigma_{\rm ess}(\mathcal A_0)= \sigma_{\rm ac}(\mathcal A_0)= \sigma_{\rm ess}(\hat{\mathcal A}_0)= \sigma_{\rm ac}(\hat{\mathcal A}_0)=\sigma_{\rm ess}(\tilde{\mathcal A}_0)= \sigma_{\rm ac}(\tilde{\mathcal A}_0)=\mathbb R .$ 
     \end{theorem}
  \begin{proof} By the unitary equivalences in \eqref{5.22}, and \eqref{5.90} is is enough to prove the theorem for $\tilde{\mathcal A}_0.$ Further, the theorem follows for  $\tilde{\mathcal A}_0$ by \eqref{5.91}. We recall that an operator of multiplication by a continuous strictly  monotonic function is absolutely continuous.
 \end{proof}   
Note that,
\beq\label{5.92}
\mathcal A_0=  (\mathcal F \mathbf V)^{-1} \tilde{\mathcal A}_0 \,(\mathcal F \mathbf V).
\ene
Then, $ \mathcal F \mathbf V$ gives a spectral representation for $\mathcal A_0.$ 

Observe, furthermore, that for any  Borel set $\Delta \subset \mathbb R \setminus\{0\}$
\beq\label{5.93}
\mathcal F E_{\hat{\mathcal A}_0 }(\Delta) = \chi_\Delta(\beta) \mathcal F P_{\rm ac}(\hat{\mathcal A_0}).
\ene
Where $ P_{\rm ac}(\hat{\mathcal A}_0)$ denotes the projector onto the absolutely continuous subspace of $\hat{\mathcal A}_0,$ that we denote by $\mathcal H_{\rm ac}(\mathcal A_0).$

We denote,
 \beq\label{5.94}\begin{array}{c}
 \tilde{\mathcal H}_0:=\oplus_{n \in \mathbb Z_0}\left [ Q_{-,0,n}\oplus  Q_{-,1,n}\oplus Q_{-,1,n} \right] \oplus_{n \in \mathbb Z_0}
 \left [  Q_{+,0,n}\oplus  Q_{+,1,n}\oplus Q_{+,1,n} \right].  \end{array}
\ene
Note that $ \tilde{\mathcal H}_0$ is isomorphic to $P_{\rm ac}(\tilde{\mathcal A}_0) \tilde{\mathcal H}.$
 
 We prepare the following proposition. 
\begin{proposition} \label{propgf} Suppose that items (b) and (c)of Assumption~\ref{assum0} hold.
Let us denote, $\mathcal F_0:= \mathcal F P_{\rm ac}(\hat{\mathcal A}_0).$
We have that  $\mathcal F_0$ is partially isometric with initial subspace   $\hat{\mathcal H}_{\rm ac}(\hat{\mathcal A}_0)$  and final subspace  $\tilde{\mathcal H}_0.$
\end{proposition}
\begin{proof} By Theorem~\ref{sp0}  $\hat{\mathcal H}_{\rm ac}(\hat{\mathcal A}_0)$ is the orthogonal complement
of the kernel of $\hat{\mathcal A}_0.$ Then, the proposition follows from \eqref{5.90}, \eqref{5.91}, and since $\mathcal F$ is unitary from $\hat{\mathcal H}$ onto $\tilde{\mathcal H}.$   
\end{proof}
We prepare the following results that we  use to write our spectral representation  \eqref{5.93} in terms of {\it trace maps}. 
 We denote by $l^2_{j},$  for $j=1,2,\dots$ the  Hilbert space of all the complex valued square summable sequences $\{ f_n\}, n=j,j+1, j+2, \dots .$ Further, for $j=-1,-2-\dots$  we designate by $l^2_{j},$ the  Hilbert space of all the complex valued square summable sequences $\{ f_n\}, n=j,j-1, j-2, \dots.$ Moreover, we designate, 
\beq\label{5.95}
\tilde{\mathcal H}_{\mp,0, j}= L^2((\beta_{\mp,  j-1 ,{\rm max}},  \beta_{\mp, j  ,{\rm max} } ), l^2_j),\qquad j \in \mathbb N,
\ene
\beq\label{5.96}
\tilde{\mathcal H}_{\mp,0,j}= L^2((-\beta_{\mp, -j,{\rm max} }, -\beta_{\mp, -j-1,{\rm max} }), l^2_j  ),\qquad     - j \in \mathbb N,
\ene
where we denote
$$
\beta_{\mp, 0,{\rm max}}:=0. 
$$
Further, we designate,
\beq\label{5.97}
\tilde{\mathcal H}_{ \mp,1}= L^2((0, \infty), l^2_1\oplus l^2_1),
\ene
and 
\beq\label{5.98}
\tilde{\mathcal H}_{ \mp,-1}= L^2((-\infty,0), l^2_{-1}\oplus l^2_{-1}).
\ene
We remark  that $\tilde{\mathcal H}_0$ can be written a follows,
\beq\label{5.99}\begin{array}{c}
\tilde{\mathcal H}_0= \oplus_{j \in \mathbb N} \tilde{\mathcal H}_{-,0, j} \oplus \tilde{\mathcal H}_{-,1}
 \oplus_{j \in \mathbb N}\tilde{\mathcal H}_{+,0, j} \oplus \tilde{\mathcal H}_{+,1}\oplus_{-j \in \mathbb N} 
  \tilde{\mathcal H}_{-,0, j} \oplus\\ \tilde{\mathcal H}_{-,-1}
 \oplus_{-j \in \mathbb N}\tilde{\mathcal H}_{+,0, j} \oplus \tilde{\mathcal H}_{+,-1}.
 \end{array}
 \ene
 Let $\Delta$ be an interval  such that
 \beq\label{5.100}
 \Delta \subset \mathbb R \setminus \ds \{ \{0\} \ds \ds \cup_{\nu=\pm, j \in \mathcal N}\{ \beta_{\nu, j, {\rm max}},  -\beta_{\nu, j, {\rm max}}\}\}.
 \ene
 Then, if $\Delta \subset (0,\infty)$ there are $j_\mp \in \mathbb N$ such that, 
 \beq\label{5.101}
 \Delta \subset  (\beta_{-,  j_--1 ,{\rm max}},  \beta_{-, j_-  ,{\rm max} } )  \cap       (\beta_{+,  j_+-1 ,{\rm max}},  \beta_{+, j_+  ,{\rm max} } ).
 \ene
  Further, if $\Delta \subset (-\infty,0)$ there are $j_\mp \in \mathbb N$  such that, 
 \beq\label{5.102}
 \Delta \subset  (-\beta_{-,  j_- ,{\rm max}},  -\beta_{-, j_- -1  ,{\rm max} } )  \cap   (-\beta_{+,  j_+ ,{\rm max}},  -\beta_{+, j_+-1  ,{\rm max} } ).
 \ene
 
 Then, by \eqref{5.93}, for  $G \in \hat{\mathcal H}_0$ and $\Delta$ as in \eqref{5.100}
\beq\label{5.103}
\left(\mathcal F E_{\hat{\mathcal A}_0 }({\Delta}) G \right)(\beta)= \chi_{{\Delta}}(\beta) \mathcal L_{j_-,j_+}(\beta) 
\mathbf F P_{\rm ac}(\hat{\mathcal A_0}) G,
\ene
where,   the  {\it trace maps}  $\mathcal L_{j_-,j_+}(\beta)$ are defined as follows,
\beq\label{5.104}
\mathcal L_{j_-,j_+}(\beta) B=\begin{pmatrix} (\mathcal L_{j_-,j_+}(\beta) B)_-\\(\mathcal L_{j_-,j_+}(\beta) B)_+
\end{pmatrix},
\ene
where, if $ \Delta \subset (0,\infty),$  
\beq\label{5.105}
  (\mathcal L_{j_-,j_+}(\beta) B)_\mp  := \begin{pmatrix}  \{ \sqrt{ |p_{\mp,0,n_\mp}(\beta)|  T_{\mp,0}(\mathcal E_{\mp,0,n_\mp}(\beta))}\,       b_{\mp,0,n_\mp}(\beta) \}, n_\mp= j_\mp, j_\mp+1, \dots\\[.5cm]
      \{ \sqrt{ |p_{\mp,1,n}(\beta) | T_{\mp,1}(\mathcal E_{\mp,1,n}(\beta)) }   \,  b_{\mp,1,+,n}(\beta)\}, n=1,2,\dots\\[.5cm]  \{\sqrt{ |p_{\mp,1,n}(\beta)|  T_{\mp,1}(\mathcal E_{\mp,1,n}(\beta)) }  \, b_{\mp,1,-,n}(\beta)\}, n=1,2,\dots
   \end{pmatrix},
   \ene
   and if $ \Delta \subset (-\infty,0),$  
\beq\label{5.106}
  (\mathcal L_{j_-,j_+}(\beta) B)_\mp  := \begin{pmatrix}  \{ \sqrt{ |p_{\mp,0,n_\mp}(\beta)|  T_{\mp,0}(\mathcal E_{\mp,0,n_\mp}(\beta))}\,       b_{\mp,0,n_\mp}(\beta) \}, n_\mp= -j_\mp,  -j_\mp-1, \dots\\[.5cm]
      \{ \sqrt{ |p_{\mp,1,n}(\beta)|  T_{\mp,1}(\mathcal E_{\mp,1,n}(\beta)) }   \,  b_{\mp,1,+,n}(\beta)\}, n=-1,-2,\dots\\[.5cm]  \{\sqrt{ |p_{\mp,1,n}(\beta)|  T_{\mp,1}(\mathcal E_{\mp,1,n}(\beta)) }  \, b_{\mp,1,-,n}(\beta)\}, n=-1,-2,\dots
   \end{pmatrix}.
   \ene
We call  $\mathcal L_{j_-,j_+}(\beta)$ {\it trace maps} because they take the trace at a fixed value of the spectral parameter $\beta.$
\section{ The Vlasov-Amp\`ere  operator}\label{vlasamp}
We begin this section defining a selfadjoint realization of the formal Vlasov-Amp\`ere operator \eqref{0.33}- \eqref{0.35}. We find it convenient to work in energy-angle variables. For this purpose we define,
\beq\label{6.1}
\hat{\mathcal A}:= \mathbf V \mathcal A \mathbf V^{-1}= \hat{\mathcal A}_0+ \hat{\mathcal V},
\ene
where
\beq\label{6.2}
\hat{\mathcal V}:=   \mathbf V \mathcal V \mathbf V^{-1}.
\ene
 In the following theorem we prove that  $\mathcal A$ and  $\hat{\mathcal A}$ are selfadjoint, with domain, respectively,
$D[\mathcal A]= D[\mathcal A_0],$ and  $D[\hat{\mathcal A}]= D[\hat{\mathcal A}_0],$ and  that their essential spectrum coincide with  the real axis.  For this purpose  we strengthen  the decay condition that we impose in $h'_\pm(v^2/2).$

\begin{theorem}\label{vasp} Suppose that Assumption~\ref{assum0} is satisfied and that    $h'_\pm(v^2/2) v^2 $  is  a integrable function of $v\in \mathbb R.$ Then.
\begin{enumerate} 
\item[ \rm(a)]The operator ${\mathcal A}= {\mathcal A_0}+ {\mathcal V}$ is selfadjoint with domain $D[{\mathcal A}]= D[{\mathcal A_0}].$ Further, the operator  $\hat{\mathcal A}=   \hat{\mathcal A_0}+ \hat{\mathcal V}$ is selfadjoint with domain $D[\hat{\mathcal A}]= D[\hat{\mathcal A_0}].$
\item[\rm (b)]
\beq\label{6.11}
\sigma_{\rm ess}({\mathcal A})= \sigma_{\rm ess}(\hat{\mathcal A}) =  \sigma({\mathcal A_0})= \sigma(\hat{\mathcal A_0})= \mathbb R.
\ene
\end{enumerate}
\end{theorem}
\begin{proof}  As $\sqrt{|h'_\pm(v^2/2)|}\, v$ are square integrable over $\mathbb R,$  by \eqref{0.35} $\mathcal V$ is a bounded operator on $\mathcal H.$ Then, by the unitary equivalence in \eqref{6.2} $\hat{\mathcal V}$ is bounded on $\hat{\mathcal H}.$ It follows from the Kato-Rellich theorem (see Theorem 4.3 in page 287) of \cite{kato}) that $\mathcal A:= \mathcal A_0+\mathcal V$ is selfadjoint in $\mathcal H$ with domain $D[\mathcal A]= D[\mathcal A_0]$ and that$\hat{\mathcal A}:= \hat{\mathcal A}_0+ \hat{\mathcal V}$ is selfadjoint in $\hat{\mathcal H}$ with domain $D[\hat{\mathcal A}]= D[\hat{\mathcal A}_0].$ This proves (a). By the unitary equivalence in \eqref{6.1} it  is  enough prove (b) for $\mathcal A.$ As by Theorem~\ref{sp0} the essential spectrum of $\mathcal A_0$ is $\mathbb R$ it follows from Weyl's criterium (see Lemma 6.17 of \cite{tes}) that for each $\lambda \in \mathbb R$ there is a singular Weyl sequence for $\mathcal A_0.$ Namely, a sequence $\Psi_n\in D[\mathcal A_0]$ such that $\|\Psi_n\|_{\mathcal H}=1,$  $\Psi_n$ tends weakly to zero in $\mathcal H$ and
$\|(\mathcal A -\lambda)\Psi_n\|_{\mathcal H}\to 0,$  as $n \to \infty.$ Further, by \eqref{5.22}, \eqref{5.90}, and \eqref{5.91} we can take the singular sequence as
$\Psi_n:=(\Psi_{n,-},0,0)^T.$
Moreover, since   $\mathcal A:= \mathcal A_0+\mathcal V,$ by \eqref{0.35} and as  $\sqrt{|h'_-(v^2/2)|}\, v$ is square integrable over $\mathbb R,$    it follows that $\Psi_n$ is also a singular Weyl sequence for $\mathcal A.$ Hence, by  Lemma 6.17 of \cite{tes} we have that
$\lambda \in \sigma_{\rm ess}(\mathcal A).$
 \end{proof}
  In the following proposition we construct explicitly a family of eigenvectors of $\mathcal A$   with eigenvalue zero and we prove that the vectors in $\mathcal H$ that are orthogonal to the eigenvectors of a sub-family of that family satisfy the Gauss law.
\begin{proposition}\label{gauss}
Suppose that Assumption~\ref{assum0} is satisfied and that    $h'_\pm(v^2/2) v^2 $  is  an integrable function of $v\in \mathbb R.$ Then, we have.
\begin{enumerate}
\item[\rm (a)]
For any function $F_0 \in L^2_0(-P/2,P/2)$  the vector

 \beq\label{6.10a}
  U_{F_0}(t,x,v):= \begin{pmatrix} - g(x) \sqrt{|h_-'( \mathcal E_-(x,v))|}  \\   g(x) \sqrt{|h_+'( \mathcal E_+(x,v))|}   \\ F_0(x)
  \end{pmatrix} \in D[\mathcal A], 
  \ene
  where $g(x):=\int_0^x F_0(y) dy.$  Moreover, $U_{F_0}$  belongs to the kernel of $ \mathcal A.$ 
  \item[\rm (b)] Zero in an eigenvalue of $\mathcal A$ and of $\hat{\mathcal A}$ of infinite multiplicity.
  
  \item[\rm (c)] Every vector $ U:=(u_-(x,v), u_+(x,v), F(x)) \in \mathcal H$ that is orthogonal to   $U_{F_0}$ for
all   $F_0$ such that, 
  \beq\label{6.10b}
   F_0\in C^\infty_0(-P/2,P/2), \qquad \int_{(-P/2, P/2)} F_0(x) dx=0
   \ene
  fulfills the Gauss law,
   \beq \label{6.10c}\begin{array}{l}
   \partial_x F(x)=  \int_{\mathbb R} \left( u_+(x,v)\,  \sqrt{|h'_+\left(\mathcal E_+(x,v)\right)|} \right. -\\[.2cm]
\left. u_-(x,v)\, \sqrt{|h'_-\left(\mathcal E_-(x,v)\right)|}\right)\, dv.
 \end{array}
 \ene.
 \end{enumerate}
 \end{proposition}
 \begin{proof}
 That   $U_{F_0}\in D[\mathcal A]$  and that $ \mathcal A U_{F_0}  =0$ follows from direct computation using the definition of $\mathcal A$ in \eqref{0.33}-\eqref{0.35} and remarking that   $ \mathcal D_\mp (\mp g(x)  \sqrt{|h'_\mp\left(\mathcal E_\mp(x,v)\right)|})= \mp F_0(x)  v \sqrt{|h'_\mp\left(\mathcal E_\mp(x,v)\right)|}   .$
 This proves (a). Item (b) follows from (a) and observing that the set of all vectors of the form \eqref{6.10a} is an infinite dimensional  linear subspace of $\mathcal H.$ To prove (c) let  $ U:=(u_-(x,v), u_+(x,v), F(x)) \in \mathcal H$ be orthogonal to $U_{F_0}$ for
  all $F_0$ that satisfy \eqref{6.10b}. Then
 \beq\label{6.10cd}\begin{array}{l}
 (U,U_0)_{\mathcal H} = \int_{(-P/2, P/2)\times \mathbb R} u_+(x,v)\, \overline{g(x)} \, \sqrt{|h'_+\left(\mathcal E_+(x,v)\right)|} dx dv- \\[.3cm] \int_{(-P/2, P/2)\times \mathbb R} u_-(x,v)\, \overline{g(x)}  \,\sqrt{|h'_-\left(\mathcal E_-(x,v)\right)|}dx dv + \int_{(-P/2, P/2)} F(x) \overline{F_0}(x)  dx=0.
 \end{array} 
 \ene
 Remark that any $ g\in C^\infty_0(-P/2, P/2)$ can be written as  $g(x)=\int_0^x F_0(y) \,dy,$ with a
 $F_0$ that satisfies \eqref{6.10b}.  Hence, by \eqref{6.10cd},
 \beq\label{6.10d}
 \left(  h, g  \right )_{L^2 (-P/2, P /2)}=0, \qquad  g\in C^\infty_0(-P/2, P/2),
 \ene
 where
 \beq\label{6.10e}
 h(x)=  - \partial_x F(x)+  \int_{\mathbb R} \left( u_+(x,v)\,  \sqrt{|h'_+\left(\mathcal E_+(x,v)\right)|} \right. -\\[.2cm]
\left. u_-(x,v)\, \sqrt{|h'_-\left(\mathcal E_-(x,v)\right)|}\right)\, dv.
\ene
By \eqref{6.10d} $ h \equiv 0$ and \eqref{6.10c} follows.
 \end{proof}
 
 We denote,
\beq\label{6.18}
M(z):= \hat{\mathcal V} R_{\hat{\mathcal A_0}}(z), \qquad z \in \mathbb C_\pm.
\ene
The operator $M(z)$ is the basic quantity that we use in our perturbation analysis. We find it convenient to split
$M(z)$ into two terms,
\beq\label{6.18b}
M(z)= M_1(z)+ M_2(z),   \qquad z \in \mathbb C_\pm,
\ene
where
\beq\label{6.18c}
M_1(z):=  \hat{\mathcal V}_1 R_{\hat{\mathcal A_0}}(z), \qquad M_2(z):=  \hat{\mathcal V}_2 R_{\hat{\mathcal A_0}}(z)
\ene
with,
\beq\label{6.18d}
   \hat{\mathcal V}_1:= \mathbf V \mathcal V_1 \mathbf V^{-1}, \qquad   \hat{\mathcal V}_2:= \mathbf V \mathcal V_2 \mathbf V^{-1},
   \ene
and (see \eqref{0.35})
   \beq\label{6.18e}\mathcal V_1:= -i \begin{bmatrix} 0&0&  v \sqrt{|h'_-\left(\mathcal E_-(x,v)\right)|}\\
  0&0& - v \sqrt{|h'_+\left(\mathcal E_+(x,v)\right)|}\\
  0&      
  0&0
  \end{bmatrix},
  \ene
     \beq\label{6.18f}\mathcal V_2:= -i \begin{bmatrix} 0&0& 0\\
  0&0& 0\\
  -\mathcal I \int_{\mathbb R} \bullet \sqrt{|h'_-\left(\mathcal E_-(x,v)\right)|} \,v dv&      
  \mathcal I \int_{\mathbb R} \bullet   \sqrt{|h'_+\left(\mathcal E_+(x,v)\right)|} \, v \,dv&0
  \end{bmatrix}.
  \ene
  For $ G=(G_-,G_+,F)^T \in \hat{\mathcal H}$ we have,
  \beq\label{6.18f1}
  M_1(z) G=\begin{pmatrix} (M_1(z) G)_-(\mu,\mathcal E_-)\\ (M_1(z) G)_+(\mu,\mathcal E_+)\\0
\end{pmatrix},
\ene
where
  \beq\label{6.18f2}
  (M_1(z)G)_\mp(\mu, \mathcal E_\mp)=\begin{pmatrix} \mp \ds\frac{ i}{z} v(\mathcal E_\mp, \mu)  \sqrt{|h'_\mp \left(\mathcal E_\mp\right)|} F(x(\mathcal E_\mp,\mu))\\[.3cm] \mp \ds\frac{ i}{z} v(\mathcal E_\mp, \mu)  \sqrt{|h'_\mp\left(\mathcal E_\mp\right)|} F(x(\mathcal E_\mp,\mu))\\[.3cm]
    \mp \ds\frac{ i}{z} v(\mathcal E_\pm, \mu)  \sqrt{|h'_+\left(\mathcal E_\mp\right)|} F(x(\mathcal E_\mp,\mu))
  \end{pmatrix}.
  \ene
 We proceed to write $M_2(z)$ in a way that is convenient for our purposes. First, we extend the definition of $\theta_{-,0}(x,\mathcal E_-)$ to $x\in [-P/2, P/2]$
and of $\theta_{+, 0, \mp}(x,\mathcal E_+),$ respectively, to $ x \in[-P/2,0]$ and to $x \in [0, P/2 ].$
\begin{definition} \label{defang}{\rm
 Recall that the angle $\theta_{-,0}(x, \mathcal E_-)$ is  defined  in  \eqref{1.79} in for $ x_-(\mathcal E_-)\leq x \leq x_+(\mathcal E_-),$ and $ \mathcal E_- \in I_{0,-}.$ Moreover, $ \theta_{-,0}(x_-(\mathcal E_-), \mathcal E_-)=0,$ and $\theta_{-,0}(x_+(\mathcal E_-), \mathcal E_- )= 1/2.$  We  extend the definition of $ \theta_{-,0}(x,\mathcal E_-)$ as follows: $\theta_{-,0}(x, \mathcal E_-):=0,$ for $- P/2\leq  x \leq x_-(\mathcal E_-),$ and $ \theta_{-,0} (x,\mathcal E_-):= 1/2,$ for $  x_+(\mathcal E_-)   \leq x \leq   P/2  .$ Similarly, recall that $\theta_{+,0,\mp}(x, \mathcal E_+)$
 are defined,  respectively, in \eqref{1.89} and in \eqref{1.91} for $ \mathcal E_+\in I_{+,0}$ and for $ x\in [-P/2, x_-(\mathcal E_+)],$ respectively, $ x \in [x_+(\mathcal E_+), P/2)].$ Moreover, $ \theta_{+,0,-}(x_-(\mathcal E_+), \mathcal E_+)=1/4,$ and $ \theta_{+,0,+}(x_+(\mathcal E_+), \mathcal E_+)=1/4.$
 We extend $\theta_{+,0,\mp}(x,\mathcal E_+)$ as follows:
 we define  $\theta_{+,0,-}(x, \mathcal E_+):= 1/4,$ for $x_-(\mathcal E_+)\leq x \leq 0,$ and  $\theta_{+,0,+}(x, \mathcal E_+):= 1/4,$ for $ 0  \leq x \leq  x_+(\mathcal E_+).$}
 \end{definition}
 
 For an operator $B$ in a Hilbert space we denote by $R_B(z)$ its resolvent \cite{kato}.
 
 By \eqref{5.4},  \eqref{5.5}, \eqref{5.84f}-\eqref{5.84k}, for $ z \in \mathbb C_\pm$
\beq\label{6.18g}
(R_{\hat{A}_0}(z)G)_\mp(\mathcal E_\mp,\mu)=\begin{pmatrix}   \sum_{n \in \mathbb Z} \ds\frac{1}{\sqrt{2\pi}} e ^{i2\pi n \mu }    \ds\frac{1}{2\pi n-zT_{\mp,0}(\mathcal E_\mp)}   b_{\mp,0,n}  T_{\mp,0}(\mathcal E_\mp)  \\[.3cm]
\sum_{n\in \mathbb Z}\ds \frac{1}{\sqrt{2\pi}}  e ^{i2\pi n \mu } \ds\frac{1}{2\pi n- zT_{\mp,1}(\mathcal E_\mp)} b_{\mp,1,+,n}( \mathcal E_\mp) T_{\mp,1}(\mathcal E_\mp)\\[.3cm] \sum_{n\in \mathbb Z}\ds \frac{1}{\sqrt{2\pi}}  e ^{i2\pi n \mu } \ds\frac{1}{2\pi n- zT_{\mp,1}(\mathcal E_\mp)} b_{\mp,1,-,n}( \mathcal E_\mp) T_{\mp,1}(\mathcal E_\mp)
\end{pmatrix},
\ene
where
\beq\label{6.6}
b_{\nu,0,n}:= (\mathbf F_1 g_{\nu,0})_n, b_{\nu,1,\upsilon, n}:= (\mathbf F_1g_{\nu,1,\upsilon})_n,\qquad \nu=\mp, \upsilon=\mp, n \in \mathbb Z. 
\ene
Note that,
\beq\label{6.6b}
 (\mathbf F G)_\mp:= \begin{pmatrix}\{b_{\mp,0,n}\}\\ \{ b_{\mp,1,+, n}\}\\\{ b_{\mp,1,-, n}\}\\ 
 \end{pmatrix}.
 \ene
By the second equality in   \eqref{6.18c} and \eqref{6.18d}, and by  \eqref{6.18f}
\beq\label{6.18h}
M_2(z)G=\begin{pmatrix} 0\\0\\ H(x)\end{pmatrix}  
\ene
where by 
 \eqref{1.106},  \eqref{1.112},  \eqref{1.117}, \eqref{1.123}, \eqref{1.135},  \eqref{1.140}, \eqref{6.18f}, and \eqref{6.18g}, 
\beq\label{6.18i}
H(x)= \sum_{\nu=\mp}\left[ H_{\nu,0}(x)+  H_{\nu,1}(x)\right],
\ene 
with:
\beq\label{6.18j}\begin{array}{l}
H_{-,0}(x):=- 2\mathcal I \int_{-\varphi_0(0)}^{-\varphi_0(-P/2)} \sqrt{|h'_-(\lambda)|} \left(\sum_{n \in \mathbb Z_0} \ds\frac{1}{\sqrt{2\pi}} \sin[ 2\pi n \theta_{-,0}(x, \lambda)]   \ds\frac{1}{2\pi n- zT_{-,0}(\lambda)}\right.\\[.3cm]\left.b_{-,0,n}(  \lambda) T_{-,0}(\lambda) \right)\, d\lambda, 
\end{array}
\ene
where we used that when $\mathcal E_-=-\varphi_0(x),$ we have that $ x=x_\mp(\mathcal E_-)$ for $\mp x >0.$
Further,

\beq\label{6.18l}\begin{array}{l}
H_{+,0}(x):= 2\mathcal I \int_{\varphi_0(P/2)}^{\varphi_0(0)} \sqrt{|h'_+(\lambda)|} \left(\sum_{n \in \mathbb Z_0} \ds\frac{1}{\sqrt{2\pi}} a_{+,0,n}(x, \lambda)   \ds\frac{1}{2\pi n- zT_{+,0}(\lambda)}\right.\\[.3cm]\left(b_{+,0,n}(  \lambda) T_{+,0}(\lambda) \right)d\lambda,  
\end{array}
\ene
with
\beq\label{6.18m}
a_{+,0,n}(x,  \mathcal E_+):= \left\{ \begin{array}{l} \pm e^{i\pi n/2} \sin[ 2\pi n \theta_{+,0,\mp }(x, \mathcal E_+)], \qquad  \mp x \geq 0, \, \text{\rm for even} \,n,
\\[.3cm] 
\mp i e^{i\pi n/2} \cos[ 2\pi n \theta_{+,0,\mp }(x, \mathcal E_+)], \qquad  \mp x \geq 0, \, \text{\rm for odd} \,n,
\end{array}\right.
\ene
where we used that when $\mathcal E_+=\varphi_0(x),$ we have that $ x=x_\mp(\mathcal E_+)$ for $\mp x >0.$
Moreover,
\beq\label{6.18n}\begin{array}{l}
H_{-,1}(x):= i\mathcal I \int_{-\varphi_0(-P/2)}^\infty \sqrt{|h'_-(\lambda)|} \left(\sum_{n \in \mathbb Z} \ds\frac{1}{\sqrt{2\pi}} \ds\frac{1}{2\pi n- zT_{-,1}(\lambda)}\right.\\[.3cm]
\left.\left[  e^{i2\pi n \theta_{-,1}(x, \lambda)} b_{-,1,+,n} (  \lambda) -   e^{-i2\pi n \theta_{-,1}(x, \lambda)} b_{-,1,-,n} (  \lambda) \right]    T_{-,1}(\lambda) \right)\, d\lambda,
\end{array}
\ene
and,
\beq\label{6.18o}\begin{array}{l}
H_{+,1}(x):= -i\mathcal I \int_{\varphi_0(0)}^\infty \sqrt{|h'_+(\lambda)|} \left(\sum_{n \in \mathbb Z} \ds\frac{1}{\sqrt{2\pi}} \ds\frac{1}{2\pi n- zT_{+,1}(\lambda)}\right.\\[.3cm]\left.
\left[  e^{i2\pi n \theta_{+,1}(x, \lambda)} b_{+,1,+,n} (  \lambda) -   e^{-i2\pi n \theta_{+,1}(x, \lambda)} b_{+,1,-,n} (  \lambda) \right]    T_{+,1}(\lambda)\right)\, d\lambda.
\end{array}
\ene

We introduce a Banach space that is appropriate for the study of the limit of $M(z)$ when $z \in \mathbb C$ tends to the real axis from above and from below. As mentioned in the introduction, we consider a Banach space of H\"older continuous functions in order to deal with the singularities of the energy-angle variables at the elliptic and the hyperbolic points of the Hamiltonians for the motion of the electrons and the ions under the potential of  the BGK wave.
 For  $0 < \alpha\leq 1,$ and $0 < \xi\leq 1$ we denote,
\beq\label{6.18p}
{\mathcal H}_{\alpha,\xi}:= {\mathcal H}_{-,\alpha}\oplus {\mathcal H}_{+,\alpha} \oplus  \hat{C}_{\xi,0}((-P/2, P/2)),
\ene
where
\beq\label{6.18q}\begin{array}{l}
{\mathcal H}_{\mp, \alpha}:=L^2(I_{\mp,0},L^2(S_1);T_{\mp,0}(\mathcal E_\mp) \,d\mathcal E_\mp )\cap C_{\alpha}(I_{\mp,0}, L^2(S_1))\oplus\\ 
L^2(I_{\mp,1}, L^2(S_1); T_{\mp,1}(\mathcal E_\mp)\, d\mathcal E_\mp) \cap C_{\alpha}(I_{\mp,1}, L^2(S_1)) \oplus\\ [0.2cm] L^2(I_{\mp,1}, L^2(S_1); T_{\mp,1}(\mathcal E_\mp)\, d\mathcal E_\mp ) \cap C_{\alpha}(I_{\mp,1}, L^2(S_1)).
\end{array}
\ene
We first consider the case of $M_1(z).$
\begin{theorem}\label{theom1}
 Suppose that Assumption~\ref{assum0} is satisfied, that    $h'_\pm(v^2/2) v^2 $  is  a integrable function of $v\in \mathbb R,$ and that $h_\pm(\lambda)$ have a continuous second derivative that fulfills,
\beq\label{6.18qww}
|h''_\pm(\lambda)| \leq C \frac{1}{1+\lambda}, \qquad \lambda \geq 0.
\ene
Let $\alpha, \rho,$ and $\xi$  satisfy,  $0<\alpha\leq 1,$   $0 <\rho \leq 1,$ and $ 0< \xi \leq 1.$ Then for
  $\gamma < \min (1/4, \xi/2)$ we have.
\begin{enumerate}
\item[\rm (a)]
 For $ z \in \mathbb C\setminus \{0\}$  $M_1(z)$ is a bounded operator from  ${\mathcal H}_{\alpha, \xi}$ into  ${\mathcal H}_{\gamma, \rho}.$ Moreover, the function $z \in \mathbb C\setminus\{0\} \to M_1(z)$  with values in $\mathcal B ({\mathcal H}_{\alpha, \xi},{\mathcal H}_{\gamma, \rho})$ is analytic. 
 \item[\rm(b)] The operator $I+M_1(z)$ is invertible in   ${\mathcal H}_{\gamma, \xi},$  with a bounded inverse, for $ z \in \mathbb C\setminus\{0\}.$ Further,   for $ G=(G_-,G_+,F)^T \in {\mathcal H}_{\gamma, \xi}$ we have
  \beq\label{6.18q1}
(I+  M_1(z))^{-1} G=  G + \begin{pmatrix} (Q_1(z) G)_-(\mu,\mathcal E_-)\\ (Q_1(z) G)_+(\mu,\mathcal E_+)\\0
\end{pmatrix},
\ene
where
  \beq\label{6.18q2}
  (Q_1(z)G)_\mp(\mu, \mathcal E_\mp)=\begin{pmatrix} \ds\mp \frac{ i}{z} v(\mathcal E_\mp, \mu)  \sqrt{|h'_\mp\left(\mathcal E_\mp\right)|} F(x(\mathcal E_\mp,\mu))\\[.3cm] \mp \ds \frac{ i}{z} v(\mathcal E_\mp, \mu)  \sqrt{|h'_\mp\left(\mathcal E_\mp\right)|} F(x(\mathcal E_\mp,\mu))\\[.3cm]
   \ds \mp \frac{ i}{z} v(\mathcal E_\mp, \mu)  \sqrt{|h'_\mp\left(\mathcal E_\mp\right)|} F(x(\mathcal E_\mp,\mu))
  \end{pmatrix}.
  \ene
 \end{enumerate}
 \end{theorem}
\begin{proof}
We remark that    for any $ 0 < \delta < 1/2$  we have,
\beq\label{6.18r}
| x({\mathcal E}_{\pm,2}, \mu)-x({\mathcal E}_{\pm,1}, \mu) |\leq C   |{\mathcal E}_{\pm,2}-{\mathcal E}_{\pm,1}|^\delta,  
\ene
for $ {\mathcal E}_{\pm, 1}$ and  ${\mathcal E}_{\pm,2}$ in $I_{\pm,0}$ or    ${\mathcal E}_{\pm,1}$ and  ${\mathcal E}_{\pm,2}$ in $I_{\pm,1},$  and  with a fixed constant $C$ independent of ${\mathcal E}_{\pm,1}, $        ${\mathcal E}_{\pm, 2}$ and  $\mu.$ This result is proved as in Lemma 3.5 of \cite{hrss}, see also Proposition A.3 of \cite{wegal}. 
 Note that in the right-hand side of \eqref{6.18r}  the exponent $\delta$ is smaller than $1/2$ instead of $1/2$ as in  \cite{hrss} and \cite{wegal}. This is so because of the logarithmic divergence of the period functions stated  in (a), (b), (f), and (g) of Proposition~\ref{proper}. 
 By \eqref{6.18r} and recalling that $v=\pm \sqrt{2(\mathcal E_\mp \pm \varphi_0(x)},$  
  it follows that for any $\Delta>0,$ for $ \mathcal E_{\pm,2}> \mathcal E_{\pm,1},$ and  $1 \geq\mathcal E_{\pm,2}-\mathcal E_{\pm,1}\geq 0$ we have, 
   
  \beq\label{6.18r2}
  \left| v(\mathcal E_{\pm,2}, \mu) \sqrt{|h'_\pm(\mathcal E_{\pm,2})|}- v(\mathcal E_{\pm,1}, \mu) \sqrt{|h'_\pm(\mathcal E_{\pm,1})|}\right| \leq C |\mathcal E_{\pm,2}-\mathcal E_{\pm,1)}|^{\delta/2},  \mathcal E_{\pm,2} \leq \Delta+1. 
  \ene 
   Further, let us take $\Delta >1$ such that
  $    0 \leq \mathcal E_\pm \mp \varphi_0(x) \leq C \mathcal E_\pm$ for $ \mathcal E_\pm \geq \Delta-1 >0.$
  Then, for $ \mathcal E_{\pm,2}> \mathcal E_{\pm,1},$ and  $1 \geq\mathcal E_{\pm,2}-\mathcal E_{\pm,1}\geq 0
  $  we have,
  \beq\label{6.18r3}\begin{array}{l}
    \left| v(\mathcal E_{\pm,2}, \mu) \sqrt{|h'_\pm(\mathcal E_{\pm,2})|}- v(\mathcal E_{\pm,1}, \mu) \sqrt{|h'_\pm(\mathcal E_{\pm,1})|}\right| \leq C |\mathcal E_{\pm,2}-\mathcal E_{\pm,1}|^{\delta/2}+\\[.3cm]
    C\, \sqrt{  v(\mathcal E_{\pm,1}, \mu)^2 |h'_\pm(\mathcal E_{\pm,2})- h'_\pm(\mathcal E_{\pm,1})|},\qquad
    \mathcal E_{\pm,2} \geq \Delta+1,
    \end{array}  \ene
    where we used \eqref{6.18r}.
   Moreover, by \eqref{6.18qww},
   \beq\label{6.18qwww}
     |v(\mathcal E_{\pm,1}, \mu)^2 (h'_\pm(\mathcal E_{\pm,2})- h'_\pm(\mathcal E_{\pm,1}))|\leq C|\mathcal E_{\pm,2}-\mathcal E_{\pm,1}|.
 \ene 
 Further, by \eqref{6.18r3} and \eqref{6.18qwww},
\beq\label{6.18qxxxa}
 \left| v(\mathcal E_{\pm,2}, \mu) \sqrt{|h'_\pm(\mathcal E_{\pm,2})|}- v(\mathcal E_{\pm,1}, \mu) \sqrt{|h'_\pm(\mathcal E_{\pm,1})|}\right| \leq C  |\mathcal E_{\pm,2}-\mathcal E_{\pm,1}|^{\delta/2},\qquad   \mathcal E_{\pm,2} \geq \Delta.
 \ene
Hence, by \eqref{6.18r2} and \eqref{6.18qxxxa},
\beq\label{6.18qaaa}
 \left| v(\mathcal E_{\pm,2}, \mu) \sqrt{|h'_\pm(\mathcal E_{\pm,2})|}- v(\mathcal E_{\pm,1}, \mu) \sqrt{|h'_\pm(\mathcal E_{\pm,1})|}\right| \leq C  |\mathcal E_{\pm,2}-\mathcal E_{\pm,1}|^{\delta/2}, 
   |\mathcal E_{\pm,2}-\mathcal E_{\pm,1}|\leq 1.
 \ene
 Moreover, by \eqref{6.18r} for $ F \in \hat{C}_{\xi,0}((-P/2,P/2)),$ 
 \beq\label{6.18qbbb}
 |F(x(\mathcal E_{\pm,2},\mu))- F(x(\mathcal E_{\pm,1}, \mu))| \leq C |\mathcal E_{\pm,2}- \mathcal E_{\pm,1}|^{\xi\delta}.
 \ene
It follows from   \eqref{6.18f1}, \eqref{6.18f2}, \eqref{6.18qaaa}, and \eqref{6.18qbbb}  that $M_1(z)$ is a bounded operator from  ${\mathcal H}_{\alpha, \xi}$ into  ${\mathcal H}_{\gamma, \rho}.$ 
Moreover, the analyticity of the function  $z \in \mathbb C\setminus\{0\} \to M_1(z)$  with values in $\mathcal B ({\mathcal H}_{\alpha, \xi},{\mathcal H}_{\gamma, \rho})$ is immediate from \eqref{6.18f1} and \eqref{6.18f2}. 
  Let us prove (b). That $I+M_1(z)$ is invertible and that the inverse is given by \eqref{6.18q1} and \eqref{6.18q2} follows from   \eqref{6.18f1}, \eqref{6.18f2} and a simple calculation. We prove that the operator  \eqref{6.18q1}, \eqref{6.18q2} is bounded in  ${\mathcal H}_{\gamma, \xi}$ as in the proof of (a).
  
  \end{proof}
In the following theorem we prove that $M_2(z),$  with $z \in \mathbb C_\pm,$ is compact from $\mathcal H_{\alpha, \xi}$ into $ \mathcal H_{\eta, \rho},$ for appropriate $\alpha, \xi, \eta,$ and $\rho,$ and that it has continuous limits in operator norm  as $z$ tends to $\mathbb R$ from above and from below.
We designate,
\beq\label{6.17}
\Gamma:= \mathbb R \setminus \left[\cup_{n \in \mathbb Z_0} \left\{ \frac{2 \pi n}{T_{-,0}(-\varphi_0(0))},  \frac{2 \pi n}{T_{+,0}(\varphi_0(P/2))} \right\} \cup \{0\}\right].
\ene

\begin{theorem}\label{theolim} Suppose that Assumption~\ref{assum0} is satisfied, that $h_\pm$ are twice continuously differentiable,  and that    $h'_\pm(v^2/2) v^2 $  is  a integrable function of $v\in \mathbb R.$ Moreover, assume  that $0 <\alpha\leq 1$, that
 $0 <\xi\leq 1,$ that $0 <\eta\leq 1,$ and that  that $0 < \rho < 1/2.$
Then.
\begin{enumerate}        
   \item[\rm (a)]   For $ z\in \mathbb C_\pm,$ the operator   $ M_2(z)$ is compact from  $\mathcal H_{\alpha, \xi}$ into $ \mathcal H_{\eta, \rho}.$ Further,  the function $ z\in \mathbb C_\pm  \to   M_{2}(z)$ with values in  $ \mathcal B( \mathcal H_{\alpha, \xi},  \mathcal H_{\eta, \rho})$
 are analytic for $ z \in \mathbb C_\pm.$ Moreover,
 
 \beq\label{6.19c}  
\lim_{ a\to \pm\infty} \|M_2(ia)\|_{ \mathcal B( \mathcal H_{\alpha, \xi},  \mathcal H_{\eta, \rho})}=0.
\ene
   \item[\rm (b)] If, furthermore, $ 0 < \rho < 1/4,$ the following limits exist in the uniform operator topology in $ \mathcal B(\mathcal H_{\alpha, \xi},  \mathcal H_{\eta, \rho}).$
\beq\label{6.19}
M_{2}(\lambda\pm i0):= \lim_{\varepsilon \downarrow 0} M_2(\lambda \pm i \varepsilon),
\qquad \lambda \in \Gamma,
\ene
and the convergence is uniform in any closed and bounded  interval contained in $\Gamma.$
\item[\rm (c)] We define,
\beq\label{6.19b}
M_{2,\pm}(z):= \left\{\begin{array}{l}     M_2(z), \qquad z \in \mathbb C_\pm,\\
 M_{2}(z \pm i 0), \qquad z \in \Gamma. 
 \end{array}\right.
 \ene
Then,  the functions $ z\in \mathbb C_\pm \cup \Gamma \to   M_{2,\pm}(z)$ with values in  $ \mathcal B( \mathcal H_{\alpha, \xi},  \mathcal H_{\eta, \rho})$
 are analytic for $ z \in \mathbb C_\pm$ and continuous for   $z\in \mathbb C_\pm \cup \Gamma.$
\item[\rm (d)]
For $z \in \Gamma$ 
the operators  $M_{2,\pm} (z)$ are compact from $\mathcal H_{\alpha, \xi}$ into $ \mathcal H_{\eta, \rho}.$

 \end{enumerate} 
\end{theorem}
\begin{proof}   We consider \eqref{6.18h} and \eqref{6.18i}. We give our proofs for the term $H_{-,0}(x)$ in \eqref{6.18j}. The remaining terms in \eqref{6.18i} are studied in a similar way. We write  $H_{-,0}(x)$ as follows,
\beq\label{6.19d}\begin{array}{l}
H_{-,0}(x)=- 2\mathcal I \int_{-\varphi_0(0)}^{-\varphi_0(-P/2)} \sqrt{|h'_-(\lambda)|}   \left( \sum_{n \in \mathbb Z_0} \ds\frac{1}{\sqrt{2\pi}} \sin[ 2\pi n \theta_{-,0}(x, \lambda)]   \ds g_{n}(z,\lambda)\right.\\[.3cm]  \left.  \ds b_{-,0,n}(  \lambda) T_{-,0}(\lambda) \right)d\lambda,
\end{array}
\ene
where,
\beq\label{6.20}
g_{n}(z,\lambda):=  \ds\frac{1}{2\pi n-z  {T_{-,0}(\lambda)} }.
\ene
We have,
\beq\label{6.20b}
g_{n}(z,\lambda):= \frac{1}{2\pi n- i  T_{-,0}(\lambda) }+ (z-i)    \ds\frac{1}{\frac{2\pi n}{T_{-,0}(\lambda)}-z}\, \frac{1}{2\pi n- iT_{-,0}(\lambda)}.
\ene
Note that by \eqref{6.20b},
\beq\label{6.21}
\left| g_n(z,\lambda)  \right| \leq  C \frac{1}{1+|n|} \left(1+\frac{1+|z|}{|{\rm Im z}|}\right)    , \qquad \lambda \in I_{-,0}, n \in \mathbb Z.
\ene
 Further, by \eqref{6.19d}, \eqref{6.21},  (a), (c),  and (d) of Proposition~\ref{proper},
\beq\label{6.22}
\left| H_{-,0}(x)\right| \leq C \|G\|_{\mathcal H_{\alpha, \xi}}, \qquad x \in [-P/2, P/2],
\ene
where the constant $C$ is uniform in compact sets of $\mathbb C_\pm.$ Moreover, assume that $ \varphi_0(x_2)\leq \varphi_0(x_1).$ Hence,
\beq\label{6.23}
H_{-,0}(x_2)- H_{-,0}(x_1)= H_{-,0}^{(1)}(x_1,x_2)+  H_{-,0}^{(2)}(x_1,x_2),
\ene
where
\beq\label{6.24}\begin{array}{l}
H_{-,0}^{(1)}(x_1,x_2)=- 2 \int_{-\varphi_0(x_2)}^{-\varphi_0(-P/2)} \sqrt{|h'_-(\lambda)|} \left(\sum_{n \in \mathbb Z_0} \ds\frac{1}{\sqrt{2\pi}}( \sin[ 2\pi n \theta_{-,0}(x_2, \lambda)]-\right.\\[.3cm]\left.\sin[ 2\pi n \theta_{-,0}(x_1, \lambda)])   \ds g_{n}(z,\lambda)  \ds b_{-,0,n}(  \lambda) T_{-,0}(\lambda) \right )d\lambda 
\end{array}
\ene 
and,
\beq\label{6.25}\begin{array}{l}
H_{-,0}^{(2)}(x_1,x_2)= 2 \int_{-\varphi_0(x_1)}^{-\varphi_0(x_2)} \sqrt{|h'_-(\lambda)|} \left(\sum_{n \in \mathbb Z_0} \ds\frac{1}{\sqrt{2\pi}}[ \sin[ 2\pi n \theta_{-,0}(x_1, \lambda)]   \ds g_{n}(z,\lambda)\right.\\[.3cm]  \left. \ds  b_{-,0,n}(  \lambda) T_{-,0}(\lambda)\right) d\lambda.
\end{array}
\ene
For $ -\varphi_0(x_2) \leq \lambda \leq -\varphi_0(-P/2),$ we have,
\beq\label{6.26}
 |\theta_{-,0}(x_2, \lambda)-\theta_{-,0}(x_1, \lambda)|\leq C \frac{1}{\sqrt{\lambda+\varphi_0(x_2)}} |x_1-x_2|.
 \ene
 Using \eqref{6.26} in \eqref{6.24}, (a), (c),  and (d) of Proposition~\ref{proper} we see that for any $0 < \varepsilon < 1/2$ there is a constant $C$ such that,
  \beq\label{6.27}
  |H_{-,0}^{(1)}(x_1,x_2)| \leq C |x_1-x_2|^\varepsilon \| G\|_{\mathcal H_{\alpha, \xi}}, \qquad x \in [-P/2, P/2],
  \ene
  where the constant $C$ is uniform in compact sets of $\mathbb C_\pm.$ Moreover, by \eqref{6.25}, (a), (c),  and (d) of Proposition~\ref{proper}, for any $ 0  < \varepsilon <1,$
    \beq\label{6.28}
  |H_{-,0}^{(2)}(x_1,x_2)| \leq C |x_1-x_2|^\varepsilon \| G\|_{\mathcal H_{\alpha, \xi}}, \qquad x \in [-P/2, P/2],
  \ene
   where the constant $C$ is uniform in compact sets of $\mathbb C_\pm.$ Estimating the remaining terms in \eqref{6.18i} as in  \eqref{6.22}-\eqref{6.25}, \eqref{6.27}, \eqref{6.28}, and using that for any $\tilde{\rho} > \rho,$
   we have that  $C_{\tilde{\rho}}(I_{-,0}, L^2(S_1)) \subset  \hat{C}_{{\rho}}(I_{-,0}, L^2(S_1))$ with a continuous imbedding  (see Proposition A.1 of \cite{wegal}) we prove that for $z \in \mathbb C_\pm$  the operator $M_2(z)$ is  bounded
   from $\mathcal H_{\alpha,\xi}$ into  $\mathcal H_{\eta,\rho}.$ Moreover, the analyticity of $M_2(z)$ follows derivating \eqref{6.19d} under the integral sign,  proceeding as in the proof of the boundedness of $H_{-,0},$ and estimating the remainig terms in   \eqref{6.18i} similarly.
    Equation \eqref{6.19c} follows observing that for $ a \in \mathbb R \setminus \{0\},$
    $$
    | g_n(ia) | \leq \frac{1}{2\pi n+|a|},
    $$
    estimating  $H_{-,0}(x)$  as in  \eqref{6.22}-\eqref{6.25}, \eqref{6.27}, \eqref{6.28},
   and estimating the remaining terms in \eqref{6.18i} in a similar way.
   Let us prove that  $M_2(z) $  is  compact. For this purpose, we prove that each one of the terms in \eqref{6.18i} defines a  compact operator from $\mathcal H_{\alpha,\xi}$ into $\hat{C}_{\rho,0}(-P/2,P/2).$ Let us consider $H_{-,0}(x).$ We denote,
   \beq\label{6.29}
  \begin{array}{l}
H_{-,0,N}(x)=- 2\mathcal I \int_{-\varphi_0(0)}^{-\varphi_0(-P/2)} \sqrt{|h'_-(\lambda)|} \left(\sum_{|n| > N} \ds\frac{1}{\sqrt{2\pi}} \sin[ 2\pi n \theta_{-,0}(x, \lambda)]   \ds g_{n}(z,\lambda)\right.\\[.3cm] \left.   \ds b_{-,0,n}(  \lambda) T_{-,0}(\lambda)\right) d\lambda.
\end{array}
\ene
Estimating as in \eqref{6.22}-\eqref{6.25}, \eqref{6.27}, and \eqref{6.28} we prove that,
\beq\label{6.30}
\lim_{N \to \infty} \|H_{-,0,N}\|\ds_{\mathcal B ( \mathcal H_{\alpha, \beta}, \hat{C}_{\rho,0}(-P/2,P/2))}=0.
\ene 
Then , it is sufficient to consider,
\beq\label{6.31}
 \begin{array}{l}
H_{-,0,\tilde{N}}(x)=- 2\mathcal I \int_{-\varphi_0(0)}^{-\varphi_0(-P/2)} \sqrt{|h'_-(\lambda)|} \left(
\sum_{n \in \mathbb Z_0: |n| \leq N} \ds\frac{1}{\sqrt{2\pi}} \sin[ 2\pi n \theta_{-,0}(x, \lambda)]   \ds g_{n}(z,\lambda)\right.\\[.3cm] \left.   \ds\  b_{-,0,n}(  \lambda) T_{-,0}(\lambda)\right) d\lambda.
\end{array}
\ene
 Suppose that $G^{(l)},$  for $l=1,2,\dots$ is a bounded sequence in $\mathcal H_{\alpha, \xi}.$ Then  the vector valued sequence $(b_{-,0,-N}^{(l)}(\lambda), \dots, b_{-,0,-1}^{(l)}(  \lambda), b_{-,0,1}^{(l)}(  \lambda),\dots, b_{-,0,N}^{(l)}(  \lambda))$ is   uniformly bounded and uniformly equicontinuous. Hence, by the  Arzel\`a--Ascoli theorem, Theorem 7.25 of \cite{ru}, it has a subsequence, that we also denote by  $(b_{-,0,-N}^{(l)}(  \lambda), \dots,b_{-,0,-1}^{(l)}(  \lambda), b_{-,0,1}^{(l)}(  \lambda),\dots, b_{-,0,N}^{(l)}(  \lambda)),$  that converges uniformly in $\overline{I_{-,0}}$ to a continuous vector valued function  $(b_{-,0,-N}(  \lambda), \dots,\linebreak b_{-,0,-1}(  \lambda), b_{-,0,1}(  \lambda),\dots, (b_{-,0,N}(  \lambda)).$   
 Moreover, as in \eqref{6.22}-\eqref{6.25}, \eqref{6.27}, and \eqref{6.28} we prove that
  \beq\label{6.32}
  \begin{array}{l}  \lim_{l \to \infty} 
- 2\mathcal I \int_{-\varphi_0(0)}^{-\varphi_0(-P/2)} \sqrt{|h'_-(\lambda)|} \left(
\sum_{n \in \mathbb Z_0: |n| \leq N} \ds\frac{1}{\sqrt{2\pi}} \sin[ 2\pi n \theta_{-,0}(x, \lambda)]   \ds g_{n}(z,\lambda)\right.\\[.3cm] \left.  \ds (b_{-,0,n}^{(l)}(\lambda)- b_{-,0,n}( \lambda)) T_{-,0}(\lambda)\right) d\lambda=0,
\end{array}
\ene
with the limit in the norm of $\hat{C}_{\rho,0}(-P/2,P/2).$
This completes the prof that the operator defined by $H_{-,0}(x)$ is compact . We   prove that the remaining terms in \eqref{6.18i} define  compact operators from $\mathcal H_{\alpha,\xi}$ into $\hat{C}_{\rho,0}(-P/2,P/2)$ in a similar way. Remark that in the case of $H_{-,1}(x)$ we first approximate it in norm by
$$
\begin{array}{l}
H_{-,1}^{(M)}(x):= i\mathcal I \int_{-\varphi_0(-P/2)}^M \sqrt{|h'_-(\lambda)|}\left( \sum_{n \in \mathbb Z} \ds\frac{1}{\sqrt{2\pi}} \ds\frac{1}{2\pi n- zT_{-,1}(\lambda)}\right.\\[.3cm] \left.
\left[  e^{i2\pi n \theta_{-,1}(x, \lambda)} b_{-,1,+,n} (  \lambda) -   e^{-i2\pi n \theta_{-,1}(x, \lambda)}   b_{-,1,-,n} (  \lambda) \right]    T_{-,1}(\lambda) \right)d\lambda,
\end{array}
$$
for $M$ large, and then, we proceed as in the case of $H_{-,0}.$ In the case of $H_{+,1}(x)$ we proceed in a similar way. This completes the proof of the compactness of $M_2(z).$ Let us prove (b). Let $I=[a,b]$ be a closed and bounded interval contained in  $\Gamma \cap (0,\infty).$ The case of $ I \subset \Gamma \cap (-\infty,0))$ is similar. 
We prove that the operator defined by each one of the terms in \eqref{6.18i} has continuous extensions from above and from below to $I.$ Let us consider the term $H_{-,0}(x).$ For clarity, we make explicit the dependence in $z$ and denote   $H_{-,0}(x,z).$  We designate
$$
J:=\left\{ n \in \mathbb N: I \cap \left(0, \frac{2\pi n}{T_{-,0}(-\varphi_0(0))}\right)\right \} \neq \emptyset, \qquad \tilde{J}= \mathbb N \setminus J. 
$$ 
We decompose $H_{-,0}(x,z)$ as follows
\beq\label{6.33}
H_{-,0}(x,z)= H_{-,0}^{(1)}(x,z)+H_{-,0}^{(2)}(x,z),
\ene
where, 
\beq\label{6.34}\begin{array}{l}
H_{-,0}^{(1)}(x,z)=- 2\mathcal I \int_{-\varphi_0(0)}^{-\varphi_0(-P/2)} \sqrt{|h'_-(\lambda)|} \left(   \sum_{n \in \tilde{J}} \ds\frac{1}{\sqrt{2\pi}} \sin[ 2\pi n \theta_{-,0}(x, \lambda)]   \ds g_{n}(z,\lambda)\right.\\[.3cm] \left.   \ds b_{-,0,n}(  \lambda) T_{-,0}(\lambda) \right)d\lambda,
\end{array}
\ene 
 and
 \beq\label{6.35}\begin{array}{l}
H_{-,0}^{(2)}(x,z)=- 2\mathcal I \int_{-\varphi_0(0)}^{-\varphi_0(-P/2)} \sqrt{|h'_-(\lambda)|} \left(\sum_{n \in {J}} \ds\frac{1}{\sqrt{2\pi}} \sin[ 2\pi n \theta_{-,0}(x, \lambda)]   \ds g_{n}(z,\lambda)\right.\\[.3cm] \left.  \ds b_{-,0,n}(  \lambda) T_{-,0}(\lambda) \right) d\lambda.
\end{array}
\ene 
 Remark that the distance,
 $$
 \text{\rm dist}\left( I, \cup_{n\in \tilde{J}} \left[0,\frac{2\pi n}{T_{-,0}(-\varphi_0(0))} \right]\right)  >0.
 $$
 Then, we prove as in the proof of (a)  that the operator defined by $H_{-,0}^{(1)}(x,z),$ for $z= \lambda\pm i\varepsilon,$ wtih $ \lambda \in I,$ extends  continuously to $ z=\lambda \in I$ in the operator norm in $\mathcal B(\mathcal H_{\alpha, \xi}, \hat{C}_{\rho,0})$ and that the convergence is uniform for $ \lambda \in I.$ Furthermore, the operator defined by  $H_{-,0}^{(1)}(x,z)$ is analytic for $ z= \lambda \in I.$  We now consider $H_{-,0}^{(2)}(x,z).$ Let $\psi \in C^1_0(\mathbb R)$ satisfy $ \psi(\lambda)=1,$ for $\lambda \in [a-\nu, b+\nu],$ for some $\nu$ such that $  [a-\nu, b+\nu] \subset \Gamma,$ and with the support of $ \psi$ contained in  $[a-\mu, b+\mu],$ for some $\mu$ such that $  [a-\mu, b+\mu] \subset \Gamma,$ and $\mu > \nu.$ We decompose  $H_{-,0}^{(2)}(x,z)$ as follows,
 \beq\label{6.36}
  H_{-,0}^{(2)}(x,z)=  H_{-,0}^{(2,1)}(x,z)+ H_{-,0}^{(2,2)}(x,z),
 \ene
 where
  \beq\label{6.37}\begin{array}{l}
H_{-,0}^{(2,1)}(x,z)=- 2\mathcal I \int_{-\varphi_0(0)}^{-\varphi_0(-P/2)} \sqrt{|h'_-(\lambda)|}\left( \sum_{n \in {J}} \ds\frac{1}{\sqrt{2\pi}} \sin[ 2\pi n \theta_{-,0}(x, \lambda)]\right.\\[.3cm] \left. (1-\psi(2\pi n/T_{-,0}(\lambda)) \ds g_{n}(z,\lambda)  \ds b_{-,0,n}(  \lambda) T_{-,0}(\lambda) \right)d\lambda,
\end{array}
\ene 
and, 
  \beq\label{6.38}\begin{array}{l}
H_{-,0}^{(2,2)}(x,z)=- 2\mathcal I \int_{-\varphi_0(0)}^{-\varphi_0(-P/2)} \sqrt{|h'_-(\lambda)|}\left( \sum_{n \in {J}} \ds\frac{1}{\sqrt{2\pi}} \sin[ 2\pi n \theta_{-,0}(x, \lambda)] \right.\\[.3cm] \left.\psi(2\pi n/T_{-,0}(\lambda)) \ds g_{n}(z,\lambda)  \ds b_{-,0,n}(  \lambda) T_{-,0}(\lambda)\right) d\lambda.
\end{array}
\ene 
For  $ z=\kappa\pm i\varepsilon,$ with $ \kappa \in I,$ we have that, $|2\pi n/T_{-,0}(\lambda)-z| \geq \tau>0,$
for all $\lambda$ in the support of $1-\psi(2\pi n/T_{-,0}(\lambda)).$ Hence, as in the proof of (a)  we prove that the operator defined by $H_{-,0}^{(2,1)}(x,z),$ for $z= \lambda\pm i\varepsilon,$ wtih $ \lambda \in I,$ extends  continuously to $ z=\lambda \in I$ in the operator norm in $\mathcal B(\mathcal H_{\alpha, \xi}, \hat{C}_{\rho, 0})$ and that the convergence is uniform for $ \lambda \in I.$ Furthermore, the operator defined by  $H_{-,0}^{(2,1)}(x,z)$ is analytic for $ z= \lambda \in I.$ We now consider $H_{-,0}^{(2,2)}(x,z)$ that is where the singularity lies. Using the notation introduced in \eqref{5.49}-\eqref{5.57} we express $H_{-,0}^{(2,2)}(x,z)$ as follows,

 \beq\label{6.39}\begin{array}{l}
H_{-,0}^{(2,2)}(x,z)=- 2\mathcal I\sum_{n \in {J}} \int_{a-\mu}^{b+\mu} \sqrt{|h'_-(\mathcal E_{-,0,n}(\beta))|}  \ds\frac{1}{\sqrt{2\pi}} \sin[ 2\pi n \theta_{-,0}(x, \mathcal E_{-,0,n}(\beta))] \\[.3cm] \psi(\beta) \ds  \frac{1}{\beta-z}   b_{-,0,n}( \mathcal E_{-,0,n}(\beta)) p_{-,0,n}(\beta)\,  d\beta.
\end{array}
\ene 
 By \eqref{5.53} and Proposition~\ref{proper} (d)
 \beq\begin{array}{c} \label{6.40}
\left |\mathcal E'_{-,0,n}(\beta)\right|=\left | p_{-,0,n}(\beta\right|\leq C \ds\frac{1}{n},\\[.2cm]
\left| \mathcal E_{-,0,n}(\beta_1)-   \mathcal E_{-,0,n}(\beta_2) \right | \leq  C\frac{1}{n} |\beta_1-\beta_2|,
 \\[.2cm]
\left| p_{-,0,n}(\beta_1)-   p_{-,0,n}(\beta_2)  \right| \leq C\frac{1}{n^2} |\beta_1-\beta_2|, \qquad \beta, \beta_1,\beta_2 \in [a-\mu, b+\mu].
 \end{array}
 \ene
  Further, the constant $C$ in \eqref{6.40}  is uniform in $ n \in \mathbb N.$ 
  We denote,
\beq\label{6.41}
k_{-,0,n }(x,\beta):=    \sqrt{|h'_-( \mathcal E_{-, 0,n}(\beta))|}
 \ds\frac{1}{\sqrt{2\pi}} \sin[2\pi n \theta_{-,0}(x, \mathcal E_{-, 0,n}(\beta))]  p_{-,0,n}(\beta).
 \ene 
 By \eqref{6.40} and \eqref{ap.12}
 \beq \begin{array}{l}\label{6.42}
 \left|  k_{-,0,n }(x,\beta)\right| \leq C \frac{1}{n},\\
 \left|  k_{-,0,n }(x,\beta_1)- k_{-,0,n }(x,\beta_2) \right|  \leq  C_\Delta \,  \frac{1}{n^{1/2+\Delta/2}} |\beta_1-\beta_2|^{1/2-\Delta/2},\\[.3cm] \qquad  0 <\Delta <1,     \beta, \beta_1,\beta_2 \in [a-\mu, b+\mu],
\end{array}
\ene
with  the constants $C$  and $C_\Delta$ in \eqref{6.42}  uniform in $ n \in \mathbb N$ and in  $x \in (-P/2, P/2) .$  By \eqref{6.39} and \eqref{6.41}, for    $ \kappa \in [a,b],$
 \beq\label{6.43}\begin{array}{l}
H_{-,0}^{(2,2)}(x, \kappa\pm i \varepsilon )=- 2\mathcal I\sum_{n \in {J}} \left[\int_{a-\mu}^{b+\mu} \left( k_{-,0,n}(x,\beta) \psi(\beta)    b_{-,0,n}( \mathcal E_{-,0,n}(\beta)) - k_{-,0,n}(x,  \kappa)\psi(\kappa) \right.  \right.\\[.3cm] \left.\left.  b_{-,0,n}( \mathcal E_{-,0,n}(\kappa) )\right)   \ds  \frac{1}{\beta-\kappa\mp i \varepsilon}\,  d\beta+  k_{-,0,n}(x,  \kappa) \psi(\kappa) b_{-,0,n}( \mathcal E_{-,0,n}(\kappa)) 
   \right.\\[.3cm] \left.(\log(b+\mu-\kappa\mp i\varepsilon)- \log(a-\mu-\kappa\mp i\varepsilon))\right],  
\end{array}
\ene 
where we take the principal branch of the logarithm continuous in the positive real axis.
Moreover,  for any  $ \upsilon >0,$ and any   $ 0< \zeta < 1/2,$ there is a constant $C$ such that,
\beq\label{6.44}\begin{array}{l}
\left|\theta_{-,0}(x_1,\mathcal E_-)- \theta_{-,0}(x_2,\mathcal E_-)\right| \leq C  |x_1-x_2|^\zeta  ,  \mathcal E_{-,0} \in[ -\varphi_0(0)+ \upsilon, -\varphi_0(-P/2)-\upsilon ], \\[.3cm] x \in [P/2, P/2].
\end{array}
\ene
This estimate is proved as in Proposition A.4 of \cite{wegal}. By the first equation in \eqref{6.40},  \eqref{6.41},
and \eqref{6.44}
for any $0< \hat{\Delta} < 1,$ and any $ 0< \zeta < 1/2,$
\beq\label{6.45}
 |k_{-,0,n }(x_2,\beta)-k_{-,0,n }(x_1,\beta)| \leq C \frac{1}{n^{\hat{\Delta}} }|x_1-x_2|^{\zeta(1-\hat{\Delta)}},\qquad
      \beta \in [a-\mu, b+\mu],
\ene
with  the constant $C$  in \eqref{6.45}  uniform in $ n \in \mathbb N$ and in  $x \in (-P/2, P/2) .$
By the second equation in \eqref{6.42} and \eqref{6.45}
\beq\label{6.46} \begin{array}{l}
\left| k_{-,0,n }(x_2,\beta)-k_{-,0,n }(x_2,\kappa)-  ( k_{-,0,n }(x_1,\beta)-k_{-,0,n }(x_1,\kappa))\right |\leq C\\[.3cm]
\left[ \frac{1}{n^{1/2+\Delta/2}} |\beta -\kappa|^{1/2-\Delta/2}\right]^{1/2+\iota}
\left[   \frac{1}{n^{\hat{\Delta}}} |x_1-x_2|^{\zeta(1-\hat{\Delta)}}  \right ]^{1/2-\iota}, \qquad 0 \leq \iota \leq 1/2.
\end{array}
\ene   
Then, given any $\rho < 1/4$ we can pick in \eqref{6.46}  $\Delta, \hat{\Delta}, \iota,$ and $\zeta$ such that,
\beq\label{6.47}
 \begin{array}{l}
\left| k_{-,0,n }(x_2,\beta)-k_{-,0,n }(x_2,\kappa)-  ( k_{-,0,n }(x_1,\beta)-k_{-,0,n }(x_1,\kappa))\right |\leq C\\[.3cm]
 \frac{1}{n^{1/2+\varpi_1}} |\beta -\kappa|^{\varpi_2}  |x_1-x_2|^\rho,
\end{array}
\ene   
  for some positive $ \varpi_1$ and $\varpi_2.$  The constant $C$  is  uniform in $ n \in \mathbb N,$  in  $x \in (-P/2, P/2),$  and in $ \beta$ and $\kappa$ in $[a-\mu, b+\mu].$
Furthermore, by \eqref{6.42}, \eqref{6.43}, \eqref{6.45},  and
\eqref{6.47}
\beq\label{6.48}
\begin{array}{l}
\lim_{\varepsilon \downarrow 0}H_{-,0}^{(2,1)}(x, \kappa \pm i \varepsilon )=- 2\mathcal I\sum_{n \in {J}} \left[\int_{a-\mu}^{b+\mu} ( k_{-,0,n}(x,\beta) \psi(\beta)    b_{-,0,n}( \mathcal E_{-,0,n}(\beta)) -\right.\\[.3cm] \left. k_{-,0,n}(x,  \kappa)\psi(\kappa)     b_{-,0,n}( \mathcal E_{-,0,n}(\kappa) ))   \ds  \frac{1}{\beta-\kappa}\,  d\beta+  k_{-,0,n}(x,  \kappa) \psi(\kappa) b_{-,0,n}( \mathcal E_{-,0,n}(\kappa)) 
  \right.\\[.3cm] \left.(\log(b+\mu-\kappa)- \log(\kappa -a+\mu )\mp i \pi)\right],   
\end{array}
\ene 
with the limit in the uniform operator topology in  $ \mathcal B(\mathcal H_{\alpha, \xi},  \mathcal H_{\eta, \rho}).$
Furthermore, the limit in \eqref{6.48} is uniform for $ \kappa \in [a,b].$ Estimating the remaining terms in \eqref{6.18i} in a similar way we complete the proof of (b). Item (c) follows from (a) and (b). Finally (d) follows from (a) and (b) as limits in norm of compact operators are compact.
\end{proof}
Let us denote,
\beq\label{649}
M_\pm(z):= M_1(z)+M_{2,\pm}(z), \qquad z \in \mathbb C_\pm \cup \Gamma.
\ene
In the following theorem we study the invertibility of $I+M_\pm(z).$
\begin{theorem}\label{invert}
Suppose that Assumption~\ref{assum0} is satisfied, that    $h'_\pm(v^2/2) v^2 $  is  a integrable function of $v\in \mathbb R,$ and that $h_\pm(\lambda)$ have a continuous second derivative that fulfills 
\eqref{6.18qww}. Then for $ \xi < 1/4$ and $ \gamma < \xi/2$ we have.
  
\begin{enumerate}
\item[\rm (a)] For  $ z \in \mathbb C_\pm,$ the operator $I+M(z)$   is invertible  in  $\mathcal H_{\gamma, \xi}.$ and its inverse is a bounded operator.
\item[\rm(b)] Denote,
\beq\label{6.49}
\mathcal N_\pm:=\{ \lambda \in \Gamma : -1  \,\text{ \rm  is an eigenvalue of} \, M_\pm(\lambda) \},
\ene
and
\beq\label{6.50}
\mathcal M_\pm:= \mathcal N_\pm \cup \left[  \cup_{n \in \mathbb Z} \left\{ \frac{2 \pi n}{T_{-,0}(-\varphi_0(0))},  \frac{2 \pi n}{T_{+,0}(\varphi_0(P/2))} \right\}\right].
\ene
Then, $\mathcal M_\pm$ are closed sets of measure zero,  $I+M_\pm (\lambda)$ is invertible for $\lambda \in \Gamma \setminus \mathcal M_\pm= \Gamma \setminus \mathcal N_\pm,$ and its inverse is a bounded operator in $\mathcal H_{\gamma, \xi}.$ 
\end{enumerate}
\end{theorem}
\begin{proof}
We denote,
\beq\label{6.51}
L(z)= I+M_1(z), \qquad z \in \mathbb C\setminus\{0\}.
\ene
By (b) of Theorem~\ref{theom1} $L(z)$ is invertible, with a bounded inverse, in  $\mathcal H_{\gamma, \xi}.$
Further,
\beq\label{6.52}
I+ M_\pm(z)= L(z)(I+ L^{-1}(z) M_{2,\pm}(z)), \qquad z \in \mathbb C_\pm\cup \Gamma.
\ene 
By  (a) of Theorem~\ref{theolim} $ L^{-1}(z)M_2(z)$ is compact for $z \in \mathbb C_\pm.$ Hence, $I+ M(z)$  is invertible and its inverse is a bounded operator in  $\mathcal H_{\gamma, \xi}$ if and only if $-1$ is not an eigenvalue of $M(z).$ However, by the second resolvent equation,
\beq\label{6.56b}
R_{\hat{A}}(z)(I+ M(z))= R_{\hat{A}_0(z)}, \qquad z \in \mathbb C_\pm,  
\ene
 If $M(z)$ would have an eigenvalue minus one,  $R_{\hat{\mathcal A_0}}(z)$  would have an  eigenvalue zero. This is impossible since $\mathbb C_\pm$ belongs to the resolvent set of $\hat{\mathcal A_0}.$
 This proves (a). Let us prove (b).  By (d) of Theorem~\ref{theolim}   $ L^{-1}(\lambda)M_{2,\pm}(\lambda)$ is compact. Then, by\eqref{6.52},   for $\lambda \in \Gamma,$ we have that  $I+M_\pm (\lambda)$ is invertible,  and its inverse is a bounded operator in $\mathcal H_{\gamma, \xi},$ if and only if $ \lambda \in \Gamma \setminus \mathcal N_\pm.$ Let us prove that $\mathcal N_\pm$ has measure zero. It is enough to prove that $\mathcal N_\pm \cap [a,b]$ is of measure zero for  $[a,b]\subset \Gamma.$ By \eqref{6.52}
 $ \lambda \in \mathcal N_\pm$ if and only if $-1$ is an eigenvalue of $ L^{-1}(\lambda)M_{2,\pm}(\lambda).$
 Let us prove that  $ L^{-1}(\lambda)M_{2,\pm}(\lambda)$ can be approximated in norm by finite rank operators.
 By the proof of Lemma 5 in page 121 of \cite{ya}, there is a sequence $\{P_n\}_{n \in \mathbb N}$ of finite rank operators in $\hat{C}_{\rho,0}((-P/2,P/2)),$   such that in the strong operator topology  in $\hat{C}_{\rho,0}((-P/2,P/2)),$  
 $$
 {\rm s}-\lim_{n \to \infty} P_n= I.
 $$
 We denote
 $$
 \mathbf P_n:=\begin{bmatrix} 0&0& 0\\
  0&0&0\\
  0&0&P_n      
\end{bmatrix}.
 $$
 Moreover, as $ L^{-1}(\lambda)$ is a bijection, $L^{-1}(\lambda) \mathbf P_n M_{2,\pm}(\lambda)$ is of finite rank,
 and as $M_{2,\pm}(\lambda)$ is compact,
 $$
 L^{-1}(\lambda)M_{2,\pm}(\lambda)= \lim_{n\to \infty}  L^{-1}(\lambda) \mathbf P_n M_{2,\pm}(\lambda)
 $$
 with the limit in the uniform operator norm in $\mathcal B(\mathcal H_{\gamma,\xi}).$ 
 Further, $ L^{-1}(z)M_{2,\pm}(z)$ is analytic for for $z \in \mathbb C_\pm,$
  continuos for $z \in \mathbb C_\pm\cup [a,b],$ is compact, and can be approximated in norm by finite rank operators.  Moreover,  $I+ L^{-1}(z) M_{2,\pm}(z)$ is invertible for $ z \in \mathbb C_\pm.$ Then, by Theorem 3 and Remark 4 in Section 8 of Chapter 1 of \cite{ya},  $\mathcal N_\pm \cap [a,b]$ is of measure zero. Let us prove that $\mathcal M_\pm$ is closed. Suppose that $\lambda_n$ converges to $\lambda$ where,
   $$
\lambda \in   \cup_{n \in \mathbb Z} \left\{ \frac{2 \pi n}{T_{-,0}(-\varphi_0(0))},  \frac{2 \pi n}{T_{+,0}(\varphi_0(P/2))} \right\}.
$$
  Then, $ \lambda \in \mathcal M_\pm$ by the definition of $\mathcal M_\pm.$  Suppose that   $\lambda_n \in \mathcal M_\pm $ converges to $ \lambda \in \Gamma.$  Then, for $n$ large enough $\lambda_n \in \mathcal N_\pm.$  Let us prove that $ \lambda \in \mathcal N_\pm.$ Otherwise, as  $L^{-1}(\lambda)M_{2,\pm}(\lambda)$ is compact, we would have that $(I+L^{-1}(\lambda)M_{2,\pm}(\lambda) )^{-1}$ would be a bounded operator in $\mathcal H_{\gamma,\xi}.$ However, by the stability of bounded invertibility theorem (Theorem 1.16 in page 196 0f \cite{kato}),  $( I+L^{-1}(\lambda)M_{2,\pm}(\lambda_n))^{-1}$ would be a bounded operator in $\mathcal H_{\gamma,\xi}$
  for $n$ large enough, in contradiction with the assumption that $\lambda_n \in \mathcal N_\pm.$ Here we used that $L^{-1}(\lambda)M_{2,\pm}(\lambda)$ is  continuous in the operator norm of  $\mathcal H_{\gamma,\xi}.$ This proves that $\mathcal M_\pm$ is closed.
\end{proof}
 \subsection{The absolutely continuous spectrum of the Vlasov-Amp\`ere operator}
 In this subsection we  identify the absolutely continuous spectrum of  the Vlasov-Amp\`ere operator $\mathcal A.$ 
By \eqref{6.56b}we have,
  \beq\label{s.66}
 R_{\hat{\mathcal A}}(z)= R_{\hat{\mathcal A}_0}(z)(I+ M(z))^{-1}, \qquad   \qquad z \in  \mathbb C_\pm.
\ene
For   $ f \in {\mathcal H}_{\gamma,\xi}, $ with $\gamma$ and $\xi$ as in Theorem~\ref{invert}, and $\mu \in \mathbb R
\setminus\mathcal M_\pm,$ we define,

\beq\label{s.67}
\ds f_{\mu\pm  i\varepsilon}:=  \left(I+M_\pm(\mu \pm i\varepsilon)\right)^{-1}\,f,  \qquad  \varepsilon \geq 0.
\ene
We define,
\beq\label{s.67b}
J_{\pm, 0,n}= (0, \beta_{\pm, n, {\rm max}}), n \in \mathbb N,\qquad   J_{\pm,0,n}= (-\beta_{\pm, -n, {\rm max}},0), -n \in \mathbb N,
\ene
and
\beq\label{s.67c}
J_{\pm,1,n}= (0, \infty), n \in \mathbb N, \qquad  J_{\pm, 1,n}= (-\infty, 0), -n \in \mathbb N.
\ene
We have:

\begin{proposition}\label{prop.stone}
Suppose that Assumption~\ref{assum0} is satisfied, that    $h'_\pm(v^2/2) v^2 $  is  a integrable function of $v\in \mathbb R,$ and that $h_\pm(\lambda)$ have a continuous second derivative that fulfills 
\eqref{6.18qww}.    Then, for every $f, g \in \mathcal H_{\gamma,\xi}$ with $ 0 < \xi < 1/4, 0 < \gamma < \xi/2,$   and for every $\mu \in\mathbb R
\setminus\mathcal M_\pm,$
\beq\label{s.68}\begin{array}{l} 
\lim_{\varepsilon \downarrow 0}\frac{1}{2\pi i}\left( \left[R_{\hat{\mathcal A}}(\mu+i\varepsilon)- R_{\hat{\mathcal A}}(\mu- i\varepsilon)\right] f, g \right)_{\hat{\mathcal H}}=\sum_{\nu=\pm} \left[ \sum_{n \in \mathbb Z_0}  \ds \chi_{J_{\nu,0,n}}(\mu) \ds T(\mathcal E_{\nu,0,n}(\mu))\right.\\[.3cm] \left. p_{\nu,0,n}(\mu)
 (\mathbf Ff_{\mu \pm i0})_{\nu,0,n}(\mathcal E_{\nu,0,n}(\mu)) \overline{(\mathbf Fg_{\mu\pm i0})_{\nu,0,n}(\mathcal E_{\nu,0,n}(\mu))}+\sum_{\upsilon=\pm}\sum_{n \in \mathbb Z_0}  \ds \chi_{J_{\nu,1,n}}(\mu) \ds \right. \\[.3cm] \left. T(\mathcal E_{\nu,1,n}(\mu)) \, p_{\nu,1,n}(\mu)
 (\mathbf Ff_{\mu \pm i0})_{\nu,1,\upsilon, n}(\mathcal E_{\nu,1,n}(\mu)) \overline{(\mathbf Fg_{\mu\pm i0})_{\nu,1,
 \upsilon}(\mathcal E_{\nu,1,n}(\mu))} \right] 
 \end{array}
\ene
where the limit exists in pointwise sense and in the sense of $L^1([a, b])$ for any closed and bounded interval $[a,b] \in \mathbb R\setminus \mathcal M_\pm.$
 Note that as   $(\mathbf Ff_{\mu \pm i0})_{\nu,0,n}(\mathcal E_{\nu,0,n}(\lambda)), 
 (\mathbf Fg_{\mu \pm i0})_{\nu,0,n}(\mathcal E_{\nu,0,n}(\lambda)),\linebreak (\mathbf Ff_{\mu\pm i0})_{\nu,1,\upsilon,n}(\mathcal E_{\nu,1,n}(\lambda)),$ and 
 $(\mathbf Fg_{\mu \pm i0})_{\nu,1,\upsilon,n}(\mathcal E_{\nu,1,n}(\lambda))$ 
  are   jointly  continuous in $\mu$ and $\lambda $  for $\nu=\pm,$ and $\upsilon=\pm,$ we can take the traces at   $ \lambda=\mu$  in \eqref{s.68}. The unitary operator $\mathbf F$ is defined in \eqref{5.84f}, \eqref{5.84h}. Further, for the definition of $ \mathcal E_{\pm, \xi, n}(\mu),$ and $p_{\pm, \xi,n}(\mu), $ for $ \xi=0,1,$
see equations \eqref{5.51}-\eqref{5.57}.
\end{proposition}
\begin{proof} By the first resolvent equation,
\beq\label{s.69a}
 R_{\hat{\mathcal A}}(\mu+i\varepsilon)-R_{\hat{\mathcal A}}(\mu-i\varepsilon)= 2i \varepsilon  R_{\hat{\mathcal A}}(\mu+i\varepsilon)
 R_{\hat{\mathcal A}}(\mu-i\varepsilon).
\ene
Then, by  \eqref{5.90}, \eqref{5.91}, \eqref{s.66}, \eqref{s.67}, and \eqref{s.69a},
\beq\label{s.69}\begin{array}{l}
\frac{1}{2\pi i}\left( \left[R_{\hat{\mathcal A}}(\mu+i\varepsilon)- R_{\hat{\mathcal A}}(\mu- i\varepsilon)\right] f, g \right)_{\hat{\mathcal H}}=    \sum_{\nu=\pm }  \ds \int_{I_{\nu,0}} \ds T(\mathcal E_{\nu,0}) \ds\frac{\varepsilon}{\pi}\frac{1}{(\mu)^2+\varepsilon^2}
 (\mathbf Ff_{\mu \pm i \varepsilon})_{\nu,0,0}(\mathcal E_{\nu,0}) \\\overline{(\mathbf Fg_{\mu\pm i \varepsilon})_{\nu,0,0}(\mathcal E_{\nu,0})}  d  \mathcal E_{\nu} +  \sum_{\nu=\pm } \sum_{\upsilon=\pm} \ds \int_{I_{\nu,1}} \ds T(\mathcal E_{\nu,1}) \ds\frac{\varepsilon}{\pi}\frac{1}{(\mu)^2+\varepsilon^2}
 (\mathbf Ff_{\mu \pm i \varepsilon})_{\nu,1,\upsilon,0}(\mathcal E_{\nu,1}) \\\overline{(\mathbf Fg_{\mu\pm i \varepsilon})_{\nu,1,\upsilon,0}(\mathcal E_{\nu,1})}  d  \mathcal E_{\nu,1} +
 \sum_{\nu=\pm} \sum_{ n \in \mathbb Z_0}  \ds \int_{J_{\nu,0,n}} \ds T(\mathcal E_{\nu,0,n}(\beta))   p_{\nu,0,n}(\beta) \ds\frac{\varepsilon}{\pi}\frac{1}{(\mu-\beta)^2+\varepsilon^2}\\
 (\mathbf Ff_{\mu \pm i \varepsilon})_{\nu,0,n}(\mathcal E_{\nu,0,n}(\beta)) \overline{(\mathbf Fg_{\mu\pm i \varepsilon})_{\nu,0,n}(\mathcal E_{\nu,0,n}(\beta))}  d \beta+  
 \sum_{\nu=\pm} \sum_{ \upsilon=\pm} \sum_{ n\in \mathbb Z_0}  \ds \int_{J_{\nu,1,n}} \ds T(\mathcal E_{\nu,1,n}(\beta))
  \\   p_{\nu,1,n}(\beta) \ds\frac{\varepsilon}{\pi}\frac{1}{(\mu-\beta)^2+\varepsilon^2}
 (\mathbf Ff_{\mu \pm i \varepsilon})_{\nu,1,\upsilon,n}(\mathcal E_{\nu,1,n}(\beta)) \overline{(\mathbf Fg_{\mu\pm i \varepsilon})_{\nu,1,\upsilon,n}(\mathcal E_{\nu,1,n}(\beta))}  d \beta.  
\end{array}
\ene
Moreover, by  \eqref{s.69} and Theorems 7  and 8 in Subsection 3 of Section 2 of Chapter 1 of \cite{ya}
\beq\label{s.71}\begin{array}{l}
\lim_{\varepsilon \downarrow 0}\frac{1}{2\pi i}\left( \left[R_{\hat{\mathcal A}}(\mu+i\varepsilon)- R_{\hat{\mathcal A}}(\mu- i\varepsilon)\right] f, g \right)_{\mathcal H}= \sum_{\nu=\pm} \left[ \sum_{n \in \mathbb Z_0}  \ds \chi_{J_{\nu,0,n}}(\mu) \ds T(\mathcal E_{\nu,0,n}(\mu))\right.\\ [.3cm]\left. p_{\nu,0,n}(\mu)
 (\mathbf Ff_{\mu \pm i0})_{\nu,0,n}(\mathcal E_{\nu,0,n}(\mu)) \overline{(\mathbf Fg_{\mu\pm i0})_{\nu,0,n}(\mathcal E_{\nu,0,n}(\mu))}+\sum_{\upsilon=\pm}\sum_{n \in \mathbb Z_0}  \ds \chi_{J_{\nu,1,n}}(\mu) \ds T(\mathcal E_{\nu,+,n}(\mu))\right. \\[.3cm] \left. p_{\nu,1,n}(\mu)
 (\mathbf Ff_{\mu \pm i0})_{\nu,1,\upsilon, n}(\mathcal E_{\nu,1,n}(\mu)) \overline{(\mathbf Fg_{\mu\pm i0})_{\nu,1,
 \upsilon,n}(\mathcal E_{\nu,1,n}(\mu))} \right] 
\end{array}
\ene
where the limit exists in pointwise sense and in the sense of $L^1([a, b])$ for any closed and bounded interval $[a,b] \in \mathbb R\setminus \mathcal M_\pm.$
 This concludes the proof of the proposition.
\end{proof}

For a selfadjoint operator  $B$ in a Hlbert space we denote by  $E_B(O)$ the spectral projector for a Borel set $O.$ Further, we designate by $\mathcal H_{\rm ac}(B)$ the absolutely continuous subspace of $B.$
\begin{theorem} \label{theoproj}
Suppose that Assumption~\ref{assum0} is satisfied, that    $h'_\pm(v^2/2) v^2 $  is  a integrable function of $v\in \mathbb R,$ and that $h_\pm(\lambda)$ have a continuous second derivative that fulfills 
\eqref{6.18qww}.   Then.
  \begin{enumerate}
\item[\rm (a)] We have, $\mathcal M_+=\mathcal M_-:=\mathcal M,$  where $\mathcal M_\pm$ are the  closed sets of measure zero  defined in Theorem~\ref{invert}. 
\item[\rm (b)]
The projector $P_{\rm ac}(\hat{\mathcal A})$ onto the subspace of absolute continuity of the Vlasov-Amp\`ere operator$\hat{\mathcal A} $ is given by,
 \beq\label{s.72}
 P_{\rm ac}(\hat{\mathcal A})=   E_{\hat{\mathcal A}}(\mathbb R \setminus \mathcal M)                        .   
 \ene

 \item[\rm(c)]
 The absolutely continuous spectrum of $\hat{\mathcal A}$ is given by,
 \beq\label{s.72b}
 \sigma_{\rm ac}(\hat{\mathcal A})=   \mathbb R.
\ene
\item[\rm (d)]
The  singular spectrum of $\hat{\mathcal A}$ is contained in    $\mathcal M$
\beq\label{2.72c}
\sigma_{\rm sing}(\hat{\mathcal A}) \subset \mathcal M.
\ene
We recall that the singular spectrum of $\hat{\mathcal A}$  is the union of the closure of the set of all eigenvalues of   $\hat{\mathcal A}$  with  the singular continuous spectrum of $\hat{\mathcal A}.$
 \end{enumerate}
 \end{theorem}
 \begin{proof}
We first prove (a) and (b).   Let $f, g \in {\mathcal H}_{\gamma,\xi},$with   $ \gamma <\xi/2$ and $0 < \xi  <  1/4.$   Then, by Stone's formula (see Theorem VII.13 and the comment in page 264 of \cite{rs1}) for any open interval   
 $(a,b) \subset \mathbb R  \setminus \mathcal M_\pm $ 
 \beq\label{s.73}
 \frac{1}{2}\left( E_{\hat{\mathcal A}}([a,b]) f,g \right)_{\hat{\mathcal H}}+ \frac{1}{2}\left( E_{\hat{\mathcal A}}((a,b)) f,g \right)_{\hat{\mathcal H}}= \lim_{\varepsilon \downarrow 0} \frac{1}{2\pi i} \int_{(a,b)} \, d\mu \left( \left[R_{\hat{\mathcal A}}(\mu+i\varepsilon)- R_{\mathcal A}(\mu- i\varepsilon)\right] f, g \right)_{\hat{\mathcal H}}.
 \ene
Introducing the limit \eqref{s.68} into  the integral in \eqref{s.73}  we get,
 \beq\label{s.74b}\begin{array}{l}
  \frac{1}{2}\left( E_{\hat{\mathcal A}}([a,b]) f,g \right)_{\hat{\mathcal H}}+ \frac{1}{2}\left( E_{\hat{\mathcal A}}((a,b)) f,g \right)_{\hat{\mathcal H}}=   \int_{(a,b)}\left[ 
   \sum_{\nu=\pm} \left[ \sum_{n \in \mathbb Z_0}  \ds \chi_{J_{\nu,0,n}}(\mu) \ds T(\mathcal E_{\nu,0,n}(\mu))
     p_{\nu,0,n}(\mu) \right.\right.\\[.3cm]\left..\left.
 (\mathbf Ff_{\mu \pm i0})_{\nu,0,n}(\mathcal E_{\nu,0,n}(\mu))\overline{(\mathbf Fg_{\mu\pm i0})_{\nu,0,n}(\mathcal E_{\nu,0,n}(\mu))}+\sum_{\upsilon=\pm}\sum_{n \in \mathbb Z_0}  \ds \chi_{J_{\nu,1,n}}(\mu) \ds \right. \right.\\[.3cm]\left. \left.T(\mathcal E_{\nu,1,n}(\mu)) p_{\nu,1,n}(\mu)
 (\mathbf Ff_{\mu \pm i0})_{\nu,1,\upsilon, n}(\mathcal E_{\nu,1,n}(\mu)) \overline{(\mathbf Fg_{\mu\pm i0})_{\nu,1,
 \upsilon,n}(\mathcal E_{\nu,1,n}(\mu))} \right]\right]  \, d\mu.
 \end{array}
\ene
Taking the limit as $ a \uparrow b$ in \eqref{s.74b} we get that $E_{\hat{\mathcal A}}(b)=0,$ for $ b \in \mathbb R
\setminus \mathcal M_\pm.$ Hence $ \hat{\mathcal A}$ has no eigenvalues in $\mathbb R
\setminus \mathcal M_\pm.$  Then, from \eqref{s.74b} we obtain,
 \beq\label{s.74}\begin{array}{l}
 \left( E_{\hat{\mathcal A}}((a,b)) f,g \right)_{\hat{\mathcal H}}=   \int_{(a,b)}\left[ 
   \sum_{\nu=\pm} \left[ \sum_{n \in \mathbb Z_0}  \ds \chi_{J_{\nu,0,n}}(\mu) \ds T(\mathcal E_{\nu,0,n}(\mu))
     p_{\nu,0,n}(\mu) \right.\right.\\[.3cm]\left..\left.
 (\mathbf Ff_{\mu \pm i0})_{\nu,0,n}(\mathcal E_{\nu,0,n}(\mu))\overline{(\mathbf Fg_{\mu\pm i0})_{\nu,0,n}(\mathcal E_{\nu,0,n}(\mu))}+\sum_{\upsilon=\pm}\sum_{n \in \mathbb Z_0}  \ds \chi_{J_{\nu,1,n}}(\mu) \ds \right. \right.\\[.3cm]\left. \left.T(\mathcal E_{\nu,1,n}(\mu)) p_{\nu,1,n}(\mu)
 (\mathbf Ff_{\mu \pm i0})_{\nu,1,\upsilon, n}(\mathcal E_{\nu,1,n}(\mu)) \overline{(\mathbf Fg_{\mu\pm i0})_{\nu,1,
 \upsilon,n}(\mathcal E_{\nu,1,n}(\mu))} \right]\right]  \, d\mu.
 \end{array}
\ene
Since the measure in the right-hand side of \eqref{s.74} is absolutely continuous, we have $ E_{\hat{\mathcal A}}((a,b)) f \subset {{\mathcal H}}_{\rm ac}(\hat{\mathcal A})$ for $f \in \mathcal H_{\gamma,\xi}.$  As $\mathcal H_{\gamma,\xi}$  is dense in $\hat{\mathcal H}$ and as  $\mathcal H_{\rm ac}(\hat{\mathcal A})$ is closed, we get
 \beq\label{s.75}
 E_{\hat{\mathcal A}}((a,b)) \hat{\mathcal H} \subset \mathcal H_{\rm ac}(\hat{\mathcal A}),
 \ene 
 for any open interval  $ (a,b) \subset \mathbb R \setminus \mathcal M_\pm. $ But then, \eqref{s.75} follows for any Borel set 
$\Delta \subset \mathbb R \setminus \mathcal M_\pm.$ Moreover, as $\mathcal M_\pm$  are closed sets of measure zero they can not support absolutely continuous spectrum, and then, by \eqref{s.75}
we get,
 \beq\label{s.75a}
 P_{\rm ac}(\hat{\mathcal A})= E_{\hat{\mathcal A}}(\sigma_{\rm ess}(\hat{\mathcal A}) \setminus \mathcal M_\pm)
= E_{\hat{\mathcal A}}(\mathbb R \setminus \mathcal M_\pm).
 \ene
 Let us denote,
 $$
 A_\pm:= \mathbb R \setminus \mathcal M_\pm.
 $$
 Then, by \eqref{s.75a},
 $$
  E_{\hat{\mathcal A}}(A_+)=  E_{\hat{\mathcal A}} (A_+\cap A_-)+  E_{\hat{\mathcal A}}(A_+\setminus A_-)=
  E_{\hat{\mathcal A}}(A_-)=   E_{\hat{\mathcal A}}(A_-\cap A_+)+  E_{\hat{\mathcal A}}(A_-\setminus A_+).
 $$
 It follows that,
 $$
   E_{\hat{\mathcal A}}(A_+\setminus A_-)=   E_{\hat{\mathcal A}}(A_-\setminus A_+).
 $$
 We denote $ B_+:= A_+\setminus A_-$ and $B _- :=A_-\setminus A_+.$  Since  $B_+\cap B_-= \emptyset,$   we have,
 $$
    E_{\hat{\mathcal A}}(B_+\cup B_-)=    E_{\hat{\mathcal A}}(B_+)+  E_{\hat{\mathcal A}}(B_-)=
    2  E_{\hat{\mathcal A}}(B_+).
 $$
 But then, $ E_{\hat{\mathcal A}}(B_+\cup B_-)=  E_{\hat{\mathcal A}}(B_+\cup B_-)^2= 2 E_{\hat{\mathcal A}}(B_+\cup B_-).$ However, this  is only possible if $E_{\hat{\mathcal A}}(B_+\cup B_-)=0.$ Hence,  $B_+\cup B_-$ belongs to the complement of the support of the spectral measure of $\hat{\mathcal A}.$  Further, as the spectrum of $\hat{\mathcal A}$ is $\mathbb R,$ the support of the spectral measure of $\hat{\mathcal A}$ is $\mathbb R.$ It follows that $B_+\cup B_-= \emptyset.$ Hence
 $B_\pm=\emptyset$ and $\mathcal M_+=\mathcal M_-.$ This proves (a) and  (b) follows from (a) and \eqref{s.75a}.  Let us prove (c).
 We already know that $\mathbb R \setminus \mathcal M \subset \sigma_{\rm ac}(\hat{\mathcal A}).$ Consider $ \beta \in\mathcal M.$ Suppose that $ \beta$ is not a limit point of $\mathbb R \setminus \mathcal M .$ Then, there  would be a neighborhood  of $\beta$ that contains no point of $\mathbb R \setminus \mathcal M.$ This would imply that there is a an open set of positive measure contained in $\mathcal M.$ However, this is
not possible since $\mathcal M$ has measure zero. Hence, every point in $ \mathcal M$ is a limit point of    $\mathbb R \setminus \mathcal M \subset \sigma_{\rm ac}(\hat{\mathcal A}),$ and as the absolutely continuous spectrum is closed, $\mathcal M \subset  \sigma_{\rm ac}(\hat{\mathcal A}).$ This proves (c). Let us prove  (d). We denote by $P_{\rm s}(\hat{\mathcal A})$ the  proyector onto the singular subspace of $\hat{\mathcal A}.$ Then, by \eqref{s.72},   $P_{\rm s}(\hat{\mathcal A})=
E_{\hat{\mathcal A}}(\mathcal M)$ and then, (d) follows.

 \end{proof}
\section{The generalized Fourier maps}
\label{fouma}
In this section we construct the generalized Fourier maps of the Vlasov-Amp\`ere  operator. 
We define first the generalized Fourier maps  $\mathcal F_\pm$    for finite linear combinations of the type,
\beq\label{s.75b}
f= \sum_{j=1}^N E_{\hat{\mathcal A}}(O_j) f^{(j)},
\ene
 where  $f_j \in \mathcal H_{\gamma,\xi},$  with $ \gamma$ and $\xi$ as in Proposition~\ref{prop.stone},       and the $O_j$ are  disjoint intervals, $O_j \cap O_k=\emptyset $ for $j \neq k,$  with $O_j  \subset \sigma_{\rm ac}(\mathcal A)\setminus \mathcal M,$ for $j=1,\dots, N.$  Recall that  $\mathcal H_{\gamma,\xi}$  was defined in \eqref{6.18p} and \eqref{6.18q}. The set of all the functions as in \eqref{s.75b} is dense in $\mathcal H _{\rm ac}(\hat{\mathcal A}).$  For $f$ as in \eqref{s.75b} we define the $\mathcal F_\pm$ as operators from $\hat{\mathcal H}$ into $\tilde{\mathcal H}_0$ as follows (recall that $\tilde{\mathcal H}_0$  was defined in  \eqref{5.94})
 
 \beq\label{s.75c}
\mathcal F_\pm f =\begin{pmatrix} \{\mathcal F_\pm  f\}_{-,n}\\
   \{\mathcal F_\pm f\}_{+,n}  \end{pmatrix}, \qquad n \in \mathbb Z_0,
\ene
where
\beq\label{s.75cc}
 \{\mathcal F_\pm  f\}_{\mp,n}= \left\{\begin{array}{c}  \{\mathcal  F_\pm  f\}_{\mp,0,n}\\
 \{\mathcal  F_\pm  f\}_{\mp,1,+,n}\\\{\mathcal  F_\pm  f\}_{\mp,1,-,n}\end{array}
\right\},\qquad n \in \mathbb Z_0.
\ene
Moreover,  for $\nu=\pm$
\beq\label{s.75e}\begin{array}{l}
  \{\mathcal  F_\nu  f\}_{\mp,0,n}:= \sum_{j=1}^N \chi_{O_j}(\beta)  \chi_{J_{\mp,0,n}}(\mu) \ds \sqrt{T(\mathcal E_{\mp,0,n}(\beta)) |p_{\nu,0,n}(\beta)|}
 (\mathbf Ff^{(j)}_{\beta +\nu\, i0})_{\mp,0,n}(\mathcal E_{\mp,0,n}(\beta)),\\
\{\mathcal  F_\nu  f\}_{\mp,1,+,n}:=\sum_{j=1}^N \chi_{O_j}(\beta) \chi_{J_{\mp,1,n}}(\beta) \ds \sqrt{T(\mathcal E_{\mp,1,n}(\beta)) |p_{\mp,1,n}(\beta)|}
 (\mathbf Ff^{(j)}_{\beta +\nu \, i0})_{\mp,1,+, n}(\mathcal E_{\nu,1,n}(\beta)), \\
 \{\mathcal  F_\nu  f\}_{\mp,1,-,n}:=\sum_{j=1}^N \chi_{O_j}(\beta) \chi_{J_{\mp,1,n}}(\beta) \ds \sqrt{T(\mathcal E_{\mp,1,n}(\beta)) |p_{\mp,1,n}(\beta)|}
 (\mathbf Ff^{(j)}_{\beta + \nu\, i0})_{\mp,1,-, n}(\mathcal E_{\nu,1,n}(\beta)).
 \end{array}
\ene
 Note that  by \eqref{s.75c}-\eqref{s.75e}, for any interval $\Delta  \subset \mathbb R\setminus \mathcal M ,$ and any $f$ as in \eqref{s.75b},
\beq\label{s.75d}
\chi_\Delta (\beta)( \mathcal F_\pm f)(\mu)= (\mathcal F_\pm E_{\hat{\mathcal A}}(\Delta) f)(\beta).
\ene
Then, for every $f,g \in \mathcal H_{\gamma,\xi}$  we can write \eqref{s.74} as
\beq\label{s.76}
\left( E_{\hat{\mathcal A}}((a,b)) f,g \right)_{\hat{\mathcal H}}= \left(\chi_{(a,b)} \mathcal F_\pm f, \mathcal F_\pm g\right)_{\tilde{\mathcal H}_{0}}.
\ene
As \eqref{s.76} holds for every  interval, it also holds for every Borel set, i.e.
\beq\label{s.77}
\left( E(O) f,g \right)_{\hat{\mathcal H}}= \left(\chi_{O} \mathcal F_\pm f, \mathcal F_\pm g\right)_{\tilde{\mathcal H}_{0}},
\ene
for every Borel set $ O \ \in \sigma_{\rm ac}(\hat{\mathcal A})\setminus \mathcal M .$  Further, \eqref{s.77} extends by linearity to all functions of the form \eqref{s.75b}.
In  particular,
 \beq\label{s.78}
\
\left( P_{\rm ac}(\hat{\mathcal A}) f,g \right)_{\hat{\mathcal H}}= = \left(\chi_{(\sigma_{\rm ac}(\hat{\mathcal A})\setminus \mathcal M)} \mathcal F_\pm f, \mathcal F_\pm g\right)_{\tilde{\mathcal H}_{ 0}}=
 \left( \mathcal F_\pm f, \mathcal F_\pm g\right)_{\tilde{\mathcal H}_{0}}.
\ene
Note that \eqref{s.78} implies that the operators $\mathcal F_\pm$ are densely defined and bounded from 
$\mathcal H_{\rm ac}(\hat{\mathcal A})$ into $\tilde{\mathcal H}_{0}.$ We extend $\mathcal F_\pm$ by continuity to bounded operators from $\hat{\mathcal H}_{\rm ac}(A)$ into
$\tilde{\mathcal H}_{o}$ and we denote the extensions by the same symbol  $\mathcal F_\pm.$ Further, we define $\mathcal F_\pm$ by zero on the singular subspace of $\hat{\mathcal A},$
\beq\label{s.79}
\mathcal F_\pm f=0, \qquad f \in \mathcal H_{\rm s}(\hat{\mathcal A}).
\ene
Then, by \eqref{s.78} and \eqref{s.79} the generalized Fourier maps $\mathcal F_\pm$ are partially isometric from $\hat{\mathcal H}$ into
$\tilde{\mathcal H}_{0}$ with initial subspace $\mathcal H_{\rm ac}(\hat{\mathcal A}).$  In the theorem below we prove that they are actually onto $\tilde{\mathcal H}_{0}.$ For this purpose, as mentioned introduction, the spectral representation of $\hat{\mathcal A}_0$ in terms of { \it trace maps}  \eqref{5.103}-\eqref{5.106} plays a crucial role.
The fact that  the $\mathcal F_\pm$ are onto  $\tilde{\mathcal H}_{0}$ is the key point in the proof of the completeness of the wave operators in Theorem~\ref{waop}. 
\begin{theorem} \label{theogf} Suppose that Assumption~\ref{assum0} is satisfied, that    $h'_\pm(v^2/2) v^2 $  is  a integrable function of $v\in \mathbb R,$ and that $h_\pm(\lambda)$ have a continuous second derivative that fulfills  \eqref{6.18qww}.  Then, the generalized Fourier maps $\mathcal F_\pm$ are partially isometric from $\hat{\mathcal H}$ onto $\tilde{\mathcal H}_{0}.$ Furthermore, the  initial subspace of  $\mathcal F_\pm$  is $\mathcal H_{\rm ac}(\hat{\mathcal A})$ and the final subspace is $\tilde{\mathcal H}_{0},$
\beq\label{s.80}
\mathcal F_\pm^\ast \mathcal F_\pm= P_{\rm ac}(\hat{\mathcal A}),\qquad  \mathcal F_\pm \mathcal F_\pm^\ast = I.
\ene
Moreover, for any  Borel function $\phi,$
\beq\label{s.81}
\phi(\hat{\mathcal A}_{\rm ac}) = \mathcal F_\pm ^\ast \, \phi(\mu) \, \mathcal F_\pm,
\ene
where $\hat{\mathcal A}_{\rm ac}$ is the absolutely continuous part of $\hat{\mathcal A},$ and by $\phi(\mu)$ we denote the operator of multiplication by the function  $\phi(\mu).$
\beq\label{s.82}
\phi(\mu) \{f_l\}=  \{  \varphi(\mu) f_l(\mu)\}.
\ene
In particular,
\beq\label{s.82b}
\hat{\mathcal A}_{\rm ac}= \mathcal F_\pm ^\ast \,\mu \, \mathcal F_\pm,
\ene
that it to say, $\mathcal F_\pm$ diagonalize $\mathcal A_{\rm ac}.$

\end{theorem}
\begin{proof}
We already know that the generalized Fourier maps $\mathcal F_\pm$ are partially isometric from $\hat{\mathcal H}$ into $\tilde{\mathcal H}_{o}$ with initial subspace $\mathcal H_{\rm ac}(\hat{\mathcal A}).$ It remains to prove that they are onto  $\tilde{\mathcal H}_{0}.$ For this purpose, note that  by   \eqref{5.103}-\eqref{5.106} for any closed interval  $\Delta \subset \mathbb R \setminus \mathcal M $ that satisfies \eqref{5.100} and  for  $f \in \mathcal H_{\gamma,\xi},$ with $\gamma < \xi/2$ and $ \xi <  1/4.$
\beq\label{s.83}
\left(\mathcal F_\pm E_{\hat{\mathcal A}} (\Delta) f\right)(\beta)= \chi_{\Delta}(\beta) \mathcal L_{j_-,j_+}(\beta) \mathbf F(I+M_\pm(\beta))^{-1} \,f.
\ene
Let us prove that for almost every $ \beta $ in ${\Delta}$ 
\beq\label{s.83c}
\hbox{\rm closure}\left( \mathcal L_{j_-,j_+}(\beta) \mathbf F(I+M_\pm(\beta))^{-1}  \mathcal H_{\gamma,\xi}\right) = Q_{j_-,j_+},
 \ene
 where, if  $\Delta \subset (0,\infty)$
 \beq\label{s.83ca}
 Q_{j_-,j_+}:= l_{j_-}\oplus l_{1}\oplus l_{1}\oplus l_{j_+}\oplus l_1\oplus l_1,
 \ene
 and if $ \Delta \subset (-\infty,0)$
  \beq\label{s.83cb}
 Q_{j_-,j_+}:= l_{-j_-}\oplus l_{-1}\oplus l_{-1}\oplus l_{-j_+}\oplus l_{-1}\oplus l_{-1}.
 \ene
 
   Since   $ p_{\mp,\upsilon, n}$ and  T$_{\mp,\upsilon},$ for $\upsilon=0,1$  are different from zero,  and and as $I+M_\pm(\beta)$ is a bijection in $\mathcal H_{\gamma,\xi}$ it is enough to prove that

\beq\label{s.83ez}
\hbox{\rm closure} \left(\mathcal L_{j_-,j_+,r}(\beta) \mathbf F \mathcal H_{\gamma,\xi} \right)=Q_{j_-,j_+},
 \ene
for almost every $ \beta$ in ${\Delta},$ and where $\mathcal L_{j_-,j_+,r}(\beta)$ is the operator, 
\beq\label{s.84}
\mathcal L_{j_-,j_+,r}(\beta) B=\begin{pmatrix} (\mathcal L_{j_-,j_+,r}(\beta) B)_-\\(\mathcal L_{j_-,j_+,r}(\beta) B)_+
\end{pmatrix},
\ene
where, if $ \Delta \subset (0,\infty),$  
\beq\label{s.84b}
  (\mathcal L_{j_-,j_+,r}(\beta) B)_\mp  := \begin{pmatrix}  \{  b_{\mp,0,n_\mp}(\beta) \}, n_\mp= j_\mp, j_\mp+1, \dots \\
      \{   b_{\mp,1, +,n}(\beta)\}, n=1,2,\dots\\  \{b_{\mp,1,-,n}(\beta)\}, n=1,2,\dots
   \end{pmatrix},
   \ene
   and if $ \Delta \subset (-\infty,0),$  
\beq\label{s.84c}
  (\mathcal L_{j_-,j_+,r}(\beta) B)_\mp  := \begin{pmatrix}  \{  b_{\mp,0,n_\mp}(\beta) \}, n_\mp= -j_\mp,  -j_\mp-1, \dots\\
      \{  b_{\mp,1,+,n}(\beta)\}, n=-1,-2,\dots\\  \{ b_{\mp,1,-,n}(\beta)\}, n=-1,-2,\dots
   \end{pmatrix}.
   \ene
   We introduce the following Hilbert space that is appropriate for our purposes,
   \beq\label{6.84c1}
{\mathcal H}_{\rm H}:= {\mathcal H}_{-,\rm H}\oplus {\mathcal H}_{+,\rm H} \oplus  H^{1,0}((-P/2, P/2)),
\ene
where
\beq\label{6.84c2}\begin{array}{l}
{\mathcal H}_{\mp, \rm H}:=L^2(I_{\mp,0},L^2(S_1);T_{\mp,0}(\mathcal E_\mp) \,d\mathcal E_\mp )\cap H_{1}(I_{\mp,0}, L^2(S_1))\oplus\\ 
L^2(I_{\mp,1}, L^2(S_1); T_{\mp,1}(\mathcal E_\mp)\, d\mathcal E_\mp) \cap H_{1}(I_{\mp,1}, L^2(S_1)) \oplus\\ [0.2cm] L^2(I_{\mp,1}, L^2(S_1); T_{\mp,1}(\mathcal E_\mp)\, d\mathcal E_\mp ) \cap H_{1}(I_{\mp,1}, L^2(S_1)).
\end{array}
\ene
Clearly, $ \mathcal  H_{\rm H} \subset \mathcal H_{\gamma, \xi}$ and the imbedding is continuous. 
We denote,
\beq\label{s.84d}
D_{j_-,j_+}(\beta):=  L_{j_-,j_+,r}(\beta) \mathbf F,\qquad D[D_{j_-,j_+}(\beta)]= \mathcal H_{\rm H}.
\ene
Remark that  $D_ {j_-,j_+}(\beta)$ is bounded from $\mathcal H_{\rm H}$ into $Q_{j_-,j_+}$ and that it is continuous in $\beta$ in operator norm. 
Then, to prove \eqref{s.83ez} it is enough to prove, 
\beq\label{s.84e}
\hbox{\rm closure}\, \left({\rm Range }\,D_{j_-,j_+}(\beta)\right)= Q_{j_-,j_+}
\ene
 for almost every $\beta$ in $\Delta.$ Let $P(\beta)$ be the projector onto the closure of ${\rm Range } \, D_{j_-,j_+}(\beta).$ We denote  by  $e_j, j=1,2,\dots,$ an orthonormal basis of $Q_{j_-,j_+}.$ Assume that  \eqref{s.84e}  does not hold for almost every $\beta$  in ${\Delta}.$ Then, there has to be  a $j$ such that
  $$
  g_j(\beta):= (I-P(\beta))e_j,
   $$
is different from zero in a set of positive measure $ O \subset {\Delta}.$ Let us prove that $g_j$ is a measurable function. For this purpose, we prove that $ I-P(\beta)$ is strongly measurable. 
Remark  that  $ I-P(\beta)$ is the projector onto the kernel of $D(\beta)^\ast$ as an operator from $Q_{j_-,j_+}$ into $ \mathcal H_{\rm H}.$ Let us denote 
$$
K(\beta):= D(\beta) D(\beta)^\ast.
$$
By Stone's formula  (see Theorem VII.13, and the comment in page 264 of \cite{rs1})  
$$
( I-P(\beta)) f = E_{K(\beta)}(\{0\})f=  \lim_{\varepsilon \downarrow 0} \frac{1}{\pi i} \int_{-1}^0 [R_{K(\beta)}(\lambda+i\varepsilon)-
R_{K(\beta)}(\lambda-i\varepsilon)]f d\lambda , \qquad f \in  Q_{j_-,j_+}.
 $$ 
It follows that  $ I-P(\beta)$ is  strongly measurable. Hence,  $g_j$ is measurable, and then, $g_j \in L^2({\Delta}, Q_{j_-,j_+}).$ Moreover, as $g_j(\beta)$ is orthogonal to the closure of  ${\rm Range } \,D_{j_-,j_+}(\beta)$ for almost every $\beta,$
$$
\left( g_j, \mathcal L_{j_-,j_+,r}(\beta) \mathbf F f\right)_{L^2({\Delta}, Q_{j_-,j_+})}=0, \qquad f \in  \mathcal H_{\rm H}.
$$
Further, as  the set $ \mathcal L_{j_-,j_+,r}(\beta) \mathbf F f,$ with $f \in \mathcal H_{\rm H}
$ is dense in $L^2({\Delta}, Q_{j_-,j_+}),$ the function $g_j(\beta)$ is zero for almost every $ \beta \in \Delta.$ This contradiction proves that 
\eqref{s.84e} holds and then, \eqref{s.83c} is also valid.
 Suppose now that some $ h \in \tilde{\mathcal H}_{0}$ is orthogonal to the range  of $\mathcal F_\pm.$ Let  $\{f_l\}_{l=1}^\infty$ be an orthonormal basis in $ \mathcal H_{\rm H}.$ Then, for every  interval  $\tilde{\Delta}$
  contained in the interior of $\Delta,$ 
\beq\label{s.86}
\left( \mathcal F_\pm E_{\hat{\mathcal A}}(\tilde{\Delta}) f_l,h \right)_{\tilde{\mathcal H}_{0}}= \int_{\tilde{\Delta}} \left( \mathcal L_{j_-,j_+}(\beta) \mathbf F(I+M_\pm(\beta))^{-1}  f_l,  h\right)_{ Q_{j_-,j_+}}=0.
\ene
Hence,
$$\left(  \mathcal L_{j_-,j_+}(\beta) \mathbf F (I+M_\pm(\beta))^{-1}   f_l,  h\right)_{Q_{j_-,j_+}}=0, \qquad  \text{\rm a.e} \, \beta \in \Delta,   l=1,\dots
$$
and by \eqref{s.83c}
$h(\beta)=0,$  for almost every  $\beta \in \Delta $. As this holds for every interval   $\Delta\in \mathbb R \setminus \mathcal M $ that satisfies \eqref{5.100} we get that $h=0.$ 
This completes the proof that the $\mathcal F_\pm$ are onto $\tilde{\mathcal H}_{0}$ and that \eqref{s.80} holds. 
Furthermore, by \eqref{s.68},
\eqref{s.75c}-\eqref{s.75d},  Stone's formula (see Theorem VII.13, and the comment in page 264 of \cite{rs1}),  the spectral theorem (see Section five of Chapter six of \cite{kato} and Theorem VIII.6 in page 263 of \cite{rs1})  for any intervals
 $\Delta_j\in \mathbb R\setminus \mathcal M,$ for $j=1,2,$  any $f_j\in \mathcal H_{\gamma,\xi},$  and for any bounded Borel  function $\phi$
\beq\label{s.86b}
\left( \phi(\hat{\mathcal A}) E_{\hat{\mathcal A}}(\Delta_1) f_1,   E_{\hat{\mathcal A}}(\Delta_2) f_2\right)_{\hat{\mathcal H}}=
\left( \phi(\cdot) \mathcal F_\pm   E_{\hat{\mathcal A}}(\Delta_1) f_1, \mathcal F_\pm E_{\hat{\mathcal A}}(\Delta_2)f_2\right)_{\tilde{\mathcal H}_{0}}.
\ene
Moreover, since the set of functions \eqref{s.75b} is dense in $\mathcal H_{\rm ac}(\hat{\mathcal A}),$ we have
\beq\label{s.87b}
\left( \phi(\hat{A})  f_1,  f_2\right)_{\tilde{\mathcal H}}=
\left( \phi(\cdot) \mathcal F_\pm f_1, \mathcal F_\pm f_2\right)_{\tilde{\mathcal H_0}}, \qquad f_1, f_2 \in \mathcal H_{\rm ac}(\hat{\mathcal A}).
\ene
Equation \eqref{s.81} in the case where $\phi$ is  bounded follows from \eqref{s.87b}. In the general case where $\phi $ is not bounded,  \eqref{s.81} follows from \eqref{s.81} in the bounded case, approximating $\phi$ by 
$ \phi_n$ where $ \phi_n(\lambda)= \phi(\lambda)$ if $|\phi(\lambda)| \leq n$ and $ \phi_n(\lambda)=0,$ if 
$|\phi(\lambda)|  >n,$ for $n=1,\dots.$
\end{proof}

\section{The wave operators}\label{wave}
In this section we prove the existence and the completeness of the wave operators. We recall that the wave operators for the pair $(\hat{\mathcal A}, \hat{\mathcal A}_0)$ are said to be complete if their range is $\hat{\mathcal H}_{\rm ac} ({\hat{\mathcal A}}).$ See Definition 1 in Section 3 of Chapter 2 of \cite{ya}. We also prove  Birman's invariance principle of the wave operators. We first prove that the weak Abelian wave operators exist and  are parcial isometries with initial subspace $\tilde{\mathcal H}_0$ and  final subspace $\hat{\mathcal H}_{\rm ac}({\hat{\mathcal A}}).$ The existence and the completeness of the wave operators follow from this result. Along the way, we also prove the stationary formulae for the wave operators. 
We begin defining the weak Abelian wave operators   as in equation 11 in page 76 of \cite{ya}
\beq\label{wa.2}
\Omega_\pm:= {\rm w-} \lim_{\varepsilon \downarrow 0} \int_0^\infty w_\varepsilon(t) P_{\rm ac}(\hat{\mathcal A})
e^{\pm it \hat{\mathcal A}} e^{\mp it\hat{\mathcal A}_0}  P_{\rm ac}(\hat{\mathcal A}_0)
 dt,
\ene
if the limit exits, where we take $w_\varepsilon(t)= 2 \varepsilon e^{-2\varepsilon t},$     and  ${\rm w-} \lim_{\varepsilon \downarrow 0}$ means the weak limit as $\varepsilon \downarrow 0.$ This is equivalent to,
\beq\label{wa.2b}
\left(\Omega_\pm f, g\right)_{\hat{\mathcal H}}:=  \lim_{\varepsilon \downarrow 0} \int_0^\infty w_\varepsilon(t) 
\left(e^{\pm it \hat{\mathcal A}} e^{\mp it\hat{\mathcal A}_0} f, g\right)_{\hat{\mathcal H}}\, dt,
\ene
for $f\in \hat{\mathcal H}_{\rm ac}(\hat{\mathcal A}_0),$ and $ g\in \hat{\mathcal H}_{\rm ac}(\hat{\mathcal A}).$

In the following theorem we prove that the weak Abelian wave operators exist, that they are given by the stationary formulae, and that they are  parcial isometries with initial subspace  $\hat{\mathcal H}_{\rm ac}(\hat{\mathcal A}_0)$ and  final subspace $\hat{\mathcal H}_{\rm ac}({\hat{\mathcal A}}).$  
\begin{theorem}\label{pwa.1}  Suppose that Assumption~\ref{assum0} is satisfied, that    $h'_\pm(v^2/2) v^2 $  is  a integrable function of $v\in \mathbb R,$ and that $h_\pm(\lambda)$ have a continuous second derivative that fulfills  \eqref{6.18qww}.
  Then, weak Abelian wave operators  \eqref{wa.2}  exist, and
\beq\label{wa.3a}
\Omega_\pm = \mathcal F_\pm^\ast {\mathcal F}_0.
\ene
Moreover the $\Omega_\pm$ are partially isometric with initial subspace $ \hat{\mathcal H}_{\rm ac}(\hat{\mathcal A}_0)$ and final subspace $\hat{\mathcal H}_{\rm ac}(\hat{\mathcal A}).$ Recall that $\mathcal F_0$ is defined in Proposition~\ref{propgf}.

\end{theorem}
\begin{proof}  Note that as $\mathcal F_0$ is  partially isometric with initial subspace  $\hat{\mathcal H}_{\rm ac}({\hat{\mathcal A}_0})$ and final subspace $\tilde{\mathcal H}_0$  and    $\mathcal F_\pm^\dagger$  are unitary from 
 $\tilde{\mathcal H}_0$ onto $\hat{\mathcal H}_{\rm ac}({\hat{\mathcal A}}),$ it is immediate that if  \eqref{wa.3a} holds    the $\Omega_\pm$ are partially isometric with initial subspace  $\hat{\mathcal H}_{\rm ac}({\hat{\mathcal A}_0})$ and  final subspace $\hat{\mathcal H}_{\rm ac}(\hat{ \mathcal A}).$ 
 We proceed to prove   \eqref{wa.3a}. We begin by proving \eqref{wa.2b}.
Since $ e^{\pm it \mathcal A} e^{\mp it\mathcal A_0}$ is unitary it is enough to prove that the limits in \eqref{wa.2b} exist in  dense sets. Then, we can replace $g$ in \eqref{wa.2b} by $ P_{\rm ac}(\hat{\mathcal A})g, $ with $g \in \mathcal H_{\gamma,\xi},$  with $1/4 > \xi > 2\gamma>0.$  Recall that  $\mathcal H_{\gamma,\xi}$ was defined in \eqref{6.18p} and \eqref{6.18q}. Furthermore, we can approximate  $ P_{\rm ac}(\hat{\mathcal A})g $ by
finite sums, $\sum_{n=1}^N    E_{ \hat{\mathcal A}}({\Delta}_n) g, $ where ${\Delta}_n$  are  bounded  closed  intervals  that satisfy \eqref{5.100}-\eqref{5.102} 
with $ \Delta$  replaced by ${\Delta}_n$ and $J_-, J_+$
replaced, respectively by $ J_{-,n}$ and $J_{+,n},$ and moreover, with  $\Delta_j\cap \Delta_n=\emptyset, j \neq n,$ and   ${\Delta}_n \cap \mathcal M=\emptyset.$  
 In the  same way, we can replace $ f$ in \eqref{wa.2b} by $\sum_{n=1}^N E_{ \hat{\mathcal A}_0}({\Delta}_n) f, $ with 
 $ f \in \mathcal H_{\gamma,\xi}.$ 
Then, in order that the weak Abelian wave operators exists it is enough to prove that
\beq\label{wa.2c}
\left(\Omega_\pm E_{\hat{\mathcal A}_0}({\Delta}_j) f, E_{\hat{\mathcal A}}({\Delta}_n) g \right)_{\hat{\mathcal H}}:=  \lim_{\varepsilon \downarrow 0} \int_0^\infty w_\varepsilon(t) 
\left(e^{\pm it \hat{\mathcal A}} e^{\mp it\hat{\mathcal A}_0}   E_{\hat{\mathcal A}_0}({\Delta}_j) f, E_{\hat{\mathcal A}}({\Delta}_n) g \right)_{\hat{\mathcal H}}\, dt.
\ene
Further, as in the proof of Lemma 1 in page 92 of \cite{ya} we prove, taking the Fourier  transform that \eqref{wa.2c} is equivalent to
\beq\label{wa.3}\begin{array}{r}
\left( \Omega_\pm  E_{\hat{\mathcal A}_0}({\Delta}_j) f,  E_{\hat{\mathcal A}}(
{\Delta}_n) g \right)_{\hat{\mathcal H}}= \lim_{\varepsilon \downarrow  0}  \ds \frac{\varepsilon}{\pi} \int_{\mathbb R}\,\left( R_{\hat{\mathcal A}_0}(\beta\pm i\varepsilon) E_{\hat{\mathcal A}_0}({\Delta}_j) f,\right. \\ \left. R_{\hat{\mathcal A}}(\beta\pm i\varepsilon)  E_{\hat{\mathcal A}}({\Delta}_n) g \right)_{\hat{\mathcal H}}\, d\beta.
\end{array}
\ene
Then, to prove \eqref{wa.2b} and \eqref{wa.3a} it is enough to prove that \eqref{wa.3} holds and that,
\beq\label{wa.3b}
\left( \Omega_\pm  E_{\hat{\mathcal A}_0}({\Delta}_j) f,  E_{\hat{\mathcal A}}(
{\Delta}_n) g \right)_{\hat{\mathcal H}}= \left( \mathcal F  E_{\hat{\mathcal A}_0}({\Delta}_j) f, \mathcal F_\pm  E_{\hat{\mathcal A}}(
{\Delta}_n) g\right)_{\tilde{\mathcal H}_0}.
\ene
 We proceed to prove \eqref{wa.3} and \eqref{wa.3b}. We first prove that when $j\neq n$ the limit in the right-hand side of \eqref{wa.3} is zero. Remark that as ${\Delta}_j \cap {\Delta}_n=\emptyset,$ by \eqref{5.93} and \eqref{s.75d} also the right-hand side of \eqref{wa.3b} is zero. Take $R$ such that ${\Delta}_j\cup{\Delta}_n \subset (-R/2, R/2).$ Then, 
  \beq\label{wa.3bxx}\begin{array}{c}
 \ds\frac{\varepsilon}{\pi} \int_{|\beta| \geq R  } \left| \left( R_{\hat{A}_0}(\beta\pm i\varepsilon) E_{\hat{\mathcal A}}({\Delta}_j) f,\right. \right.  \left. \left. R_{\hat{\mathcal A}}(\beta\pm i\varepsilon)  E_{\hat{\mathcal A}}({\Delta}_n) g \right)_{\hat{\mathcal H}}\right |\, d\beta \leq \\ 
C \int_{|\beta| \geq R}  \frac{1}{|\beta|^2} d\beta \ds \leq \frac{C}{R}.
\end{array}
\ene
Let $ \hat{\Delta}_n, n=1,\dots,N$ be closed bounded intervals that satisfy the same conditions as $\Delta_n$ 
and that contain $ \Delta_n$ on its interior. 
 By the spectral theorem (see Section five of Chapter six of \cite{kato} and Theorem VIII.6 in page 263 of \cite{rs1})
$$
\|R_{\hat{\mathcal A}}(\beta\pm i\varepsilon)  E_{\hat{\mathcal A}}({\Delta}_n) \| \leq C, \qquad \beta \in  \mathbb R \setminus \hat{\Delta}
_n.
$$
Then, by Schwarz's inequality and the spectral theorem,
\beq\label{wa.3ax}\begin{array}{c}
 \frac{\varepsilon}{\pi} \int_{[\mathbb R \setminus\hat{ \Delta}_n]\cap [-R,R]}  \left|\left( R_{\hat{A}_0}(\beta\pm i\varepsilon) E_{\hat{\mathcal A}_0}({\Delta}_j) f,\right. \right.  \left. \left. R_{\hat{\mathcal A}}(\beta\pm i\varepsilon)  E_{\hat{\mathcal A}}({\Delta}_n) g \right)_{\hat{\mathcal H}}\right |\, d\beta \leq \\
C   \int_{[\mathbb R \setminus \hat{\Delta}_n]\cap [-R,R]}  \left[ \int_{{\Delta}_j} \frac{\varepsilon^2}{(\beta-\lambda)^2+\varepsilon^2} \frac{d}{d \lambda}
(E_{\hat{\mathcal A}_0}(\lambda)f,f)_{\hat{\mathcal H}} \, d\lambda \right]^{1/2}\, d\beta . 
\end{array}
\ene
Hence, by Lebesgue dominated convergence theorem,
\beq\label{wa.3c}
\ds \lim_{\varepsilon \downarrow  0}   \frac{\varepsilon}{\pi} \int_{[\mathbb R \setminus \hat{\Delta}_n]\cap [-R,R]} \left|\left( R_{\hat{\mathcal A}_0}(\beta\pm i\varepsilon) E_{\hat{\mathcal A}_0}(\Delta_j)) f,\right. \right.\\ \left. \left. R_{\hat{\mathcal A}}(\beta\pm i\varepsilon)  E_{\hat{\mathcal A}}({\Delta}_n) g \right)_{\hat{\mathcal H}}\right |\, d\beta=0.
\ene
By \eqref{wa.3bxx} and \eqref{wa.3c},
\beq\label{wa.3cxx}
\ds \lim_{\varepsilon \downarrow  0}   \frac{\varepsilon}{\pi} \int_{\mathbb R \setminus \hat{\Delta}_n} \left |\left( R_{\hat{A}_0}(\beta\pm i\varepsilon) E_{\hat{\mathcal A}}({\Delta}_j) f,\right.\right. \\ \left. \left.R_{\hat{\mathcal A}}(\beta\pm i\varepsilon)  E_{\hat{\mathcal A}}({\Delta}_n) g \right)_{\hat{\mathcal H}}\right |\, d\beta=0.
\ene
We similarly prove,
\beq\label{wa.3ccee}\begin{array}{l}
\ds \lim_{\varepsilon \downarrow  0}   \frac{\varepsilon}{\pi} \int_{\mathbb R \setminus \hat{\Delta}_j} \left |\left( R_{\hat{A}_0}
(\beta\pm i\varepsilon) E_{\hat{\mathcal A}}({\Delta}_j) f,
  R_{\hat{\mathcal A}}(\beta\pm i\varepsilon)  E_{\hat{\mathcal A}}({\Delta}_n) g \right)_{\hat{\mathcal H}}\right |\, d\beta=0.
\end{array}
\ene
By \eqref{wa.3cxx} and \eqref{wa.3ccee} the limit in the right-hand side  of \eqref{wa.3} is zero for $j \neq n.$
We consider now the case $j=n.$ 
 
Since $\mathcal M$ has measure zero,
\beq\label{wa.3cd}\begin{array}{l}
   \frac{\varepsilon}{\pi} \int_{ \mathcal M} \left( R_{\hat{\mathcal A}_0}(\beta\pm i\varepsilon) E_{\hat{\mathcal A}_0}({\Delta}_j) f,\right.  \left. R_{\hat{\mathcal A}}(\beta\pm i\varepsilon)  E_{\hat{\mathcal A}}({\Delta}_j) g \right)_{\hat{\mathcal H}}\, d\beta=0.
\end{array}
\ene
Moreover, by  \eqref{5.103} and \eqref{s.66}
\beq\label{wa.3d}\begin{array}{l}
\ds \frac{\varepsilon}{\pi} \int_{ \mathbb R \setminus \mathcal M}\,d\beta\,\left( R_{\hat{\mathcal A}_0}(\beta\pm i\varepsilon) E_{\hat{\mathcal A}_0}({\Delta}_j) f, R_{\hat{\mathcal A}}(\beta\pm i\varepsilon)  E_{\hat{\mathcal A}}({\Delta}_j) g \right)_{\hat{\mathcal H}}= \int_{  \mathbb R \setminus \mathcal M} \, d\beta   \int_{{\Delta}_j} d\upsilon \\[10pt]
\ds \frac{1}{\pi} \frac{\varepsilon} {(\beta-\upsilon)^2+\varepsilon ^2}   \left(\mathcal L_{j_-,j_+}(\upsilon) \mathbf F E_{\hat{\mathcal A}_0}({\Delta}_j) f,   \mathcal L_{j_-,j_+}(\upsilon) \mathbf F  \left(I+M(\beta\pm i\varepsilon) \right)^{-1}    E_{\hat{\mathcal A}}({\Delta}_j) g\right )_{Q_{j_-,j_+}}.         
 \end{array}
 \ene
Recall that $Q_{j_-,j_+}$ is defined in \eqref{s.83ca} and \eqref{s.83cb}. Moreover, by  \eqref{wa.3d},  Theorem~\ref{invert}  and Theorems 7 and 8 in Subsection 3 of Section 2 of Chapter 1 of \cite{ya}
\beq\label{wa.3e}\begin{array}{l}
\ds\lim_{\varepsilon \downarrow  0}   \frac{\varepsilon}{\pi} \int_{ \mathbb R \setminus \mathcal M}\,\left( R_{\hat{\mathcal  A}_0}(\beta\pm i\varepsilon) E_{\hat{\mathcal A}_0}({\Delta}_j) f,\right.  \left. R_{\hat{\mathcal A}}(\beta\pm i\varepsilon)  E_{\hat{\mathcal A}}({\Delta}_j) g \right)_{\hat{\mathcal H}}\, d\beta=    \int_{ {\Delta}_j} \,d\beta \\
 \left(  \mathcal L_{j_-,j_+}(\beta) \mathbf F E_{\hat{\mathcal A}_0}({\Delta}_j) f , \right.  \left.\mathcal L_{j_-,j_+}(\beta) \mathbf F  \left(I+ M(\beta\pm i 0) \right)^{-1}    E_{\hat{\mathcal A}}({\Delta}_j) g \right)_{Q_{j_-,j_+}}.        
 \end{array}
 \ene
 Further, by \eqref{5.103}, \eqref{s.75c}-\eqref{s.75e}, and \eqref{wa.3e}
 \beq\label{wa.3f}\begin{array}{l}
\ds \lim_{\varepsilon \downarrow  0}   \frac{\varepsilon}{\pi} \int_{\mathbb R \setminus \mathcal M }\,\left( R_{\hat{\mathcal A}_0}( \beta\pm i\varepsilon) E_{\hat{\mathcal A}_0}({\Delta}_j) f,\right.  \left. R_{\hat{\mathcal A}}(\beta\pm i\varepsilon)  E_{\hat{\mathcal A}}({\Delta}_j) g \right)_{\hat{\mathcal H}}\, d\beta=\\
\left( \mathcal F_0  E_{\hat{\mathcal A}_0}({\Delta}_j)f , \mathcal F_\pm E_{\hat{\mathcal A}}({\Delta}_j  ) g\right)_{\tilde{\mathcal H}_{0}}.
\end{array}
\ene
By \eqref{wa.3cd},  \eqref{wa.3f}, and as the limit in the right-hand side of \eqref{wa.3} is zero for $ j\neq n,$  equations \eqref{wa.3} and \eqref{wa.3b} hold. This completes the proof of the theorem.
\end{proof}
In the following theorem we prove our results on the existence and the completeness of  the wave operators.
\begin{theorem}\label{waop}
 Suppose that Assumption~\ref{assum0} is satisfied, that    $h'_\pm(v^2/2) v^2 $  is  a integrable function of $v\in \mathbb R,$ and that $h_\pm(\lambda)$ have a continuous second derivative that fulfills  \eqref{6.18qww}.
  Then, the wave operators,
\beq\label{wa.6}
W_\pm= {\rm s-} \lim_{t\to \pm \infty} e^{it \hat{\mathcal A}} e^{-it\hat{\mathcal A}_0}  P_{\rm ac}(\hat{\mathcal A}_0),
\ene
where $ {\rm s-} \lim_{t\to \pm \infty} $ means the strong limit as $t\to \pm \infty,$  exist and are complete. They are partially isometric with initial subspace $ \hat{\mathcal H}_{\rm ac}(\hat{\mathcal A}_0)$ and final subspace $\hat{\mathcal H}_{\rm ac}(\hat{\mathcal A}).$ Moreover, the intertwining relations hold,
\beq\label{wa.7}
\phi(\hat{\mathcal A}_{\rm ac}) W_\pm= W_\pm \phi(\hat{\mathcal A}_0),
\ene
for every Borel function $\phi.$ Furthermore, the stationary formulae for the wave operators are valid,
\beq\label{wa.8}
W_\pm= \mathcal F^\ast_\pm \mathcal F.
\ene
\end{theorem} 
\begin{proof}
By Theorem~\ref{pwa.1} the $\Omega_\pm$  are partially isometric with initial subspace $ \hat{\mathcal H}_{\rm ac}(\hat{\mathcal A}_0)$ and final subspace $\hat{\mathcal H}_{\rm ac}(\hat{\mathcal A}).$
Then, by Corollary 2 in page 74 and the remarks in page 76 of \cite{ya}, the wave operators $W_\pm $   exist,  are complete, and coincide with the $\Omega_\pm.$ Furthermore, by \eqref{wa.3a} the stationary formulae \eqref{wa.8} hold. Finally, by Theorem 4 in page 69 of \cite{ya} the intertwining relations \eqref{wa.7} are valid for any bounded Borel function $\phi.$ In the general case where $\phi $ is not bounded,  \eqref{wa.7} follows from \eqref{wa.7} in the bounded case, approximating $\phi$ by 
$ \phi_n$ where $ \phi_n(\lambda)= \phi(\lambda)$ if $|\phi(\lambda)| \leq n$ and $ \phi_n(\lambda)=0,$ if 
$|\phi(\lambda)|  >n,$ for $n=1,\dots.$
\end{proof}

We proceed to prove Birman's invariance principle of the wave operators. For this purpose, we make explicit the dependence of the wave operators on $\hat{\mathcal A}$ and $\hat{\mathcal A}_0,$   and denote by $W_\pm(\hat{\mathcal A}, \hat{\mathcal A}_0)$ the wave operator for the pair $\hat{\mathcal A}, \hat{\mathcal A}_0.$ 
Let $\phi$ be a real valued Borel function defined on $\mathbb R=\sigma(\hat{\mathcal A}).$ We decompose $\mathbb R$ as follows
\beq  \label{wa.13}
\mathbb R= \cup_{n=1}^N \Lambda_n,
\ene
where the $\Lambda_n,$ for $n=1,\dots, N$ are  nonintersecting intervals, and $N$ is a positive integer, or $N=\infty.$  Suppose that $\phi$ is absolutely continuous in each of the $\Lambda_n, n=1,\dots,$ and that $\phi'(\beta) >0,$ for almost every $\beta$ in $\Lambda_n.$
Then, by Lemma 1 in page 87 of \cite{ya}
\beq\label{wa.14}
\hat{\mathcal H}_{\rm ac}(\phi(\hat{\mathcal A}))= \hat{\mathcal H}_{\rm ac}(\hat{\mathcal A}),
\ene
and
\beq\label{wa.15}
\hat{\mathcal H}_{\rm ac}(\phi(\hat{\mathcal A}_0))= \hat{\mathcal H}_{\rm ac}(\hat{\mathcal A}_0).
\ene
  Birman's invariance principle is the following theorem.  
  \begin{theorem} \label{bir}  Suppose that Assumption~\ref{assum0} is satisfied, that    $h'_\pm(v^2/2) v^2 $  is  a integrable function of $v\in \mathbb R,$ and that $h_\pm(\lambda)$ have a continuous second derivative that fulfills  \eqref{6.18qww}.
  Assume that $\phi$ is a real valued Borel function defined in $\mathbb R,$ that is absolutely continuous on each of the $\Lambda_n$ in \eqref{wa.13} and that $\phi'(\beta) >0,$ for almost every $\beta \in \Lambda_n,$ for $n=1,\dots, N.$ Then, the wave operators $W_\pm(\phi(\hat{\mathcal A}), \phi(\hat{\mathcal A}_0))$ exist and,
  \beq\label{wa.16}
  W_\pm(\phi(\hat{\mathcal A}), \phi(\hat{\mathcal A}_0))= W_\pm(\hat{\mathcal A}, \hat{\mathcal A}_0).
  \ene
  In particular, the wave operators $W_\pm(\phi(\hat{\mathcal A}), \phi(\hat{\mathcal A}_0))$ are complete. They are 
  partially isometric with initial subspace $ \hat{\mathcal H}_{\rm ac}(\hat{\mathcal A}_0)$ and final subspace $\hat{\mathcal H}_{\rm ac}(\hat{\mathcal A}).$
  \end{theorem}
  \begin{proof} The Abelian wave operators, $\Theta_\pm(\hat{\mathcal A}, \hat{\mathcal A}_0)$ are defined by the following limit in page 76 of \cite{ya},
  \beq\label{wa.17}
  \lim_{\varepsilon \downarrow 0} \int_0^\infty\varepsilon e^{-\varepsilon t} \| e^{i\pm t\hat{\mathcal A}}  e^{\mp it \hat{\mathcal A}_0 t} P_{\rm ac}(\hat{\mathcal A}_0)   f- \Theta_\pm(\hat{\mathcal A},\hat{\mathcal A}_0) f\|^2_{\hat{\mathcal H}} dt, \qquad f \in \hat{\mathcal H}.
  \ene
   By the remarks in page 76 of \cite{ya}, as by Theorem~\ref{waop} the wave operators $W_\pm(\hat{\mathcal A}, \hat{\mathcal A}_0)$ exist, the Abelian wave operators exist and,
  \beq\label{wa.18}
  \Theta_\pm(\hat{\mathcal A},\hat{\mathcal A}_0)= W_\pm(\hat{\mathcal A}, \hat{\mathcal A}_0).
  \ene
  Then, by Theorem 1 in page 110 of \cite{ya} the Abelian wave operators $\Theta_\pm(\phi(\hat{\mathcal A}),\phi(\hat{\mathcal A}_0))$ exist and, furthermore,
  \beq\label{wa.19}
  \Theta_\pm(\phi(\hat{\mathcal A}),\phi(\hat{\mathcal A}_0))=   \Theta_\pm(\hat{\mathcal A},\hat{\mathcal A}_0)= W_\pm(\hat{\mathcal A}, \hat{\mathcal A}_0).
  \ene
  Further,  by Theorem~\ref{waop} the wave operators $W_\pm(\hat{\mathcal A}, \hat{\mathcal A}_0)$ are   partially isometric with initial subspace $ \hat{\mathcal H}_{\rm ac}(\hat{\mathcal A}_0)$ and final subspace $\hat{\mathcal H}_{\rm ac}(\hat{\mathcal A}).$
 Then, it follows from \eqref{wa.19} that the wave operators $  \Theta_\pm(\phi(\hat{\mathcal A}),\phi(\hat{\mathcal A}_0))$ are   partially isometric with initial subspace $ \hat{\mathcal H}_{\rm ac}(\hat{\mathcal A}_0)$ and final subspace $\hat{\mathcal H}_{\rm ac}(\hat{\mathcal A}).$
 Hence,  by the remarks in page 76 of \cite{ya} the wave operators  $W_\pm(\phi(\hat{\mathcal A}), \phi(\hat{\mathcal A}_0))$ exist and
  \beq\label{wa.20}
   W_\pm(\phi(\hat{\mathcal A}), \phi(\hat{\mathcal A}_0))= \Theta_\pm(\phi(\hat{\mathcal A}),\phi(\hat{\mathcal A}_0)).
  \ene
  Equation \eqref{wa.16} follows from \eqref{wa.19} and \eqref{wa.20}. Finally, by \eqref{wa.16} the wave operators $  W_\pm(\phi(\mathcal A), \phi(\mathcal A_0))$ are   partially isometric with initial subspace $ \hat{\mathcal H}_{\rm ac}(\hat{\mathcal A}_0)$ and final subspace $\hat{\mathcal H}_{\rm ac}(\hat{\mathcal A})$
 since this is true for the wave operators $W_\pm(\hat{\mathcal A},\hat{\mathcal A}_0).$
  \end{proof}
The scattering operator is defined as follows,
\beq\label{wa.21}
S:= W_+^\ast W_-.
\ene 
 By Theorem~\ref{waop} $S$  is unitary on $\hat{\mathcal H}.$
 \section{Asymptotic stability, large time asymptotics and Landau damping}\label{ladam}
 In this section we obtain our results in asymptotic stability namely,  in the large time asymptotic behaviour of the 
 phase-space densities and on Landau damping for the electric field. Let us take,
   \beq\label{9.1}
  U(0):= \begin{pmatrix} u_{-,0}(x,v)\\ u_{+,0}(x,v)\\ F_0(x).
  \end{pmatrix} \in {\mathcal H}_{\rm ac}(\mathcal A). 
  \ene
 Then, the solution to the Vlasov-Amp\`ere system \eqref{0.32}-\eqref{0.35} is given by,
  \beq\label{9.2}
  U(t):= e^{-it \mathcal A} U(0)= \begin{pmatrix} u_{-}(t,x,v)\\ u_{+}(t,x,v)\\ F(t,x)
  \end{pmatrix} \in {\mathcal H}_{\rm ac}(\mathcal A). 
  \ene
As $ U(0) \in \mathcal H_{\rm ac}(\mathcal A)$ and   as ${\mathcal H}_{\rm ac}({\mathcal A}) \subset ({\rm Ker}[{\mathcal A}]^\perp,$  it follows from Proposition~\ref{gauss} (c) that $U(0)$ fulfills the Gauss law \eqref{6.10c}. In the same way,  we see  that $U(t)$ satisfies the Gauss law \eqref{6.10c}.
Recall that the unperturbed Vlasov--Amp\`ere system is given by,
\beq\label{9.3}
i \partial_t U(t)= {\mathcal A}_0 U(t).
\ene
We denote
\beq\label{9.4}
U^{(\pm)}(0):=  {\mathbf V}^\ast  W_\pm^\ast \mathbf V U(0) \in {\mathcal H}_{\rm ac}({\mathcal A}_0),
\ene
where $\mathbf V$ is the unitary operator from $\mathcal H$ onto $\hat{\mathcal H}$ defined  in \eqref{5.7}, \eqref{5.8}. We denote by $U^{(\pm)}(t)$ the solution to \eqref{9.3} with initial value at $t=0$ equal to   $U^{(\pm)}(0)$,
\beq\label{9.5}
  U^{(\pm)}(t)= e^{-it {\mathcal A}_0} U^{\pm}(0).
\ene
 Then, by Theorem~\ref{waop}  
\beq\label{9.6}
\lim_{t \to \pm \infty}\left\|U(t)-U^{(\pm)}(t)\right\|_{\mathcal H}=0.
\ene
In the following theorem, using equations \eqref{9.6}, we obtain a precise description of the large time  behaviour of the solutions to the Vlasov- Amp\`ere system \eqref{0.32}-\eqref{0.35}.   
\begin{theorem}\label{landam}
 Suppose that Assumption~\ref{assum0} is satisfied, that    $h'_\pm(v^2/2) v^2 $  is  a integrable function of $v\in \mathbb R,$ and that $h_\pm(\lambda)$ have a continuous second derivative that fulfills  \eqref{6.18qww}.   Let $ f_\mp(t,x,v)$ and $F(x,t)$  be, respectively,  the phase-space densities, and  the electric field  of the solution $U(t)$  to the Vlasov-Amp\`ere system  \eqref{0.32}-\eqref{0.35}  given in \eqref{9.2}, with $U(0) \in {\mathcal H}_{\rm ac}(\mathcal A).$ Moreover, let $ f^{(\pm)}_-(t,x,v),$   $ f^{(\pm)}_+(t,x,v),$ and  $F^{(\pm)}(t,x)$ be, respectively,  the phase-space densities, and the electric field,  of the solutions \eqref{9.5} to the unperturbed Vlasov-Amp\`ere system \eqref{9.3} with $ U^{\pm}(0) \in {\mathcal H}_{\rm ac}(\mathcal A_0)= {\rm Ker}[\mathcal A_0]^\perp.$ 
We have:
\begin{enumerate}
\item[\rm (a)] 
\beq\label{9.7}\begin{array}{l}
\lim_{t \to \infty}  \| f_-(t,\cdot,\cdot)-   f^{(+)}_-(t,\cdot,\cdot) \|_{ L^2((-P/2,P/2) \times \mathbb R)}=0,
\\[.3cm]
\lim_{t \to \infty}  \| f_+(t,\cdot,\cdot)-   f^{(+)}_+(t,\cdot,\cdot) \|_{ L^2((-P/2,P/2) \times \mathbb R)}=0,
\\[.3cm]
\lim_{t \to -\infty}  \| f_-(t,\cdot,\cdot)-   f^{(-)}_-(t,\cdot,\cdot) \|_{ L^2((-P/2,P/2) \times \mathbb R)}=0,
\\[.3cm]
\lim_{t \to -\infty}  \| f_+(t,\cdot,\cdot)-   f^{(-)}_+(t,\cdot,\cdot) \|_{ L^2((-P/2,P/2) \times \mathbb R)}=0.
\end{array}
\ene 
This gives a precise description of the large time behaviour of the phase-space densities. Namely,  for large times they are asymptotic to solutions of the unperturbed  Vlasov-Amp\`ere system \eqref{9.3}. This implies that they follow the trajectories of the solutions of Newton's equations \eqref{0.15b} for electrons and ions with the potential of the  BGK wave, in the sense that they are transported along these trajectories.  

\item[\rm (b)]   $F^{(\pm)}(t,x)\equiv 0,$ for $ t \in \mathbb R$ and  $x \in (-P/2, P/2).$  Furthermore, for any $ 0 < \alpha < 1/2,$
\beq\label{9.8}
\lim_{t \to \pm \infty} \| F(t,\cdot)\|_{\hat{C}_{\alpha}((-P/2,P/2))}=0.
\ene
This is the Landau damping of the electric field.
\end{enumerate} 
\end{theorem}
\begin{proof} Item (a) follows from  \eqref{9.6}. Let us prove (b).
 As  $U^{(\pm)}(t,x)\in  {\mathcal H}_{\rm ac} ({\mathcal A}_0) \subset {\rm Ker}[{\mathcal A}_0]^\perp$
 it follows from Theorem~\ref{thoesefad} (d) that $F^{(\pm)}(t,x,v)\equiv 0.$ Moreover, by \eqref{9.6}
\beq\label{9.9}
\lim_{t \to \pm \infty} \| F(t,\cdot)\|_{L^2_0((-P/2,P/2))}= \lim_{t \to \pm \infty} \| F(t,\cdot,)-F^{(\pm)}(t,\cdot)\|_{L^2_0((-P/2,P/2))}=0.
\ene
As $ U(t)$ fulfills the Gauss law \eqref{6.10c},
\beq\label{9.10}
\| F\|_{H_1((-P/2,P/2)))}\leq C \|U(t)\|_{\mathcal H}= C \|U(0)\|_{\mathcal H}.
\ene
Let $\varphi \in C^1(\mathbb R)$ satisfy $ \varphi(x)=0,$ for $ x \leq  -P/2 +\varepsilon,$ and 
$ \varphi(x)=1$ for $ x \geq  -P/2 +2\varepsilon,$ for some $\varepsilon$ with $ 0 < \varepsilon < P.$
We decompose $F(t,x)$ as follows,
\beq\label{9.11}
F(t,x)= F_1(t,x)+ F_2(t,x), F_1(t,x):= \varphi(x) F(t,x), F_2(t,x):= (1-\varphi(x)) F(t,x).
\ene
Hence,
\beq\label{9.12}\begin{array}{l}
|F_1(t,x)|^2=2 \left | \int_{-P/2}^x \,  {\rm Real }(\partial_y F_1(t,x)) \overline{F_1(t,y)}) dy \right|  \leq
\\[.3cm] 2  \|F(t,\cdot)\|_{L^2_0((-P/2,P/2))} \, \|F(t,\cdot)\|_{H_1((-P/2,P/2))}.
\end{array}
\ene
Estimating $F_2(t,x)$ in a similar way and using \eqref{9.10} we obtain,
\beq\label{9.13}
|F(t,x)| \leq  C \sqrt{   \|F(t,\cdot)\|_{L^2_0((-P/2,P/2))}}.
\ene
Moreover,
\beq\label{9.14}\begin{array}{l}
|F(t,x_1)- F(t,x_2)|= \left| \int_{x_1}^{x_2}  \partial_y F(t,y)  dy \right| \leq \sqrt{|x_1-x_2|} \sqrt{ \|F(t,\cdot)\|_{H_1((-P/2,P/2))}} \leq\\[.3cm]
C \sqrt{|x_1-x_2|},
\end{array}
\ene\
where we used  \eqref{9.10}. By \eqref{9.13} and \eqref{9.14}, for $ 0< \upsilon < 1,$
\beq\label{9.15}
|F(t,x_1)- F(t,x_2)| \leq C |x_1-x_2|^{\upsilon/2}\,  \| F(t,\cdot)\|_{L^2_0((-P/2,P/2))}^{(1-\upsilon)/2}.
\ene
Equation \eqref{9.8} follows from \eqref{9.9}, \eqref{9.13}, \eqref{9.15}, and recalling that $ \hat{C}_{\alpha}((-P/2,P/2)) \subset {C}_{\upsilon}((-P/2,P/2))$ for $0 < \alpha <  \upsilon \leq 1,$ as we mentioned in the introduction.
\end{proof}

\appendix 

\section{Appendix}\label{apex}

\renewcommand{\theequation}{\thesection.\arabic{equation}}

\newtheorem{theorem2}{THEOREM}[section]
\renewcommand{\thetheorem}{\arabic{section}.\arabic{theorem}}

\newtheorem{prop2}[theorem2]{Proposition}
\newtheorem{lemma2}[theorem2]{Lemma}
 In the   following proposition we state relevant properties of the period function   for the electrons and for the ions. 
 For similar results in the one-species case with a fixed ion background see \cite{bruno2}.
  \begin{prop2}\label{proper} 
Suppose that item (b) of  Assumption~\ref{assum0} holds. Then, we have.
 \begin{enumerate}
 \item[\rm (a)] 
 \beq\label{ap.3}T_{-,0}(-\varphi_0(P/2)-\varepsilon)= -\frac{2}{\sqrt{\varphi_0''(P/2)}} \ln\varepsilon +C+o(1), \qquad  \varepsilon \downarrow 0.
 \ene
 \item[\rm (b)] 
    \beq\label{ap.4}T_{-,1}(-\varphi_0(P/2)+\varepsilon)= -\frac{1}{\sqrt{\varphi_0''(P/2)}} \ln\varepsilon +C+O(\varepsilon), \qquad  \varepsilon \downarrow 0.
    \ene
 \item[\rm (c)] \beq\label{ap.5}
 \lim_{\varepsilon \downarrow 0} T_{-,0}(-\varphi_0( 0)+\varepsilon)= \frac{2\pi }{\sqrt{-\varphi_0''(0)}}.
 \ene
  \item[\rm(d)]  $T_{-, 0}(\mathcal E_-)$ is twice continuously differentiable in $I_{-,0}$ and  if furthermore, \eqref{1.22b} in  item (c) of Assumption~\ref{assum0} holds, it is  strictly increasing  in  $I_{-,0}.$

 \item[(e)] $T_{-,1} (\mathcal E_-)$ is twice continuously differentiable and strictly decreasing in
 $I_{-,1}.$ Further,
 \beq\label{ap.5b}
 T_{-,1} (\mathcal E_-)=O\left(\frac{1}{\sqrt{\mathcal E_-}}\right), \qquad \mathcal E_- \to \infty.
 \ene
     \item[\rm (f)] 
 \beq\label{ap.6}T_{+,0}(\varphi_0(0)-\varepsilon)= -\frac{2}{\sqrt{-\varphi_0''(0)}} \ln\varepsilon +C+o(1), \qquad  \varepsilon \downarrow 0.
 \ene
  \item[\rm (g)] 
    \beq\label{ap.7}T_{+,1}(\varphi_0(0)+\varepsilon)= -\frac{1}{\sqrt{-\varphi_0''(0)}} \ln\varepsilon +C+O(\varepsilon), \qquad  \varepsilon \downarrow 0.
    \ene
 \item[\rm (h)] \beq\label{ap.8}
 \lim_{\varepsilon \downarrow 0} T_{+,0}(\varphi_0( P/2)+\varepsilon)= \frac{2\pi }{\sqrt{\varphi_0''(P/2)}}.
 \ene
  \item[\rm (i)]  $T_{+, 0}(\mathcal E_+)$ is twice continuously differentiable  in  $I_{+,0}$
  and  if furthermore, \eqref{1.22bb} in  item (c) of Assumption~\ref{assum0} holds, it is  strictly increasing  in  $I_{+,0}.$ 
  
  \item[\rm (j)] $T_{+,1} (\mathcal E_+)$ is twice continuously differentiable and strictly decreasing in
 $I_{+,1}.$ Further,
 \beq\label{ap.8b}
 T_{+,1} (\mathcal E_+)=O\left(\frac{1}{\sqrt{\mathcal E_+}}\right), \qquad \mathcal E_+ \to \infty.
 \ene
 
     \end{enumerate}
 \end{prop2}
 \begin{proof} Items (a), (b), and (c) are proved, respectively, in Lemmata 6.4, 6.2, and 6.5 of \cite{bruno2}. To prove (f) and (h) we define $\tilde{\varphi}_0(x)= \varphi_0(x), -P \leq x \leq 0.$ We denote $\tilde{\mathcal E}_+:= \frac{1}{2}v^2+ \tilde{\varphi}(x).$ We designate by $\tilde{T}(\tilde{\mathcal E}_+)$ the period function of $\tilde{\varphi}_0(x).$ Remark that $\tilde{T}(\tilde{\mathcal E}_+)=  T_{+,0}(\mathcal E_+ ),$ if  $\tilde{\mathcal E}_+= \mathcal E_+.$ Then, (f) and (h) follow, respectively, as in the proofs of Lemmata 6.4 and 6.5 of \cite{bruno2} with minor changes. Item (g) is proven as in the proof of Lemma 6.2 of \cite{bruno2} with minor changes. Items (e)
 and (j) are immediate from the definitions of  $T_{-,1} (\mathcal E_-)$ and of $T_{+,1} (\mathcal E_+)$ respectively in \eqref{1.81} and \eqref{1.93}. Let us prove (d).  In Theorem2.1 of \cite{cw} it is proved that $T_{-,0}(\mathcal E_-)$ is continuously differentiable on $I_{-,0}$ and the derivative is computed as follows
 \beq\label{ap.9}
 \frac{d}{d \mathcal E_-} T_{-,0}(\mathcal E_-)= \frac{2}{\mathcal E_-+\varphi_0(0)} \int_0^{x_+(\mathcal E_-)}\,\frac{(\varphi'_0(x))^2+2 (\varphi_0(0)-\varphi_0(x)) \varphi_0''(x)  }{ \sqrt{2(\mathcal E_-+\varphi_0(x))} \, (\varphi'_0(x))^2}\, dx.
 \ene
 Then, by \eqref{1.22b} in item (c) in Assumption~\ref{assum0} $\frac{d}{d \mathcal E_-} T_{-,0}(\mathcal E_-)>0.$
 Further, we prove that $T_{-,0}(\mathcal E_-)$  has a continuous second derivative for $\mathcal E_- \in I_{-,0}$ as in the proof of item (b) in Proposition A.2 of \cite{wegal}. 
 This proves (d). To prove (i) we proceed as in the proof of 
 (f) and (h).  We prove that  $\tilde{T}(\tilde{\mathcal E}_+)$ is  twice continuously differentiable and strictly increasing for $ \mathcal E_+ \in I_{+,0}$ as in the proof of (d) using \eqref{1.22bb}. Further, as  
 $\tilde{T}(\tilde{\mathcal E}_+)=  T_{+,0}(\mathcal E_+ ),$ if  $\tilde{\mathcal E}_+= \mathcal E_+,$ item (i) follows.

 \end{proof}
 In the following proposition we obtain properties of the angle variables for the electrons and for the ions. Recall
 that in Definition~\ref{defang} we have extended the domain of  the variable $x$   for the angles $\theta_{\mp,0}(x,\mathcal E_\mp).$ 
 \begin{prop2}\label{propang}  Suppose that item (b) of  Assumption~\ref{assum0} holds
 Then,  the following is true.
  \begin{enumerate}
  \item[(a)] For every $0 <\delta < \frac{1}{2}( \varphi_0(0)-\varphi_0(P/2)  )$ there is a constant $C$ such that,
   \beq\label{ap.12}\begin{array}{l} 
\left | \theta_{-,0}(x, \lambda_1)- \theta_{-,0}(x, \lambda_2)\right | \leq C \sqrt{|\lambda_1-\lambda_2|}, \\\lambda_1, \lambda_2 \in  [-\varphi_0(0)+\delta, - \varphi_0(P/2)-\delta],  x \in [-P/2, P/2].
\end{array}
\ene
  \item[(b)]
  For every $0 <\delta< \frac{1}{2} (\varphi_0(0)-\varphi_0(P//2)) $ there is a constant $C$ such that,
   \beq\label{ap.13}\begin{array}{l}
\left | \theta_{+,0,\mp}(x, \lambda_1)- \theta_{+,0,\mp}(x, \lambda_2)\right | \leq C \sqrt{|\lambda_1-\lambda_2|}, \\ \lambda_1,  \lambda_2 \in [  \varphi_0(P/2)+\delta, \varphi_0(0)-\delta ], x \in [-P/2, P/2].
\end{array}
\ene
  \item[(c)] For every $ \delta >0$  there is a constant $C$ such that,
   \beq\label{ap.13b}\begin{array}{l}
\left | \theta_{-,1}(x, \lambda_1)- \theta_{-,1}(x, \lambda_2)\right | \leq C {|\lambda_1-\lambda_2|}, \\\lambda_1, \lambda_2 \in [-\varphi_0(P/2)+\delta, \infty), x \in[-P/2,P/2].   
 \end{array}
 \ene
   \item[(d)] For every $ \delta >0$  there is a constant $C$ such that,
   \beq\label{ap.13c}\begin{array}{l}
\left | \theta_{+,1}(x, \lambda_1)- \theta_{+,1}(x, \lambda_2)\right | \leq C {|\lambda_1-\lambda_2|},\\ \lambda_1, \lambda_2 \in [\varphi_0(0)+\delta, \infty), x \in [-P/2, P/2].
 \end{array}
 \ene
 \end{enumerate}
\end{prop2}  
\begin{proof} Item (a) is proved in  (c) of  Proposition A.2 of \cite{wegal}. To prove (b) we proceed as in the proof of
 (f), (h), and (i) of  Proposition~\ref{proper}. We give the proof for $\theta_{+,0,-}.$ The case of $\theta_{+,0,+}$ follows in the same way. We define $\tilde{\varphi}_0(x)= \varphi_0(x), -P \leq x \leq 0.$ We denote $\tilde{\mathcal E}_+:= \frac{1}{2}v^2+ \tilde{\varphi}(x).$ We designate by $\tilde{x}_\mp(\tilde{\mathcal E}_+)$ the solutions to $\tilde{\varphi_0}(\tilde{x}_\mp(\tilde{\mathcal E}_+))= \tilde{\mathcal E}_+$ with 
$\tilde{x}_-(\tilde{\mathcal E}_+) < \tilde{x}_+(\tilde{\mathcal E_+}).$ We designate by $\tilde{T}(\tilde{\mathcal E}_+)$ the period function of $\tilde{\varphi}_0(x),$
\beq\label{ap.14}
\tilde{T}(\tilde{\mathcal E}_+):= 2 \int_{\tilde{x}_-(\tilde{\mathcal E}_+)}^{\tilde{x}_+(\tilde{\mathcal E}_+)}\,
\frac{1}{\sqrt{\tilde{\mathcal E}_+- \tilde{\varphi}_0(x)}} \, dx, \qquad \tilde{\mathcal E}_+ \in (\tilde{\varphi}_0(-P/2), \tilde{\varphi}_0(-P)).
\ene
 Remark that $\tilde{T}(\tilde{\mathcal E}_+)=  T_{+,0}(\mathcal E_+ ),$ if  $\tilde{\mathcal E}_+= \mathcal E_+.$ 
Further, we denote by $\tilde{\theta}(x, \tilde{\mathcal E}_+)$ the angle for  $\tilde{\varphi}_0(x),$
\beq\label{ap.15}
\tilde{\theta}(x, \tilde{\mathcal E}_+):= \frac{1}{\tilde{T}(\tilde{\mathcal E}_+)} \int_{\tilde{x}_-(\tilde{\mathcal E}_+)}^x\, \frac{1}{\sqrt{\tilde{\mathcal E}_+- \tilde{\varphi}_0(x)}} \, dx, \qquad \tilde{x}_-(\tilde{\mathcal E}_+) \leq x \leq \tilde{x}_+(\tilde{\mathcal E_+}_+).
\ene
Note that, if $\mathcal E_+= \tilde{\mathcal E}_+,$
\beq\label{ap.16}
\theta_{+,0,-}(x, \mathcal E_+)= \tilde{\theta}(x, \tilde{\mathcal E}_+)- \tilde{\theta}(P/2, \tilde{\mathcal E}_+), \qquad  x \in [-P/2, x_-(\mathcal E_+)],
\ene 
where we used that   $  x_-(\mathcal E_+)=  \tilde{x}_+(\tilde{\mathcal E}_+).$
Then, (b) for $\theta_{+,0,-}$ follows from  item  (c)  of  Proposition A.2 of \cite{wegal}. Items (c) and (d) are immediate from the definition of $\theta_{\mp,1},$ respectively in \eqref{1.82}, and \eqref{1.94}.
\end{proof}
 
 \noindent {\bf Acknowledgement}

\noindent Ricardo Weder is an Emeritus National Researcher of SNII-SECIHTI, M\'exico. He thanks SNII-SECIHTI
for the support to this research.
 
\end{document}